\documentclass[9pt]{article}
\usepackage[utf8]{inputenc} 

\usepackage{amsmath}    
\usepackage{amssymb}
\usepackage{bbm}
\usepackage{graphicx}   
\usepackage{verbatim}   
\usepackage{color}      
\usepackage{hyperref}   
\usepackage{enumerate}
\usepackage{listings}
\usepackage{xcolor}
\usepackage{slashed}
\usepackage{mathrsfs}
\usepackage{amsthm}
\usepackage{tikz}
\usepackage{cancel}
\usepackage{bbold}
\usepackage{pgfplots}
\usepackage{csvsimple}
\usepackage{adjustbox}
\usepackage{grffile}
\usepackage{todonotes}
\usepackage{longfbox}

\usepackage{ulem}
\usepackage{booktabs}

\usepackage{lscape}

\usepackage[final]{showlabels}

\newcommand{\del}{\partial}
\newcommand{\iden}{\mathbbm{1}}
\renewcommand{\theta}{\vartheta}
\renewcommand{\phi}{\varphi}
\newcommand{\vecc}[2]{\left ( \begin{array}{c}#1\\#2\\ \end{array}\right )}
\newcommand{\veccc}[3]{\begin{pmatrix}#1\\#2\\#3 \end{pmatrix}}
\newcommand{\dd}{\mathrm{d}}

\renewcommand{\and}{\wedge}

\renewcommand{\div}{\text{div\,}}
\renewcommand{\vec}{\mathbf}

\newcommand{\id}{\mathbb{1}}

\newcommand{\dt}{\Delta t}
\newcommand{\dx}{\Delta x}
\newcommand{\dy}{\Delta y}

\definecolor{mygray}{rgb}{0.5,0.5,0.5}
\definecolor{mymauve}{rgb}{0.58,0,0.82}

\definecolor{mygreen}{rgb}{0,0.6,0}

\newtheorem{theorem}{Theorem}

\newtheorem{proposition}[theorem]{Proposition}

\newtheorem{remark}[theorem]{Remark}

\newtheorem{example}[theorem]{Example}

\title{  Stationarity preserving   nodal Finite Element  methods for  multi-dimensional linear hyperbolic balance laws\\
via a  Global Flux quadrature formulation}

\author{Wasilij Barsukow$^a$, Mario Ricchiuto$^b$, Davide Torlo$^c$\\[20pt]
$^a$Institut de Math\'ematiques de Bordeaux (IMB),  CNRS UMR 5251, \\ 351 Cours de la Lib\'eration, 33405 Talence, France - {\tt wasilij.barsukow@math.u-bordeaux.fr}\\[5pt]
$^b$INRIA, Univ. Bordeaux, CNRS, Bordeaux INP, IMB, UMR 5251, \\ 200 Avenue de la Vieille Tour, 33405 Talence cedex, France - {\tt mario.ricchiuto@inria.fr}\\[5pt]
$^c$Dipartimento di Matematica G. Castelnuovo, \\ Universit\`a di Roma La Sapienza, Roma, Italy - {\tt davide.torlo@uniroma1.it}
}

\begin{document}

\maketitle
\begin{abstract} 
We consider linear, hyperbolic systems of balance laws in several space dimensions. They possess non-trivial steady states, which result from the  equilibrium between derivatives 
of the  unknowns in different directions, and the sources. Standard numerical methods fail to account  for this equilibrium, and 
include stabilization that  destroys it.  This manifests itself in a diffusion of states  that are supposed to remain stationary. 
We derive new stabilized high-order Finite Element methods based on a Global Flux quadrature: we reformulate the entire spatial operator as a mixed derivative of a single quantity, referred to as global flux. All spatial derivatives and the sources are thus treated simultaneously, and our methods are
stationarity preserving. 
Additionally, when this formulation is combined with interpolation on Gauss-Lobatto nodes,  
 the new methods are super-convergent at steady state. Formal consistency estimates, and strategies to construct well-prepared initial data are 
provided. The  numerical results confirm the theoretical predictions, and show the tremendous benefits of the new formulation.
\end{abstract}



\section{Introduction}\label{sec:intro}

This paper focuses on the discretization of multi-dimensional hyperbolic balance laws.   
Our final aim is  to be able to treat non-linear PDEs such as the Euler  equations with gravity, friction or chemical reactions,
or the shallow water equations with bathymetric, Coriolis, and wind stress source terms.
Having in mind this objective,  in this work we consider the  simpler setting of  linear hyperbolic balance laws,
and in particular the linear wave equation in first-order form with general source terms: 
\begin{align}
 \begin{cases}
	\del_t \vec v + \nabla p = \vec S_{\vec v}, \qquad \vec v \colon \Omega \subset \mathbb R^d \to \mathbb R^d,\\
	\del_t p + \nabla \cdot \vec  v = S_p, \qquad p \colon \Omega \subset \mathbb R^d \to \mathbb R.
	\end{cases} \label{eq:acoustic-source-alld}
\end{align}
In $d=2$ space dimensions, we write
\begin{align}\label{eq:acoustic-source}
	\begin{cases}
 \del_t u + \del_x p = S_u, &  \\
 \del_t v + \del_y p = S_v, \\
 \del_t p + \del_x u + \del_y v = S_p, 
	\end{cases}
\end{align}
for $u,v,p \colon \Omega \subset \mathbb R^2 \to \mathbb R$, and
with 
$\vec {v} = (u,v)$ and $\vec S_{\vec v}=(S_u,S_v)$. 
These equations can be used as toy-models for more complex systems such as the shallow water equations, Euler equations, and  magneto-hydrodynamics.

We also introduce the compact form
\begin{align}
 \del_t \mathsf{Q} + \nabla\cdot\mathsf{F} &= \mathsf{S} \;,\quad
   \mathsf{F} = (J^x\mathsf{Q} \;, J^y \mathsf{Q})
\end{align}
 of the system \eqref{eq:acoustic-source-alld}
 with, in two space dimensions, 
\begin{align}\label{eq:acoustics_matrix}
  \mathsf{Q} &= \veccc{u}{v}{p}, & J^x &= \left( \begin{array}{ccc} 0 & 0 & 1 \\ 0 & 0 &  0 \\ 1 & 0 & 0 \end{array} \right ), & J^y &= \left( \begin{array}{ccc} 0 & 0 & 0 \\ 0 & 0 &  1 \\ 0 & 1 & 0 \end{array} \right ), &\mathsf{S}  & =  \veccc{S_u}{S_v}{S_p}.
\end{align}
System \eqref{eq:acoustic-source-alld} admits non-trivial  stationary states  governed  by the conditions
\begin{align}
  \nabla \cdot \vec v &= S_p,  \label{eq:steady-constraint-div} \\
    \nabla p&=   \vec S_{\vec v}.     \label{eq:steady-constraint-grad}
\end{align}

\subsection{Well balanced and stationarity preserving methods in multiple dimensions}

An important challenge in numerics for hyperbolic PDEs is to find a stabilization for numerical methods such that the solution is not entirely dissipated in long-time simulations.
Typical applications of such PDEs require studying the long-time evolution of perturbations of  stationary states, without any spurious numerical waves. 
To address this challenge, we aim at numerical methods capable of providing enhanced approximations of genuinely multi-dimensional stationary states.

This is already a complicated task when considering the homogeneous version of \eqref{eq:acoustic-source}. 
As we will explain shortly,  with classical  stabilization strategies
data needs to fulfill $\del_x u = 0$ and $\del_y v = 0$ to remain stationary, i.e. classical methods keep stationary only an insignificant subset of all divergence--free vector fields $\del_x u + \del_y v = 0$.
Numerical methods whose stationary states are described by a discretization of $\nabla \cdot \vec v = 0$ without further constraints possess a rich set of stationary states and are called \emph{stationarity preserving}  \cite{barsukow17a}. One can show that the low Mach number limit of the Euler equations is related to the long-time limit of linear acoustics \cite{jung2022steady,jung2024behavior}, and that stationarity preserving methods are also involution preserving \cite{barsukow17a}.

Classical stabilization operators  are inspired by the smoothing associated to parabolic regularization. To first order, for the homogeneous vector equation in \eqref{eq:acoustic-source-alld}, in the simplest setting
the stabilization is of the form 
\begin{align}
\del_t u + \del_x p = \dfrac{h}{2}  \del_x^2 u, \\
\del_t v + \del_y p = \dfrac{h}{2}  \del_y^2 v ,
\end{align}
having denoted by $h$ the mesh size. One can immediately verify that,
unless $\partial_x u =\partial_y v=0$, divergence--free vector fields are not preserved as $\nabla \cdot \vec v = 0$ is not a stationary solution of the stabilized problem. The difficulties essentially come about because  one-dimensional   stabilization is used in a multi-dimensional context.
A way to re-establish consistency is to account for the multidimensional coupling of the velocity components,
and   replace the above stabilization by a gradient of the divergence. This leads to a vector Laplacian stabilization of the form \cite{morton01,sidilkover02,jeltsch06,mishra09preprint,lung14,barsukow17a,brt25}
$$
\del_t \vec v + \nabla p = \dfrac{h}{2}     \nabla(\nabla\cdot \vec v) = \dfrac{h}{2}   \vecc{\del_x^2 u + \del_x\del_y v}{\del_x \del_y u + \del_y^2 v}.
$$
Observe that this way, mixed derivatives $\del_{xy}$ appear, i.e., this is a truly multi-dimensional stabilization,  and it reduces to the previous stabilization in purely one-dimensional situations.
While modified equation analysis gives a hint of what needs to be done in the fully discrete setting, it is not sufficient. Instead, as shown in~\cite{brt25}, one needs to ensure that the same discrete divergence is in the kernel of both
the unstabilized discrete divergence operator, and of the discrete stabilization. In \cite{barsukow17a} a general way of achieving this in the context of Finite Difference methods was presented, and in \cite{brt25} this was achieved for the streamline upwind Petrov-Galerkin (SUPG) method.


For balance laws, the stabilization needs to take into account the source terms, since it is no longer the solenoidal constraint that needs to be satisfied, but
\eqref{eq:steady-constraint-div}. The corresponding generalization of conformal approaches 
focused on the solenoidal condition is far from trivial. This is why most effective
multi-dimensional stationarity preserving   schemes  for systems of balance laws to this date rely on correction techniques
that require the precise knowledge of the steady solution \cite{GCD18,BERBERICH2021104858,math10010015,birke2024,DUMBSER2024112875}.

\subsection{Multidimensional Global Flux quadrature}\label{sec:GF_multiD_intro}
In this work, we extend \cite{brt25}, which is dedicated to preserving divergence-free equilibria, to linear problems with source terms.
Our approach exploits the Global Flux (GF) idea~\cite{gascon2001construction,bouchut2004frontal,cheng2019new,chertock2022well,ciallella2022global}, 
originally tailored for one-dimensional balance laws, and developed for 
Finite Volume~\cite{gascon2001construction,cheng2019new,chertock2022well,ciallella2022global}, Discontinuous Galerkin~\cite{MANTRI2023112673}, 
Finite Differences~\cite{KAZOLEA2025106646}, and stabilized Finite Element methods \cite{mra24}. Relations with other approaches, such as e.g. Residual Distribution
schemes, are discussed in~\cite{Abgrall2022}.
 
%

For one-dimensional hyperbolic balance laws of the form 
\begin{equation}
	\partial_t \mathsf{Q} +\partial_x \mathsf{F}(\mathsf{Q}) = S(x,\mathsf{Q}) \label{eq:balancelawgeneral}
\end{equation}
the GF approach amounts to discretizing, instead of \eqref{eq:balancelawgeneral}, the equivalent problem 
\begin{equation}
	\partial_t \mathsf{Q} +\partial_x \mathsf G(x,\mathsf{Q})=0, \quad \text{with } \mathsf G(x,\mathsf{Q}):= \mathsf{F}(\mathsf{Q})- \underbrace{\int_{x_0}^x S(s, \mathsf{Q}(s,t)) \dd s}_{=: \mathsf K(x, \mathsf{Q})} .
\end{equation}
At the discrete level, it means that there is only one differential operator discretizing $\partial_x$, 
while discrete equilibria are fully determined by the approximation of the source integral $\mathsf K$.
As shown in \cite{Abgrall2022}, the simplest of such methods is the usual source term upwinding which dates back to the early
ideas of Roe \cite{Roe87},   and early works on  the so-called C-property \cite{bv94}.
In certain cases, they can be explicitly and fully described,  independently on the   discretization of  $\partial_x$, using ODE integration methods
 \cite{MANTRI2023112673,KAZOLEA2025106646}.

The above framework is purely one dimensional. The first genuinely multi-dimensional and arbitrarily high-order generalization is   discussed in   \cite{brt25} in the context of nodal, 
continuous, stabilized  Finite Element methods.  A  (low order)  Finite Volume version of the same  is proposed  in   \cite{bcrt25}. 
The present work  incorporates source terms in the multi-dimensional setting and aims at designing arbitrarily high-order Finite Element formulations that preserve the corresponding multi-dimensional discrete equilibria. To this end, the main idea is to rewrite the last equation of \eqref{eq:acoustic-source} as
 $$
 \partial_t p +\partial_{xy}\left( \int_{y_0}^y u(x,s) \dd s  \right)  
 +\partial_{xy}\left( \int_{x_0}^x v(s,y) \dd s \right) - 
\partial_{xy}\left(  \int_{x_0}^x  \int_{y_0}^y  S_p(s_1,s_2) \dd s_2\dd s_1  \right )  =0.
 $$
Setting 
\begin{equation}\label{eq:definition_intu_intv}
	\begin{cases}
		\mathcal U(x,y):= \int_{y_0}^y u(x,s) \dd s,\\[5pt]
		\mathcal V(x,y):= \int_{x_0}^x v(s,y) \dd s,\\[5pt]
		\mathcal K_p(x,y):= \int_{x_0}^x  \int_{y_0}^y  S_p(s_1,s_2) \dd s_2\dd s_1 
	\end{cases} 
\end{equation}
and
\begin{equation}\label{eq:definition_Gp}
	G_p := \mathcal U+\mathcal V -\mathcal K_p,
\end{equation}
we obtain
\begin{align}
 \partial_t p + \partial_{xy}G_p =0. \label{eq:gf2d}
\end{align}
\eqref{eq:definition_Gp} is a genuinely multi-dimensional GF accounting also for source terms.
As we will show below, although  equivalent to  \eqref{eq:acoustic-source}, 
formulation \eqref{eq:gf2d} is  significantly more effective for obtaining stationarity preserving methods.

The pressure gradient balance \eqref{eq:steady-constraint-grad} is treated here using one-dimensional  global integration independently on the $x$ and $y$ directions, i.e.,
\begin{equation}\label{eq:definition_Gv}
 \boldsymbol{\mathcal K}_{\vec v}=(\mathcal K_u,\mathcal K_v)=
 \left(
  \int_{x_0}^x S_u(s,y) \dd s,\,
  \int_{y_0}^y S_v(x,s) \dd s
 \right).
 \end{equation}
Altogether, one obtains the following multi-dimensional GF model 
\begin{equation}\label{eq:md-GF-continuous}
	\begin{cases}
\partial_t u  + \del_x (p - \mathcal K_u) =0,\\
\partial_t v  + \del_y (p - \mathcal K_v) =0,\\
\partial_tp +  \partial_{xy} G_p =0,
	\end{cases}
\end{equation}
which is equivalent to \eqref{eq:acoustic-source}, and which will serve as the basis of discretization as shown below.


Stationary states of \eqref{eq:md-GF-continuous} are characterized by
\begin{equation}\label{eq:div_free_GFsteady}
\begin{aligned}
		G_p= & \sigma^x_p(x) + \sigma^y_p(y),\\
			p-\mathcal K_u = & \sigma^y_u(y),\\
			p-\mathcal K_v=&   \sigma^x_v(x),
\end{aligned}
\end{equation}
for arbitrary functions $ \sigma^x_p,  \sigma^y_p,  \sigma^y_u,  \sigma^x_v\colon\mathbb R\to \mathbb R$ that depend on the boundary conditions. 
 The truly multi-dimensional condition $\nabla \cdot \vec v = S_p$ is now replaced by imposing that $G_p=\mathcal U+ \mathcal V- \mathcal K_p$ is a sum of a function in $x$ and a function in $y$, and similarly, \eqref{eq:steady-constraint-grad} is replaced by imposing that $p-K_u$ and $p-K_v$ vary only in $y$ and $x$, respectively. One could say that we have rewritten the model in such a way that dimensional splitting is no longer a problem. Of course, in the end, a dimensionally split method for the GF formulation \eqref{eq:md-GF-continuous} will yield a truly multi-dimensional method for the original formulation \eqref{eq:acoustic-source}. 

Finally, we would like to remark that firstly, the above reformulation is used to modify the quadrature appearing in the discretization.
This is why this method is better referred to as a  Global Flux \emph{quadrature} approach.   Moreover, the wording global 
refers to the coupling between all the different terms in the PDE, and must not be confused with
the presence of global, integrated quantities. Indeed, at discrete level only differences of them will appear, the lower integration bound can be chosen differently in each cell, and the resulting method is local,  with a usual size of the stencil.

%
%
%
%
%
%
\subsection{Main contribution and outline of the paper}

In this paper, we propose to start from \eqref{eq:md-GF-continuous} to construct, analyze, and implement new numerical methods with enhanced stationarity preserving properties
for multi-dimensional hyperbolic balance laws. 
In Section~\ref{sec:acoustics}, we 
discuss
the types of source terms that we will include, 
and give some insights into the properties of the exact steady states, with a focus on the boundary conditions.
Section~\ref{sec:fem} introduces the  discrete framework of continuous tensor nodal Finite Elements, the standard
Galerkin variational forms, and stabilization based on the SU and OSS methods.
In Section~\ref{sec:GF_source}, we introduce the novel multi-dimensional GF framework into the Galerkin  discrete variational statement.
We prove the structure preserving properties of the new discrete equations, and give a strong characterization of the discrete equilibria,  including an explicit definition and 
estimates of their nodal consistency properties, and  we analyze the interplay between them and the boundary conditions. We then discuss how to modify the stabilization terms such that they are compatible with the discrete equilibria in Section~\ref{sec:GF_stab}.
This extends the work of \cite{brt25} by accounting for forcing terms of different forms (in particular, not only Coriolis-type).
In Section~\ref{sec:DeC}, we introduce the Deferred Correction time discretization that allows to obtain arbitrarily high-order schemes while maintaining the equilibria as steady states.
%
Section~\ref{sec:numerics} shows numerical benchmarks up to order 6 confirming the dramatic enhancements brought by our approach.
Section~\ref{sec:conclusions} ends with open issues and future research directions.



\section{Examples of steady states
}\label{sec:acoustics}
In this section we review the types of steady equilibria we wish to approximate numerically. 
Besides the satisfaction of the  partial differential equation, these solutions may require
the specification of appropriate boundary conditions when they do not have a compact support.

System   \eqref{eq:acoustic-source}
involves both a vector source term $\vec S_{\vec v}$, and a scalar source term $S_p$. The latter is essentially a pressure/mass  pulsation,
while the vectorial source can be seen as a force/acceleration.   The former will be assumed to be of the following general form:
\begin{equation}\label{eq:definition_Sv}
\vec S_{\vec v}=  c(x,y)\vec v^{\perp} - f(x,y) \vec v  + \boldsymbol{\tau}(x,y)
\end{equation}
with $c,f : \mathbb R^2 \to \mathbb R^+$, $\boldsymbol{\tau}:\mathbb R^2 \to \mathbb R^2$, having set  $\vec v^{\perp}= (v, -u)$. The first term models a Coriolis force, the second friction, the last external forcing.

The stationary states are then governed by
\begin{align}
	\nabla p &= \vec S_{\vec v} = c(x,y)\vec v^{\perp} - f(x,y) \vec v  + \boldsymbol{\tau}(x,y),\\
	\nabla \cdot \vec  v &= S_p, 
\end{align}
from which it follows that the curl of $\vec S_{\vec v}$ has to vanish at steady state:
\begin{align}\label{eq:steady-compatibility-curl}
0  &= \nabla^{\perp}\cdot\vec S_{\vec v} =   c(x,y) S_p -f(x,y) \nabla^{\perp}\cdot\vec v   +\nabla^{\perp}\cdot\boldsymbol{\tau}(x,y) 
+ \vec v\cdot\big(\nabla  c(x,y)   - \nabla^{\perp}  f(x,y)\big),
\end{align}
with $\nabla^{\perp}:=(\partial_y,-\partial_x)$; observe that $\nabla^{\perp}\cdot \vec v^{\perp} = \nabla \cdot \vec v$.

Another relation 
at steady state is obtained by combining the steady states of the first two equations in \eqref{eq:acoustic-source},
and integrating them in $x$ and $y$:
\begin{equation}\label{eq:steady-compatibility-Sv}
p(x,y) =\int_{x_0}^x S_u\dd x + p(x_0,y) = \int_{y_0}^y S_v\dd y + p(x,y_0),
\end{equation}
where $p(x,y)$ denotes the pressure at equilibrium. 
This can be recast as 
\begin{equation}\label{eq:steady-compatibility-p}
p(x,y) -  \{ \lambda  p(x_0,y) + (1 - \lambda) p(x,y_0)  \}  = \lambda   \int_{x_0}^x S_u\dd x   + (1-\lambda)  \int_{y_0}^y S_v\dd y \;,\quad \forall \lambda \in \mathbb{R}.
\end{equation}
The above relations highlight a coupling between  the different variables via the boundary conditions. More generally, using the structure of the equations,
and the general form of the steady solutions, one can derive compatibility relations between boundary pressure, and boundary velocities. This allows to set 
conditions on either of the variables independently on what data is given on the boundaries. 
In the following, we study in more detail three types of moving equilibria that can be obtained with different source terms. 
Compatibility relations for the pressure/velocity boundary values at steady state are provided for each of them.  

%

\subsection{Steady states with Coriolis term only}
This is a classical case studied also in recent works using other approaches \cite{audusse2025energy,audusse2018analysis,bouchut2004frontal}.
Introducing only the Coriolis source term (i.e. $f=0$, $\vec \tau = 0$, $S_p=0$), the equilibria can be described by
\begin{equation}\label{eq:coriolis_steady}
	\begin{cases}
		\partial_x p - c v=0,\\
		\partial_y p + c u=0,\\
		\partial_x u + \partial_y v =0.
	\end{cases}
\end{equation}
These include an interaction between the pressure and the velocities. In particular, if the boundary conditions are
prescribed on the pressure $p(x,y_0)$ and $p(x_0,y)$, then we can see from \eqref{eq:steady-compatibility-Sv}
that 
\begin{equation}\label{eq:coriolis_GFsteady}
	\begin{cases}
		p(x,y) - \int_{x_0}^x c(s,y) v(s,y) \dd s 
		= p(x_0,y),\\
		p(x,y) + \int_{y_0}^y c(x,s) u(x,s) \dd s 
		= p(x,y_0),\\
		\mathcal U+\mathcal V= \sigma^x_p(x) + \sigma^y_p(y).
	\end{cases} 
\end{equation}
where $\sigma^x_p$ and $\sigma^y_p$ are arbitrary functions of $x$ and $y$, respectively.
The first two equations imply 
\begin{equation}\label{eq:compatibility_coriolis}
	\int_{x_0}^x c(s,y) v(s,y) \dd s+\int_{y_0}^y c(x,s) u(x,s) \dd s =  p(x,y_0)  - p(x_0,y).
\end{equation}
If the  Coriolis coefficient $c$ is constant, then     \eqref{eq:compatibility_coriolis}  implies 
\begin{equation}
  \mathcal U + \mathcal V  = \dfrac{ p(x,y_0)  - p(x_0,y)}{c} \;\;\Rightarrow\;\;
   \sigma^x_p(x) = \dfrac{ p(x,y_0)}{c}\;,\;\; \sigma^y_p(y) = - \dfrac{ p(x_0, y)}{c}.
\end{equation}

In case the boundary conditions are given on the velocities, we now have
\begin{equation}\label{eq:coriolis_GFsteady}
	\begin{cases}
		p(x,y) - \int_{x_0}^x c(s,y) v(s,y) \dd s 
		= \sigma^y_u(y) = p(x_0,y) ,\\
		p(x,y) + \int_{y_0}^y c(x,s) u(x,s) \dd s 
		= \sigma^x_v(x) = p(x,y_0),
	\end{cases} 
\end{equation}
where $\sigma^x_v(x)= p(x,y_0)$ and $\sigma^y_u(y)=p(x_0,y)$ are to be determined.
To combine the two relations we evaluate the first for $y=y_0$,  and the  second for $x=x_0$. Setting $p_0=p(x_0,y_0)$ we can write the  compatibility relations:
\begin{equation} 
 \sigma^y_u(y)  = p(x_0,y) = p_0 - c \int_{y_0}^y u(x_0,s)\dd s\;,\quad
  \sigma^x_v(x) = p(x,y_0)=p_0 + c \int_{x_0}^x v(s,y_0)\dd s.
\end{equation} 



For constant $c$, inspired by the tests proposed by \cite{ricchiuto2021analytical,brt25}, we consider steady vortex solutions with a solenoidal velocity field
\begin{align} \label{eq:coriolis_equilibrium}
 \vec v(x,y) &= - h(\rho(x,y)) \vecc{x-x_0}{y - y_0}^\perp
\end{align}
with $\rho(x,y)=\|\vec x-\vec x_0\|$ and $h:\mathbb R^+ \to \mathbb R$. 
The pressure is obtained by solving $\nabla p = c \,\vec v^\perp = -c \, h(\rho(x,y)) (\vec x-\vec x_0)$; using that $\rho \nabla \rho = \vec x-\vec x_0$ one finds
\begin{align}\label{eq:coriolis_equilibrium_pressure}
 p(x,y) &= P_0- c\, g(\rho(x,y)) 
\end{align}
with $g'(\rho) = \rho h(\rho)$ and $P_0 = \text{const}$.



\subsection{Steady states with mass source only}\label{sec:mass_source_theory}
In this case, we consider only a  scalar source term $S_p \neq 0$ (that could, depending on the context, be called mass or pressure source). 
It leads to equilibria with a non-vanishing divergence:
\begin{equation}
	\begin{cases}
	\div \vec v = S_p,\\
	p \equiv C\in \mathbb R,
	\end{cases}
\end{equation}
which,  using \eqref{eq:div_free_GFsteady}, becomes
\begin{equation}\label{eq:mass_source_GFsteady}
\begin{cases}
	 \mathcal{U}+\mathcal{V}-\mathcal{K}_p=\sigma^x(x)+\sigma^y(y),\\
	p\equiv C\in \mathbb R.
\end{cases}	
\end{equation}
Note  the pressure is constant, due to the absence of sources in the velocity equations. We  thus do not write compatibility 
relations for the boundary conditions in this case.

%
An example of the previous equilibrium  can be obtained setting
\begin{equation}\label{eq:mass_steady}
	\begin{cases}
		p(x,y)\equiv p_0,\\
		\vec v(x,y):= -a\, h(\rho(x,y))  \vecc{x-x_0}{y-y_0}^\perp   + b \,\nabla g(x,y),\\
		S_p(x,y) :=  b \, \Delta g(x,y),\\
		\vec S_{\vec v} (x,y) := 0, 
	\end{cases}
\end{equation}
where $\rho(x,y)=\|\vec x-\vec x_0\|$, $a$ and $b$ are constants, and   $g: \Omega \to \mathbb R$ and $h: \mathbb{R}^+ \to \mathbb{R}$ are  sufficiently smooth functions. 


We can easily modify the above problem to obtain a moving solution. This is useful for validation purposes.
For example, given a constant vector $\vec a = (a_x,a_y)\in\mathbb{R}^2$, we can show that a translating solution is given by  
\begin{equation}\label{eq:mass_moving}
\begin{cases}
		p(x,y)\equiv p_0 +  b\, \vec a\cdot\nabla  g(x-a_x t,y-a_y t),\\
		\vec v(x,y):=  a\, h(\rho(\vec x - \vec x_0(t))) (\vec x - \vec x_0(t))^\perp +  b\, \nabla g(x-a_x t,y-a_y t),\\
		S_p(x,y) :=  b\, \Delta g(x-a_x t,y-a_y t) - b\, \vec a \cdot \nabla (\vec a\cdot\nabla g) ,\\
		\vec S_{\vec v} (x,y) := 0, 
	\end{cases}
\end{equation}
with $\vec x_0(t) = (x_0-a_xt,y_0-a_yt)$.

\subsection{Steady state with all velocity sources: the Stommel Gyre}\label{sec:stommel_gyre}
We consider the general case in which all momentum sources are present, such that the steady state is governed by
\begin{align}
\nabla\cdot \vec v&=0,\\
 \nabla p &= c(x,y)\vec v^{\perp} - f(x,y) \vec v  + \boldsymbol{\tau}(x,y).
\end{align}
Integrating both components one obtains
$$
\begin{aligned}
p =  & \int_{x_0}^x cv \,\dd s- \int_{x_0}^x fu\,\dd s + \int_{x_0}^x \tau_u \,\dd s +  \sigma^y_u(y),\\
p =  & -\int_{y_0}^y cu \,\dd s- \int_{y_0}^y fv\,\dd s + \int_{y_0}^y \tau_v\,\dd s + \sigma^y_v(x),
\end{aligned}
$$ 
where    $\sigma^x_v=p(x,y_0)$, and  $\sigma^y_u=p(x_0,y)$ are to be determined. To this end we evaluate the first above
for $y=y_0$, and the second for $x=x_0$ to obtain the compatibility relations (as before $p_0=p(x_0,y_0)$):
 \begin{equation} \label{eq:pressure_bc_condition}\begin{aligned}
p(x_0,y) &= p_0 -\int_{y_0}^y (cu)(x_0,s) \,\dd s -\int_{y_0}^y fv(x_0,s)\,\dd s+\int_{y_0}^y \tau_v(x_0,s)\,\dd s, \\
p(x,y_0)  & =p_0 +\int_{x_0}^x (cv)(s,y_0)\,\dd s - \int_{x_0}^x fu(s,y_0)\,\dd s+ \int_{x_0}^x \tau_u(s,y_0) \,\dd s.
\end{aligned}
\end{equation}
Note that, unless $c$ is constant, we cannot use the mass equation to couple the two expressions.

A classical example studied in meteorology is based on geostrophic  shallow water equations, under a linearization assumption around the reference depth $h_0=1$ and with a re-scaling of gravity $g=1$,  the so-called Stommel Gyre model \cite{stommel1948westward,comblen2010practical}. 
The model is the acoustic system \eqref{eq:acoustic-source} with 
\begin{itemize}
\item a Coriolis source term with the coefficient $c(x,y )  := c_1 y + c_0$ (linear in the latitude $y$),
\item a friction source term, where we use a constant $f \in \mathbb R^+$,
\item a source in the velocity equations (wind forcing) in the form $\displaystyle \begin{cases}
		\tau_u = -F \cos(\pi y/b),\\
		\tau_v = 0,
	\end{cases}
$
with $F \in\mathbb R$ constant,
\item no pressure source: $S_p=0$.
\end{itemize}
A well-known particular steady state on the rectangular domain $\Omega = [0,\lambda]\times[0,b]$ is (see  \cite{stommel1948westward} for details):
\begin{align}
& \psi(x,y)= \gamma (b/\pi)^2 \sin(\pi y/b) (k e^{Ax} + w e^{Bx}-1)\\
&	\begin{cases}
		u(x,y) = \partial_y \psi(x,y) = \gamma (b/\pi) \cos(\pi y /b) (ke^{Ax}+we^{Bx}-1),\\
		v(x,y) = -\partial_x \psi(x,y) = -\gamma (b/\pi)^2 \sin(\pi y/b) (kAe^{Ax}+wBe^{Bx}),
	\end{cases} \qquad \text{with } \begin{cases}
	\alpha = \frac{c_0}{f},\\
	\gamma = \frac{F \pi }{bf},
	\end{cases}
\end{align}
with the coefficients $k, w$ given by 
\begin{equation}
	k = \frac{1-e^{B\lambda}}{e^{A\lambda}-e^{B\lambda}},\qquad w=1-k,
\end{equation}
and $A,B$ are the characteristic roots given by
\begin{equation}\label{eq:gyre4}
	\begin{split}
		A=& -\dfrac{\alpha}{2} +\sqrt{\dfrac{\alpha^2}{4} + (\pi/b)^2}\;,  \quad A^2 = (\pi/b)^2 -\alpha A,\\
		B=&-\dfrac{\alpha}{2} -\sqrt{\dfrac{\alpha^2}{4} + (\pi/b)^2}\;,   \quad B^2 = (\pi/b)^2 -\alpha B.
	\end{split}
\end{equation}
The steady pressure field   can be expressed as 
\begin{equation}\label{eq:gyre5}
	\begin{split}
		p = -&F \left( \dfrac{k}{A} e^{Ax} +  \dfrac{w}{B} e^{Bx}  \right) - F(b/\pi)^2 \left(  k A e^{Ax} + w B e^{Bx}  \right)( \cos(\pi\,y/b) -1 ) \\
		-& \left\{c(x,y)\gamma(b/\pi)^2\sin(\pi\,y/b) +\gamma c_0 (b/\pi)^3( \cos(\pi\,y/b) -1 ) \right\}\left(ke^{Ax} + w e^{Bx}-1\right).
	\end{split}
\end{equation}
This is a complex multi-dimensional case that couples all effects, and does not exhibit the same simple structure in terms of boundary conditions.
The variability of the Coriolis coefficient makes it easier  in practice  to assume the velocities to be given on the boundaries, and look for compatible integration
constants for the pressure.   


\section{Discrete setting: Finite Elements  on Cartesian meshes}\label{sec:fem}

The discretization of hyperbolic equations with continuous Finite Element methods (FEM) requires some stabilization techniques. We focus here on the Streamline Upwind (SU) stabilization \cite{hughes86,michel2021spectral,michel2022spectral}, and on the orthogonal sub-scale stabilization (OSS) \cite{CODINA1997373,CODINA20001579,michel2022spectral}. For additional details on the spatial discretization of acoustic equations with stabilized continuous FEM we refer to \cite{brt25}.

\subsection{Tensor-product Finite Element spaces}

Let us consider a rectangular domain $\Omega := \Omega^x \times \Omega^y=[x_0,x_e]\times [y_0,y_e] \subset \mathbb R^2$ and a Cartesian mesh with cells $E_{ij}:=E^x_i \times E^y_j := [x_{i},x_{i+1}]\times [y_j,y_{j+1}]$ with $i=0,\dots, N_x-1$, $j=0,\dots,N_y-1$, $|E^x_i|=\Delta x$, $|E_j^y|=\Delta y$  and $h:=\min(\dx,\dy)$. We define the numerical domains as $\Omega^x_{\dx}:=\cup_{i=0}^{N_x-1} [x_i,x_{i+1}]$, $\Omega^y_{\dy}:=\cup_{j=0}^{N_y-1} [y_j,y_{j+1}]$, and $\Omega_{h}:=\cup_{i=0}^{N_x-1}\cup_{j=0}^{N_y-1} [x_i,x_{i+1}]\times [y_j,y_{j+1}]$.
To be fully precise,  since $\Omega$ is rectangular, $\Omega = \Omega_h$, so we will simply use $\Omega$ below.

We define continuous nodal FEM spaces of polynomial degree $K$ by introducing $K+1$ points in each one-dimensional cell $x_{i,p}\in E^x_i$,  for $p=0,\dots,K,$ and $y_{j,\ell} \in E^y_j$, for $\ell=0,\dots,K$. These points are used to define the Lagrangian basis functions. 
We will use Gauss--Lobatto points in each cell and $x_{i,0}=x_{i-1,K}$ for all $i=1,N_x-1$ and similarly for $y$, to obtain continuity all over the domain.

The relevant spaces over one/two-dimensional domains are 
\begin{subequations}
	\begin{align}
		V^K_{\Delta x}(\Omega^x_{\dx}):=&\left\lbrace q \in \mathcal{C}^0(\Omega_{\dx}^x): q|_E\in \mathbb P^K(E), \forall E \in \Omega^x_{\dx}\right\rbrace,\\
		V^K_{\dy}(\Omega^y_{\dx}):=&\left\lbrace q \in \mathcal{C}^0(\Omega_{\dy}^y): q|_E\in \mathbb P^K(E), \forall E \in \Omega^y_{\dy}\right\rbrace,\\
		V^K_h(\Omega):=&\left\lbrace q \in \mathcal{C}^0(\Omega): q|_E\in \mathbb Q^K(E), \forall E \in \Omega\right\rbrace,
	\end{align}
\end{subequations}
where $ \mathbb P^K$ is the space of univariate polynomials of degree up to $K$, and $\mathbb Q^K$ is the tensor product space of such polynomials in $x$ and $y$ directions.
For simplicity we will occasionally write $V_{\dx}\equiv V^K_{\Delta x}(\Omega^x_{\dx})$, $V_{\dy}\equiv V^K_{\dy}(\Omega^y_{\dx})$, $V_h\equiv V^K_h(\Omega)$.
In each one-dimensional cell $E_i^x$, we consider the Lagrangian basis functions $\varphi^x_{i,p} \in V_{\dx}$ such that $\varphi^x_{i,p}(x_{i,\ell}) = \delta_{p,\ell}$, the Kronecker delta, for every $p,\ell =0,\dots, K$ and $i=0,\dots,N_x-1$. Moreover, since any $\varphi\in V_{\dx}$ must be continuous, 
$$\text{supp}(\varphi^x_{i,p})= E^x_i\text{ for }p=1,\dots,K-1,\quad \text{supp}(\varphi^x_{i,0})=\text{supp}(\varphi^x_{i-1,K}) = E^x_{i-1} \cup E^x_i,$$ 
recalling that $\varphi^x_{i-1,K}=\varphi^x_{i,0}$. The same holds for the basis functions of the $y$ space $\varphi^y_{j,\ell}$.
Finally, the approximation space is $V_{\dx}^K(\Omega_{\dx}^x)=\text{span}\lbrace \varphi^x_{i,p} \rbrace_{i,p=0}^{N_x-1,K}$. We use the same spaces to discretize vector components as those we use for scalars.

This leads to the definition of $V_h$ as a tensor product of the function spaces, with the basis $\varphi_{i,p;j,\ell}(x,y):=\varphi^x_{i,p}(x) \varphi^y_{j,\ell}(y)$ for $i=0,\dots,N_x-1$, $j=0,\dots,N_y-1$, $p,\ell = 0,\dots,K$.
Finally, we will describe any function  $q_h \in V_{\dx}$ having a 
 uni-directional dependence on $x$  as  


\begin{align}
	q_h(x) = \sum_{i=0}^{N_x-1}\sum_{p=0}^K \mathrm q_{i,p}\,\varphi_{i,p}|_{E^x_i}(x) \label{eq:qhdef1d}.
\end{align}
Note that the summation runs over $K+1$ nodes in each element $p\in \{0,\dots,K\}$, corresponding to a degree $K$ interpolation, but $\mathrm q_{i,0}=q_{i-1,K}$ for all $i=1,\dots,N_x$.
Similarly,   a  function $q_h \in V_h$ with dependence on both spatial variables is expanded as
\begin{align}
	q_h(x,y) =\sum_{i=0;j=0}^{N_x-1;N_y-1}\sum_{p=0;\ell=0}^{K;K} \mathrm q_{i,p;j,\ell}\, \varphi^x_{i,p}(x)|_{E^x_i}\,\varphi^y_{j,\ell}(y)|_{E^y_j}, \label{eq:qhdef2d}
\end{align}
with similar constraints on the degrees of freedom $\mathrm q_{i,p;k,\ell}$.
For periodic or homogeneous Dirichlet boundary conditions this can be written as 
\begin{equation}\label{eq:2DFEM_functions}
	q_h(x,y) =\sum_{i=0;j=0}^{N_x-1;N_y-1}\sum_{p=1;\ell=1}^{K;K} \mathrm q_{i,p;j,\ell}\, \varphi^x_{i,p}(x)\,\varphi^y_{j,\ell}(y)=\sum_{i=0;j=0}^{N_x-1;N_y-1}\sum_{p=1;\ell=1}^{K;K} \mathrm q_{i,p;j,\ell}\, \varphi_{i,p;j,\ell}(x,y).
\end{equation}
 We will use the roman font $\mathrm q$ for the degrees of freedom that refer to a function $q_h \in V_h$ in a discrete Finite Element space. See Figure~\ref{fig:FEM_DOFs} for a graphical representation of degrees of freedom (DOFs) $\mathrm q_{i,p;j,\ell}$ in a cell $E_{ij}$ for $K=4$.
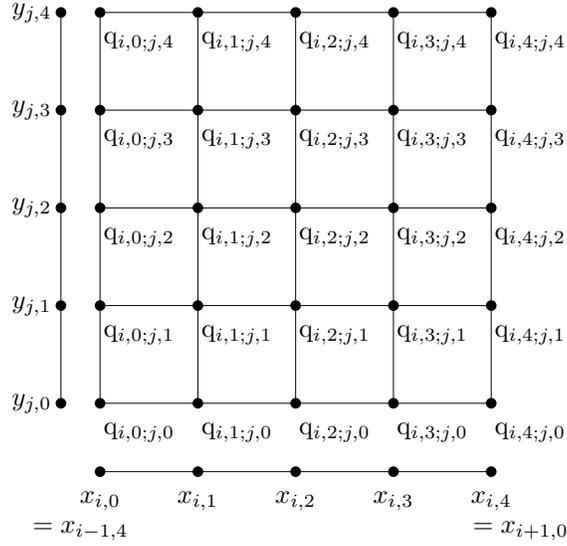
\begin{figure}
	\centering
	\adjustbox{max width=0.49\textwidth}{
		\begin{tikzpicture}[scale=1.3]
			\draw[step=1cm,black,very thin] (0,0) grid (4,4);
			\foreach \x in {0,...,4}
			{
				\foreach \y in {0,...,4}
				{
					\draw (\x+0.4,\y-0.3) node{$\mathrm q_{i,\x;j,\y}$};
					\node[circle, fill=black, inner sep=0.5mm] at (\x,\y){};
				}
			}
			\draw[step=1cm,black,very thin] (0,-0.7) -- (4,-0.7);
			\foreach \x in {0,...,4}
			{
				\draw (\x,-1) node{$x_{i,\x}$};
				\node[circle, fill=black, inner sep=0.5mm] at (\x,-0.7){};
			}
			\draw (4.3,-1.3) node{$=x_{i+1,0}$};
			\draw (-0.2,-1.3) node{$=x_{i-1,4}$};
			\draw[step=1cm,black,very thin] (-0.4,0) -- (-0.4,4);
			\foreach \y in {0,...,4}
			{
				\draw (-0.7,\y) node{$y_{j,\y}$};
				\node[circle, fill=black, inner sep=0.5mm] at (-0.4,\y){};
			}
		\end{tikzpicture}
	}
	\caption{Notation of the degrees of freedom for a function $q$ in element $E_{ij}$ for $\mathbb Q^4$ elements}
\label{fig:FEM_DOFs}
\end{figure}

\subsection{Weak formulation}\label{sec:classical_weak_formulation}

The unstabilized (central) continuous Galerkin weak form of the system \eqref{eq:acoustic-source}  reads
\begin{equation}\label{eq:centralGalerkin_bilinear} 
	\begin{cases}
		\int_\Omega \varphi \big(\partial_t u_h + \partial_x p_h -(S_{u})_ h  \big)\,\dd x \dd y=0, \, &\forall \varphi \in V_h^0,\\
		\int_\Omega \varphi \big(\partial_t v_h + \partial_y p_h -(S_{v})_ h  \big)\,\dd x \dd y=0, \, &\forall \varphi \in V_h^0,\\
		\int_\Omega \varphi \big(\partial_t p_h + \partial_x u_h+ \partial_y v_h -(S_{p})_h  \big)\,\dd x \dd y=0, \, &\forall \varphi \in V_h^0,
	\end{cases}
\end{equation}
subject to suitable boundary conditions (BC), with $V_h^0:=\lbrace \varphi \in V_h: \varphi(x,y) =0 \text{ for }(x,y)\in\partial \Omega\rbrace$.

Let us denote with Greek letters the tuples used for defining the one-dimensional 
basis functions, e.g. $\alpha = (i,p)$, and with $\mathcal{Q}^x:=\lbrace \alpha=(i,p): 0\leq i <N_x,\, 1\leq p \leq K  \rbrace$ the set of all unique indices. Using the discretization of the variables as in \eqref{eq:2DFEM_functions}, we can rewrite the central Galerkin discretization with one-dimensional matrices defined for all $\alpha, \beta \in \mathcal{Q}^x$ (similarly for $\mathcal{Q}^y$) as (see also \cite{brt25} for a lengthier description) 
\begin{equation}\label{eq:operators}
\begin{aligned}
	({M}_x)_{\alpha, \beta} := \int_{\Omega_x} \varphi_{\alpha}(x) \, \varphi_\beta (x) \dd x, \qquad&({D}_x)_{\alpha, \beta} := \int_{\Omega_x} \varphi_{\alpha}(x) \,  \frac{\dd}{\dd x} \varphi_\beta (x) \dd x , \\
	({D}_x^x)_{\alpha, \beta} := \int_{\Omega_x} \frac{\dd}{\dd x}\varphi_{\alpha}(x) \, \frac{\dd}{\dd x}\varphi_\beta (x) \dd x, \qquad&({D}^x)_{\alpha, \beta} := \int_{\Omega_x} \frac{\dd}{\dd x}\varphi_{\alpha}(x) \,   \varphi_\beta (x) \dd x .
	\end{aligned}
\end{equation}
Using these operators, \eqref{eq:centralGalerkin_bilinear} becomes in a vectorial form  
\begin{equation}\label{eq:centralGalerkin_matrix}
	\begin{cases}
		{M}_x\otimes {M}_y \partial_t \mathrm{u} + {D}_x\otimes {M}_y \mathrm{p} -{M}_x\otimes {M}_y \mathrm{S}_u =0,\\
		{M}_x\otimes {M}_y \partial_t \mathrm{v} + {M}_x\otimes {D}_y \mathrm{p} -{M}_x\otimes {M}_y \mathrm{S}_v =0,\\
		{M}_x\otimes {M}_y \partial_t \mathrm{p} + {D}_x\otimes {M}_y \mathrm{u}+ {M}_x\otimes {D}_y \mathrm{v} -{M}_x\otimes {M}_y \mathrm{S}_p =0,
	\end{cases}
\end{equation}
where $\otimes$ denotes the Kronecker product between matrices and $\mathrm{u},\,\mathrm{v},\,\mathrm{p},\,\mathrm{S}_u,\,\mathrm{S}_v,\,\mathrm{S}_p$ are the corresponding coefficients of the functions in $V_h$. For more details, see \cite{brt25}. In practice, this might necessitate the evaluation of some source terms, e.g. $c(x,y)v_h(x,y)$ or $s^u_h(x,y)$ in the same Gauss--Lobatto quadrature points.

\subsection{Stabilization approaches}
For initial value problems, the Lax-Richtmyer stability of   central collocated approximations   as the Continuous Galerkin Finite Element methods   can be shown
under a  $\dx^2$  time step constraint (see e.g. \cite{FDM-Leveque}, chapter 10). Note that this stability criterion only guarantees the boundedness of the discrete solution,
and  is weaker than the usual Fourier (von Neumann) stability criterion requiring the norm of the complex amplification factor to be below 1.

For initial-boundary value problems this  can be improved via appropriate treatments of the boundary conditions \cite{abgrall2020analysis}. It is, however, generally accepted that central collocated methods require stabilization.
Differently from methods using discontinuous data,   we cannot introduce Riemann Problems and upwinding for stabilization. Thus, some stabilizing operator must be explicitly added to the variational  equations.
Below, we present two stabilization approaches
that allow to use explicit time integration with CFL conditions, where $\dt$ depends linearly on $\dx$.


\subsubsection{Streamline upwind  stabilization}
We discuss here the Streamline-Upwind  (SU) stabilization, leading to the well known SUPG (Streamline-Upwind Petrov-Galerkin) discretization  \cite{brooks82,hughes86}.
For our system  the  SU  stabilization term reads
\begin{equation}\label{eq:SUPG_stab}
	\mathcal{ST}\!_\text{SU}\left(\varphi,q_h=(u_h,v_h,p_h)\right) := \int_{\Omega} \alpha h (\partial_x \varphi \, J^x +  \partial_y \varphi \, J^y)(\partial_t q_h + \nabla \cdot \mathsf F(q_h) - \mathsf S(q_h)) \,\dd x\, \dd y,
\end{equation}
with $J^x$ and $J^y$ the Jacobian of the fluxes as in \eqref{eq:acoustics_matrix}. The energy stability associated to this approach is discussed in a more general setting in appendix~\ref{sec:app-stabSU}.
The stabilization terms added to each equation read
\begin{equation}\label{eq:SUPG_weak}
	\begin{cases}
		\int_{\Omega} \alpha h \, \partial_x \varphi \left( \partial_t p_h + \partial_x u_h +\partial_y v_h -(S_{p})_{h} \right)\dd x\dd y,\\
		\int_{\Omega} \alpha h \,\partial_y \varphi \left( \partial_t p_h + \partial_x u_h +\partial_y v_h -(S_{p})_{h}\right)\dd x\dd y,\\
		\int_{\Omega} \alpha h \,\left[ \partial_x \varphi \left( \partial_t u_h + \partial_x p_h -(S_{u})_{h} \right) +\partial_y \varphi \left( \partial_t v_h + \partial_y p_h -(S_{v})_{h} \right) \right]\dd x\dd y,
	\end{cases}
\end{equation}
or, in vectorial form,
\begin{equation}\label{eq:SUPG_matrix}
	\begin{cases}
		&\alpha h \left( {D}^x \otimes {M}_y \partial_t \mathrm{p}+ {D}^x_x \otimes {M}_y \mathrm{u} +{D}^x \otimes {D}_y  \mathrm{v} -{D}^x \otimes {M}_y \mathrm{S}_p   \right),\\
		&\alpha h \left( {M}_x \otimes {D}^y \partial_t \mathrm{p}+ {D}_x \otimes {D}^y \mathrm{u} +{M}_x \otimes {D}^y_y  \mathrm{v} -{M}_x \otimes {D}^y \mathrm{S}_p   \right),\\
		&\alpha h \left( {D}^x \otimes {M}_y \partial_t \mathrm{u}+ {D}^x_x \otimes {M}_y \mathrm{p}- {D}^x \otimes {M}_y \mathrm{S}_u\right) +\\
		&\quad \qquad \alpha h \left( {M}_x \otimes {D}^y \partial_t \mathrm{v} + {M}_x \otimes {D}^y_y \mathrm{p} -{M}_x \otimes {D}^y \mathrm{S}_v  \right).
	\end{cases}
\end{equation}

An explicit time discretization of this method will be discussed in section~\ref{sec:DeC}. The spectral stability of this approach using several explicit 
time-discretization methods is studied in \cite{michel2021spectral,michel2022spectral}. 

\begin{remark}[Streamline upwinding  and steady state preservation]
Being a residual based stabilization, and involving grad(div) type terms, the  SU stabilization   approach has in principle all the desired features to preserve moving equilibria.
Unfortunately,  as shown in \cite{brt25}, the SUPG  scheme is not stationarity preserving because
the   discrete Galerkin  variational form  and the streamline upwind stabilization operator
have different kernels, which do not include the same discrete approximation of the  physical steady state \cite{brt25}. 
This  ultimately leads to the dissipation of steady states, which can only be preserved within the accuracy of the method, thus on sufficiently fine meshes and for short simulation times.
\end{remark}

\subsubsection{Orthogonal Sub-scale stabilization (OSS)}

The Orthogonal Sub-scale Stabilization (OSS) adds a penalty term on the weak derivative of the residual. It was originally presented for Stokes equation \cite{CODINA1997373}, then for convection--diffusion--reaction problems \cite{CODINA20001579,OSSCodinaBadia} and for hyperbolic problems \cite{michel2021spectral,michel2022spectral}.
A multi-dimensional steady state compliant version of this method for homogeneous problems has been proposed already in \cite{brt25}.
For our system, it can be written as
\begin{equation}\label{eq:OSS_weak}
	\begin{aligned}
		&\int_{\Omega} \alpha h \partial_x \varphi \left( \nabla\cdot \vec v_h -(S_p)_h  -w^{ \nabla\cdot \vec v}    \right)\dd x\dd y,\\
		&\int_{\Omega} \alpha h \partial_y \varphi \left( \nabla\cdot \vec v_h -(S_p)_h  -w^{ \nabla\cdot \vec v}    \right)\dd x\dd y,\\
		&\int_{\Omega} \alpha h \partial_x \varphi \left( \partial_x p_h-(S_u)_h -w^{\partial_x p}    \right)\dd x\dd y +\int_{\Omega} \alpha h \partial_y \varphi \left( \partial_y p_h  -(S_v)_h  -w^{\partial_y p}\right)\dd x\dd y,
	\end{aligned}
\end{equation}
with the   $L^2$ projections $w^{ \nabla\cdot \vec v}$, $w^{\partial_x p}$ and $w^{\partial_y p}$ defined by
\begin{align}\label{eq:OSS_weak1}
	\int_\Omega  \varphi (w^{ \nabla\cdot \vec v}  - ( \nabla\cdot \vec v_h -(S_p)_h)) &=0,\\
		\int_\Omega  \varphi (w^{\partial_x p}  - ( \partial_x p_h -(S_u)_h)) &=0, \\
				\int_\Omega  \varphi (w^{\partial_y p}  - ( \partial_y p_h -(S_v)_h))& =0. 
\end{align}
By construction, the terms of the form  $\partial_x p_h  -(S_u)_h - w^{\partial_x p_h}$ are not zero and introduce $L^2$ stabilization.
The energy stability associated to this approach is discussed in a more general setting in appendix \ref{sec:app-stabOSS}.
See also \cite{CODINA20001579,OSSCodinaBadia,michel2021spectral,michel2022spectral,brt25} for more details.
The OSS operator can also be cast in compact matrix form.  Computations show that several terms cancel out. For example, for the first term we have  that 
$$
\mathrm w^{\nabla \cdot \vec v}:= (M_x^{-1}\otimes M_y^{-1}) \left[ D_x\otimes M_y \mathrm u + M_x\otimes D_y \mathrm v - M_x \otimes M_y \mathrm S_p \right]
$$
and
$$
\begin{aligned}
		&\alpha h \left( D^x_x \otimes  M_y \mathrm{u}   - D^x M_x^{-1}  D_x \otimes  M_y \mathrm{u} +  \cancel{D^x \otimes D_y \mathrm{v} -  D^x \otimes  D_y  \mathrm{v} }+  \cancel{D^x \otimes  M_y \mathrm{S}_p- D^x \otimes  M_y \mathrm{S}_p} \right)\\= & \alpha h  (D^x_x- D^x M_x^{-1}  D_x) \otimes  M_y \mathrm{u} .
		\end{aligned}
$$
Proceeding similarly for all terms, we obtain the compact expressions   
\begin{equation}\label{eq:OSS_stab_mat}
	\begin{cases}
		\alpha h  Z_x \otimes  M_y \mathrm{u},\\
		\alpha h  M_x \otimes  Z_y \mathrm{v},\\
		\alpha h \left(  Z_x \otimes  M_y \mathrm{p} +  M_x \otimes  Z_y \mathrm{p}\right),
	\end{cases}
\end{equation}
having defined the matrices
\begin{equation}\label{eq:def_Zmatrix}
\begin{aligned}
 Z_x:= & D^x_x- D^x M_x^{-1}  D_x, \qquad
 Z_y:= &D^y_y- D^y M_y^{-1}  D_y .
 \end{aligned}
\end{equation}
   All the differences between sources and their projection cancel out in the stabilization terms, as shown in the first equation. This will not be the case
in the multi-dimensional GF well-balanced formulation discussed later.  We refer to   \cite{michel2021spectral,michel2022spectral}  
 for a study of the  spectral stability of this approach with  several  explicit  time-discretization methods.  

\begin{remark}[OSS and steady state preservation]
The OSS is based on the approximation of  grad(div)  stabilization terms and it is a good candidate for stationarity preservation.
However, similarly to SUPG,  in \cite{brt25}  the authors  show that  in its  standard formulation the OSS method is not stationarity preserving. It thus
requires sufficiently fine meshes to appropriately resolve  multi-dimensional steady states. 
As for SUPG, this is related to the difference in the kernels of the discrete Galerkin  variational form  and of the  stabilization operator. We refer to  \cite{brt25}  for details.
\end{remark}

\section{Multidimensional Global Flux formulation}\label{sec:GF_source}

We now modify the discrete variational form \eqref{eq:centralGalerkin_bilinear} by incorporating the multi-dimensional GF approach, which
follows the ideas of Section~\ref{sec:GF_multiD_intro}.    As in \cite{brt25}, we will
\begin{enumerate}
\item define a modified variational form whose kernel contains an appropriate discrete notion of steady state; \label{it:discrsteadyitem}
\item characterize the discrete steady state by obtaining a pointwise error estimation w.r.t. the physical/exact steady solution;
\item define a modification of the stabilization  operators consistent with the same steady state. \label{it:discrstabitem}
\end{enumerate}
The nature of the discrete steady state introduced in \ref{it:discrsteadyitem} is such that it is possible (in \ref{it:discrstabitem}) to find a stabilization operator with the same steady state, something not possible otherwise.
As we will see, our approach allows to obtain methods with physically relevant discrete stationary states, which can be precisely characterized.
However, these steady states are not point-wise exact approximations.
For stationarity preserving schemes  point-wise evaluations of the stationary states of the PDE will in general develop
some small amplitude initial layer \emph{before} settling down onto the numerical steady state, which is a good approximation
of the exact one (if the boundary conditions are appropriately set). 
Stationarity preserving methods outperform in any case the standard ones which generally dissipate all but the simplest equilibria.

If one is interested in simulations of very small perturbations, we need to ensure the absence of any initial layer, i.e., that the discrete data are well-prepared with respect to the discrete equilibria of the method. 
In this case, we need to add to the above list:
\begin{enumerate}
\setcounter{enumi}{3}
\item an initialization procedure that projects data  on the discrete kernel. 
\end{enumerate}
We will discuss all these aspects in this section, except the modification of the stabilization terms that will be defined in Section~\ref{sec:GF_stab}.

%

\subsection{Variational form  using Global Flux quadrature}

Our starting point is the formulation \eqref{eq:md-GF-continuous} of section \ref{sec:GF_multiD_intro}:
\begin{equation*}
	\begin{cases}
\partial_t u  + \del_x (p - \mathcal K_u) =0,\\
\partial_t v  + \del_y (p - \mathcal K_v) =0,\\
\partial_tp +  \partial_{xy} G_p =0.
	\end{cases}
\end{equation*}

As already remarked, in the first two equations, a one-dimensional GF formulation is applied, while in the last equation we present a genuinely multi-dimensional GF approach.
Similar results can be obtained by treating the gradient as the divergence of a diagonal tensor.

To achieve a discretization of \eqref{eq:md-GF-continuous},  we need to 
define discrete counterparts of the integrated variables   $\mathcal{U}$, $\mathcal{V}$, $\mathcal{K}_p$ and $\boldsymbol{\mathcal{K}}_{\vec v}$
introduced  in  \eqref{eq:definition_intu_intv}   and \eqref{eq:definition_Gv}.
Following \cite{MANTRI2023112673,brt25},  we will approximate these variables in the same polynomial space as the actual unknowns, i.e.,
with the notation of \eqref{eq:qhdef2d},

\begin{equation}\label{eq:GF-polynomials1}
 \begin{split}
\mathcal{U}_h = & \sum\limits_{i=0;j=0}^{N_x-1;N_y-1} \sum\limits_{p=0;\ell=0}^{K;K} \mathrm{U}_{i,p;j,\ell}\, \varphi_{i,p}^x(x)|_{E_i^x}\varphi_{j,\ell}^y(y)|_{E_j^y}\,,\\
\mathcal{V}_h = &\sum\limits_{i=0;j=0}^{N_x-1;N_y-1} \sum\limits_{p=0;\ell=0}^{K;K} \mathrm{V}_{i,p;j,\ell}\,\varphi_{i,p}^x(x)|_{E_i^x}\varphi_{j,\ell}^y(y)|_{E_j^y} \,, \\
(\boldsymbol{\mathcal{K}}_{\vec v})_h = &\sum\limits_{i=0;j=0}^{N_x-1;N_y-1} \sum\limits_{p=0;\ell=0}^{K;K} (\boldsymbol{\mathrm{K}}_{\vec v})_{i,p;j,\ell}\,\varphi_{i,p}^x(x)|_{E_i^x}\varphi_{j,\ell}^y(y)|_{E_j^y}\,,\\
( \mathcal{K}_{p})_h =&\sum\limits_{i=0;j=0}^{N_x-1;N_y-1} \sum\limits_{p=0;\ell=0}^{K;K} ( \mathrm{K}_{p})_{i,p;j,\ell}\,\varphi_{i,p}^x(x)|_{E_i^x}\varphi_{j,\ell}^y(y)|_{E_j^y} \,.
 \end{split}
\end{equation}
We need to define the coefficients $\mathrm{U}_{i,p;j,\ell}$, $\mathrm{V}_{i,p;j,\ell}$, $(\boldsymbol{\mathrm{K}}_{\vec v})_{i,p;j,\ell}$, and $( \mathrm{K}_{p})_{i,p;j,\ell}$.
To this end, in each cell $E_{ij}:=E^x_i \times E^y_j$ we set (we simplify the notation by avoiding the restriction to the elements which is  clear from the context)
\begin{equation}\label{eq:GF-polynomials2}
 \begin{split}
 \mathrm{U}_{i,p;j,\ell}  = & \mathrm{U}_{i,p;j,0}  + \int_{y_{j,0}}^{y_{j,\ell}} u_h(x_{i,p},y)\,\dd y= \mathrm{U}_{i,p;j,0}+ \sum\limits_{n=0}^{K} \mathrm u_{i,p;j,n} \int_{y_{j,0}}^{y_{j,\ell}} \varphi_{j,n}^y(y)\,\dd y,\quad \ell=1,\dots,K,\, \forall p, \\
  \mathrm{V}_{i,p;j,\ell}  = & \mathrm{V}_{i,0;j,\ell}  + \int_{x_{i,0}}^{x_{i,p}} v_h(x,y_{j,\ell})\,\dd x= \mathrm{V}_{i,0;j,\ell}+ \sum\limits_{m=0}^{K} \mathrm v_{i,m;j,\ell} \int_{x_{i,0}}^{x_{i,p}} \varphi_{i,m}^x(x)\,\dd x,\quad p=1,\dots,K,\, \forall \ell,\\
  ( \mathrm{K}_{p})_{i,p;j,\ell}& =  ( \mathrm{K}_{p})_{i,0;j,0} +  \int_{x_{i,0}}^{x_{i,p}}\int_{y_{j,0}}^{y_{j,\ell}}  (S_p)_h(x,y)\,dy\,\dd x \\
  &=  ( \mathrm{K}_{p})_{i,0;j,0} + \sum\limits_{m=0;n=0}^{K;K} (\mathrm S_p)_{i,m;j,n}
  \int_{x_{i,0}}^{x_{i,p}} \varphi_{i,m}^x(x)\,\dd x \;
  \int_{y_{j,0}}^{y_{j,\ell}} \varphi_{j,n}^y(y)\,\dd y
  \quad (p,\ell) \in \lbrace 0,\dots,p\rbrace ^2 \setminus (0,0),
  \end{split}
\end{equation}
for all $i = 0,\dots, N_x-1$ and $j=0,\dots,N_y-1$. The above relations define summations along grid lines. We are left to define the interface values, which can be found imposing continuity as
\begin{align}
	\mathrm U_{i,p;j,0}&:=U_{i,p;j-1,K} \qquad \text{for all } i=0,\dots,N_x-1,\,j=1,\dots,N_y-1,\,p=0,\dots,K,\\
	 \mathrm V_{i,0;j,\ell}&:=V_{i-1,K;j,\ell} \qquad \text{for all } i=1,\dots,N_x-1,\,j=0,\dots,N_y-1,\,\ell=0,\dots,K, \\
	 (\mathrm K_p)_{i,0;j,0}&:=(K_p)_{i-1,K;j-1,K} \qquad \text{for all } i=1,\dots,N_x-1,\,j=1,\dots,N_x-1.
\end{align}
The remaining boundary values can be set to 0 \cite{MANTRI2023112673}, as the integral is defined up to a constant. The value of these integrated quantities is not so relevant as the integral operators will always be paired with a derivative operator. Hence, only their difference with respect to the border of the cell will be relevant. This means that only \eqref{eq:GF-polynomials2} will be relevant for the actual method.
 
The definitions of the components of $(\boldsymbol{\mathrm{K}}_{\vec v})_{i,p;j,\ell}$ are analogous to those of $ \mathrm{V}_{i,p;j,\ell}$ for  $( \mathrm{K}_{u})_{i,p;j,\ell}$,
and of $ \mathrm{U}_{i,p;j,\ell}$ for  $( \mathrm{K}_{v})_{i,p;j,\ell}$.   
The above relations define summations along grid lines. 
The initial states
in these summations are obtained imposing continuity across elements. In other words we have   $\mathrm{U}_{i,p;j,0}=  \mathrm{U}_{i,p;j-1,K}$,  
$\mathrm{V}_{i,0;j,\ell}=\mathrm{V}_{i-1,K;j,\ell}$ and similarly for all the others. After assembly, all these quantities are defined up to arrays containing constant entries, which will not affect  the final discretization. For this reason, these global integration constants are simply taken equal to zero as in \cite{MANTRI2023112673}. 

%
%
All the above computations   can be re-written through two local linear operators that we denote by one-dimensional  matrices $I_x$ and $I_y$, defined as 
\begin{equation}\label{eq:integral_matrix}
	(I_x)_{i_1,p_1;i_2,p_2}:= \int_{x_{i_1,0}}^{x_{i_1,p_1}} \varphi^x_{i_2,p_2}(x) \dd x,\qquad (I_y)_{j_1,p_1;j_2,p_2}:= \int_{y_{j_1,0}}^{y_{j_1,\ell_1}} \varphi^y_{j_2,\ell_2}(y) \dd y.
\end{equation}
In this way, we can define the discrete GF variables as
\begin{equation}\label{eq:GF-vars}
	\mathrm{U}    = \id_x \otimes I_y \mathrm u  , \qquad 
	\mathrm{V}    = I_x \otimes \id_y \mathrm v,\qquad 
	{\mathrm{K}}_u  = I_x \otimes \id_y \mathrm S_u,\qquad
	{\mathrm{K}}_v  = \id_x \otimes I_y \mathrm S_v,\qquad
	{\mathrm{K}}_p  = I_x \otimes I_y \mathrm S_p,
\end{equation}
with $\id_x$ and $\id_y$ being the identity matrices.
As noted in \cite{MANTRI2023112673,brt25}, when using Gauss-Lobatto interpolation points the  local  operators $I_x$ and $I_y$ are by definition  the  Butcher tableaux
of the well known LobattoIIIA collocation method for ODE integration \cite{hairer}. This is important for the consistency analysis  below.

Discrete equations are now obtained using the Galerkin projection  \eqref{eq:md-GF-continuous}. The resulting (unstabilized central)  discrete variational equations read as follows 
\begin{subequations}\label{eq:centralGalerkinGF_mat}
\begin{equation}\label{eq:centralGalerkinGF_matrix}
	\begin{cases}
	{M}_x\otimes {M}_y\, \partial_t \mathrm{u} + {D}_x\otimes {M}_y \,\mathrm{p} -{D}_xI_x\otimes {M}_y \,\mathrm{S}_u =0,\\
	{M}_x\otimes {M}_y\, \partial_t \mathrm{v} + {M}_x\otimes {D}_y\, \mathrm{p} -{M}_x\otimes {D}_yI_y\, \mathrm{S}_v =0,\\
	{M}_x\otimes {M}_y \,\partial_t \mathrm{p} + {D}_x\otimes {D}_yI_y \,\mathrm{u}+ {D}_xI_x\otimes {D}_y\, \mathrm{v} -{D}_xI_x\otimes {D}_yI_y\, \mathrm{S}_p =0.
	\end{cases}
\end{equation}
Note that the quantity ${D}_x\otimes {D}_yI_y \,\mathrm{u}+ {D}_xI_x\otimes {D}_y\, \mathrm{v}$ in the last equation is exactly the GF $\mathrm{DIV}$ operator 
introduced in \cite{brt25}. Here, we additionally account for the source terms in a multi-dimensional fashion.
Furthermore, we can factor out the differential operators in all the evolution terms as
\begin{equation}\label{eq:centralGalerkinGF_matrix_collected}
	\begin{cases}
		\mathrm R_c^u(\mathrm{u},\mathrm{v},\mathrm{p}) := 
		{M}_x\otimes {M}_y\, \partial_t \mathrm{u} + {D}_x\otimes {M}_y \left(  \mathrm{p} -I_x\otimes \id_y \,\mathrm{S}_u \right) =0,\\
		\mathrm R_c^v(\mathrm{u},\mathrm{v},\mathrm{p}) := 
		{M}_x\otimes {M}_y\, \partial_t \mathrm{v} + {M}_x\otimes {D}_y \left(  \mathrm{p} -\id_x\otimes I_y\, \mathrm{S}_v \right) =0,\\
		\mathrm R_c^p(\mathrm{u},\mathrm{v},\mathrm{p}) := 
		{M}_x\otimes {M}_y \,\partial_t \mathrm{p} + {D}_x\otimes {D}_y \left( \id_x\otimes I_y \,\mathrm{u}+ I_x\otimes \id_y\, \mathrm{v} -I_x\otimes I_y\, \mathrm{S}_p \right)=0.
	\end{cases}
\end{equation}
\end{subequations}
This is the key to being able to later define a stabilization that does not destroy the stationary state.

\begin{remark}[Symmetry of the GF quadrature\label{rem:GFsymm}] Definitions \eqref{eq:GF-polynomials2}  involve integrating, and marching in the positive direction of the reference axes
in an element wise manner.  This may lead that the conclusion that the resulting method has a directional dependence.  This is in fact not true, and one easily checks
that the same method is obtained when \eqref{eq:GF-polynomials2} are reversed and integration is started from the end points. This has been shown in  detail in
Proposition 9 of \cite[section \S4.3, page  13]{brt25},  and is briefly recalled for completeness. Consider for example the reversed  variable
$$
 \tilde V_{i,p;j,\ell} :=  \tilde V_{i,K;j,\ell}  +  \int_{x_{i,K}}^{x_{i,p}} v(s,y_{j,\ell}) \, ds  .
$$
We can easily show that
$$
 \tilde V_{i,p;j,\ell}  =  V_{i,p;j,\ell}    +  \Delta^{V}_{i;j,\ell}  \;,\quad  \Delta^{V}_{i;j,\ell}:=   \tilde V_{i,K;j,\ell} - \tilde V_{i,0;j,\ell} -   \int^{x_{i,K}}_{x_{i,0}} v(s,y_{j,\ell}) \, ds
$$
with $V_{i,p;j,\ell}  $ defined in   \eqref{eq:GF-polynomials2}. The quantity  $ \Delta^{V}_{i;j\ell}$ is the same for all nodes, and as a consequence  we have  on any $E_{ij}$
$$
\partial_x  \tilde{\mathcal V}_h=\partial_x   \mathcal V \qquad  \Rightarrow \qquad   D_x \otimes   D_y \tilde V = D_x \otimes   D_y  V.
$$
Similar developments can be performed for all the variables involved in   \eqref{eq:GF-polynomials2}, and in the variational form \eqref{eq:centralGalerkinGF_matrix_collected}.
This shows that the GF quadrature  formulation is symmetric and independent of the constants dropped.
We refer to \cite{brt25} for more.
\end{remark}

\begin{example}[GF $\mathbb Q^1$ operators]\label{ex:P1-GF}
	The main difference with the classical Finite Element method discussed in section~\ref{sec:classical_weak_formulation} is
 the replacement (in the divergence and source integral)
   of the mass matrices by the operators $D_xI_x$ and similarly in $y$.  We can compare the resulting expressions (in $x$-direction) at lowest order of accuracy, where they can be written as finite differences (note that the integrals are computed with the same Gauss--Lobatto quadrature defining the basis functions, hence, the mass matrix is not exact):
	\begin{itemize}
		\item  continuous Finite Element method: \\$ M_x \simeq \left(0, 1 ,0 \right), \quad  D_x \simeq (-\frac12, 0 , \frac12), \quad  D^x \simeq (\frac12, 0 , -\frac12), \quad  D_x^x \simeq (-1, 2 , -1)$;
		\item GF Finite Element method: $I_x = \begin{pmatrix}
			0&0\\\frac12 & \frac12 
		\end{pmatrix}$, so e.g., 
		$$\mathcal{V}_h(x_{i+1},y_j) = \mathcal{V}_h(x_{i},y_j) + \frac12 \left(v_h(x_i,y_j)+v_h(x_{i+1},y_j)\right).$$ This results in
		$ D_x I_x = \left(\frac14, \frac12,\frac14\right) \text{ (new mass matrix)}, \quad   D_x^x I_x = \left(\frac12, 0,-\frac12\right) = D^x$. 
         As has been shown in \cite{brt25}, at lowest order the GF approach yields Finite Differences that appear in stationarity preserving methods for linear acoustics such as \cite{morton01}. This is in particular true for the coefficients $\left(\frac14, \frac12,\frac14\right) $ of the ``mass matrix'' $D_xI_x$.
	\end{itemize}
\end{example}

\subsection{Discrete equilibria and  super-convergence}\label{sec:properties}
In this section, we aim at characterizing explicitly the discrete steady states of the GF variational form \eqref{eq:centralGalerkinGF_matrix_collected}, and at analyzing their consistency with the steady states of the PDE.

%

%

\subsubsection{Characterization of the discrete steady states}

\begin{proposition}[Discrete steady states of the GF discretization]\label{th:div_kernel}
Equations  \eqref{eq:centralGalerkinGF_mat} obtained from the GF formulation \eqref{eq:md-GF-continuous} of \eqref{eq:acoustic-source}
admit the discrete steady states characterized by 
	\begin{align}
		\begin{split}
	               \mathrm{p}_{i,k;j,s} -({\mathrm{K}}_u)_{i,k;j,s} &=  \sigma^y_u(j,s),\\
	               \mathrm{p}_{i,k;j,s} -({\mathrm{K}}_v)_{i,k;j,s} &= \sigma^x_v(i,k),\\
		\mathrm{U}_{i,k;j,s}+\mathrm{V}_{i,k;j,s} - ({\mathrm{K}}_p)_{i,k;j,s} &=\sigma^x_p(i,k)+\sigma^y_p(j,s). \label{eq:discreteUplusV}
		\end{split}
	\end{align}
	where $\sigma^y_u, \sigma^y_p$ are arbitrary functions of $j,s$ ($y$-direction), and $\sigma^x_v, \sigma^x_p$ arbitrary functions of $i,k$ ($x$-direction) only.
\end{proposition}
\begin{proof}
	By construction $D_x\otimes D_y (\sigma^x_p + \sigma^y_p )=0$  as $D_y \sigma^x_p=0$ and  $D_x\sigma^y_p=0$, and similarly $D_x\sigma^y_u=0$  and $D_y\sigma^x_v=0$.
	Thus, the data \eqref{eq:discreteUplusV} is 
	in the kernel of the spatial operator of \eqref{eq:centralGalerkinGF_mat}. 
\end{proof}

The above result provides a global representation of the discrete steady states. 
A more local description of  these states can be obtained as follows. 
Consider the  local assembly
$$
\left[ D_x\otimes D_y(\mathrm{U} + \mathrm{V} -	{\mathrm{K}}_p  )\right]_{\alpha;\beta} =\sum_{E\ni (\alpha;\beta) } \left[ D_x^E \otimes D_y^E \left(  (\iden_x^E \otimes I_y^E ) \, u^E +
   (I_x^E \otimes \iden_y^E)  \, v^E  - (I_x^E \otimes I_y^E  )\, S_p^E
   \right) \right]  _{\alpha;\beta}   ,
$$
where the superscript $^E$ denotes the local operators and arrays in a generic element $E$ containing the point $(\alpha,\beta)$, e.g. $(D_x^E)_{\alpha_1,\alpha_2} := \int_{E^x} \varphi^x_{\alpha_1}(x) \frac{\dd}{\dd x} \varphi^x_{\alpha_2}(x) \,\dd x$. 
Let us now  focus on    element $E=E_{ij}$,  and consider  the local arrays
$$
[u_{0}^E]_{i,s;j,p}:= u_{i,0;j,p}\quad\forall s=0,\dots,K, \qquad [v_0^E]_{i,s;j,p}:= v_{i,s;j,0}\quad\forall p=0,\dots,K.
$$
 Using the fact that $D_x^E \otimes \id_y^E u_0^E=\id_x^E \otimes D_y^Ev_0^E=0$, we can immediately write
 $$
\left[D_x\otimes D_y(\mathrm{U} + \mathrm{V} -	{\mathrm{K}}_p  )\right]_{\alpha;\beta} 
= \sum_{E\ni (\alpha;\beta)} [(D_x^E \otimes D_y^E)\Phi^E_p]_{\alpha;\beta},
$$
where $\Phi^E_p$ is the array of integrated residuals 
\begin{equation}\label{eq:div-residuals0}
\Phi^E_p :=    (\iden_x^E\otimes I_y^E)   (u^E -u_0^E)    + (I_x^E\otimes \iden_y^E)   (v^E -v^E_0)  - (I_x^E \otimes I_y^E  \, S_p^E)  ,
\end{equation}
whose entries are readily shown to be given by
\begin{equation}\label{eq:div-residuals1}
\begin{split}
[\Phi^E_p]_{i,s;j,t} = &   \int_{y_{j,0}}^{y_{j,t}}\!\!( u_h(x_{i,s},y) - u_h(x_{i,0},y))\dd y +  \int_{x_{i,0}}^{x_{i,s}}\!\!( v_h(x,y_{j,t}) - v_h(x,y_{j,0}))\dd x  
-  \int_{x_{i,0}}^{x_{i,s}} \!\! \int_{y_{j,0}}^{y_{j,t}}\!\! (S_p)_h\dd x \dd y \\ =&  \int_{x_{i,0}}^{x_{i,s}}  \int_{y_{j,0}}^{y_{j,t}} ( \partial_x u_h +\partial_yv_h - (S_p)_h )\dd x\dd y. 
\end{split}
\end{equation}
In a similar manner, we can easily show that 
\begin{equation}\label{eq:grad-residuals}
\begin{aligned}
&\left[ D_x \otimes \id_y ( \mathrm{p}  -	\mathrm{K}_u  ) \right] _{\alpha;\beta}  = \sum_{E\ni (\alpha;\beta)} [D_x^E\otimes \id_y^E \Phi^E_u]_{\alpha;\beta}\;,\quad [\Phi^E_u]_{i,s;j,t} =&  \int_{x_{i,0}}^{x_{i,s}}( \partial_x p_h-(S_u)_h)(x,y_{j,t}) \dd x,\\
&\left[ \id_x \otimes D_y ( \mathrm{p}  -	{\mathrm{K}}_v  )\right]_{\alpha;\beta}  = \sum_{E\ni (\alpha;\beta)} [\id_x^E \otimes D_x^E \Phi^E_v]_{\alpha;\beta}\;,\quad [\Phi^E_v]_{i,s;j,t} =& \int_{y_{j,0}}^{y_{j,t}}( \partial_y p_h-(S_v)_h)(x_{i,s},y) \dd y.
\end{aligned}
\end{equation}
If now  we consider the {\it  residual vector} 
\begin{equation}\label{eq:residuals}
\Phi^E =( \Phi^E_u, \Phi^E_v, \Phi^E_p)^t,
\end{equation} 
the components $(i,s;j,t)$ of $\Phi^E$ are integrals of the steady state residual, over sub-cells $[y_{i,0},y_{i,t}]$,  $[x_{i,0},x_{i,s}]$,  and $[x_{i,0},x_{i,s}]\times[y_{i,0},y_{i,t}]$ respectively.
We   can now reformulate the  discrete steady states as follows.
\begin{proposition}[Discrete steady states of the GF  equations and vanishing subcell integrals]\label{th:div_kernel-phi}
The  GF quadrature variational equations  \eqref{eq:centralGalerkinGF_mat}
admit the discrete steady states characterized by the condition
\begin{equation}\label{eq:phi_eq_0}
\begin{split}
\Phi^E_{i,s;j,t} = 0\qquad 
 \forall s,t  \text{ and } \,\forall E_{i,j} ,
\end{split}
\end{equation}
where $\Phi^E$ are the arrays of sub-cell integrated residuals   \eqref{eq:div-residuals1}, \eqref{eq:grad-residuals} and \eqref{eq:residuals}.
\end{proposition}
%

Several remarks are in order: 
\begin{enumerate}
\item  {\it From collocated to face averaged data.} The proposed multi-dimensional GF  approach introduces 
a link between a collocated representation based on nodal values, and mimetic approaches based on    face-averaged  and cell-averaged quantities.
The variables   $\mathcal U_h$, $\mathcal V_h$ are integrated values of the  velocities in the directions normal to the faces, 
Similarly,   $(\mathcal K_p)_h$ is a sub-cell averaged term  times a cell volume. These are natural objects to express   integrals  of the form
\begin{align}
	\begin{split}
 \iint (\del_x u + \del_y v-S_p) \dd x \dd y &= \int [u]_x \dd y + \int [v]_y \dd x  - \iint  S_p \dd x \dd y\\&
 = \left[ \int u \dd y \right]_x + \left[ \int v \dd x \right ]_y - \iint  S_p \dd x \dd y,
	\end{split}
\end{align}
which characterize  the steady state. The concepts generalize naturally to 3 space dimensions.
\item  {\it Vanishing residuals.} Characterization \eqref{eq:phi_eq_0} of the steady states  is fully local to an element, and  
involves in principle as many conditions as  the number of nodes in the mesh.  
However,  one easily checks that  there are at most $3\times K^2\times (N_x\times N_y) +  K(N_x+N_y)$ independent conditions, for the  system under consideration.
These can be obtained on each element from the integrals on the smallest
 sub-quadrilaterals (and 1D sub-cells) defined by the Gauss-Lobatto points.
The remaining ones  can be obtained as sum of these integrals.
These   conditions characterize the dimension of kernel of the methods proposed.
This residual property   somewhat generalizes   the steady state preserving
properties of   previously developed residual distribution  methods \cite{HDR-MR,RD-ency,AR:17,Abgrall2022}.
\item {\it Number of boundary conditions.}  The total number of nodal unknowns  when using an interpolation degree $K$ is     $  3\times ( 1+ KN_x)\times (1+ KN_y)  $ for the  system under consideration. Assuming that the number of equations obtained from the vanishing residuals condition are independent, then a unique
equilibrium state would be determined by imposing at least $3\times ( 1+ KN_x)\times (1+ KN_y) - 3\times K^2\times (N_x\times N_y) -  K(N_x+N_y) $ boundary values.
%
This is still often not the case, as \eqref{eq:phi_eq_0} gives often redundant conditions, which makes this number of boundary conditions only 
a lower bound to characterize an equilibrium. In practice, the conditions imposed depend for each problem on the data actually available on each boundary. 
\end{enumerate}

\subsubsection{Consistency analysis}\label{sec:consistency_analysis}

The  next step is to provide consistency estimates of the discrete equilibria identified above. 
To this end, we need to  improve and generalize the nodal consistency argument  given in \cite{brt25}.

\begin{proposition}[Nodal   super-convergence   and line-by-line/row-by-row projection]\label{th:consistency1}
On a Cartesian domain $\Omega = [x_0,x_e]\times[y_0,y_e]$,  consider   a  steady state solution composed of a vector field $\vec{v}_e=(u_e,v_e)$ and a pressure $p_e$ with   sufficient regularity,
say  $u_e,v_e,p_e \in C^{P}(\Omega)$  with $P$ large enough. 
Assume that $(\vec{v}_e,p_e) $    satisfy  
\begin{align*}
\partial_xp_e(x,y)&=S_u(\vec v_e,p_e;x,y),\\
\partial_yp_e(x,y)&=S_v(\vec v_e,p_e;x,y),\\
\partial_x u_e(x,y)+\partial_y v_e(x,y)&=S_p(x,y),
\end{align*}
 at every point in space, and in particular at all collocation points.  Let $ (\vec{v}^{\partial}_e,p^{\partial}_e)=(\vec{v}_e,p_e)_{\partial  \Omega} $ denote
 the boundary values of the steady solution.   Given an ODE $U'(t)=F(U,t)$, let us consider multi-stage  integrators
\begin{equation}\label{eq:ode-int}
U_k-U_0 = (I_xF)_k\;, \quad U_k-U_0 =  (I_yF)_k\;, 
\end{equation}
which are  exact when $F$ is a polynomial of degree  $M$, also for an intermediate stage $k$.  Then, 
there exists
a discrete projection $(\vec{v}_e,p_e)_h$ of the exact solution
on the discrete Finite Element space such that
\begin{itemize}
\item the   order of accuracy    of the projection  only depends on   the accuracy of the  integrators, in particular we have  $(\vec{v}_e,p_e)_h = (\vec{v}_e,p_e) + \mathcal{O}(h^{M})$; 
\item  there exist order $\mathcal{O}(h^{M})$ perturbed boundary values   $(\tilde{ \vec{v}}^{\partial}_e,\tilde{p}^{\partial}_e) =  (\vec{v}_e,p_e)_{\partial  \Omega}+  \mathcal{O}(h^{M})$,
 such that  $(\vec{v}_e,p_e)_h$ is in the kernel of  the GF variational form \eqref{eq:centralGalerkinGF_matrix}, and in particular
 it verifies  Propositions \ref{th:div_kernel} and \ref{th:div_kernel-phi},  with   $(\vec{v}_e,p_e)_h\big|_{\partial  \Omega} =(\tilde{\vec{v}}^{\partial}_e,\tilde{p}^{\partial}_e)$ 
and  with  $(u_e)_h\big|_{\partial  \Omega} =u^{\partial}_e$, $(v_e)_h\big|_{\partial  \Omega} =v^{\partial}_e$,  $(p_e)_h\big|_{\partial  \Omega} =p^{\partial}_e$
holding true on at least one of the four sides of $\partial\Omega$ for each variable;  
\item   the GF variational formulation using Gauss-Lobatto  points  admits super-convergent
 discrete steady states verifying the  error estimate of order $\mathcal{O}(h^{M})$ with $M=K+2$  for polynomial interpolation of degree $K\ge2$.
\end{itemize}
\begin{proof}
See Appendix \ref{app:cons-proof1}.
\end{proof}
\end{proposition}

We would like to make  several    remarks. 
One is that the accuracy of the integrators $M$ is not necessarily related to the polynomial degree of the finite element space.
Examples using multi-step--like procedures with  accuracy independent on the polynomial representation are provided in the finite-volume context in \cite{KAZOLEA2025106646}.  Similar approaches are possible
in the Finite Element setting and will be explored in the future.  Second, as already remarked, boundary conditions play a fundamental role in the definition of the  discrete steady states. 
This is especially true  when accounting for the  presence  of sources, which make the use of  homogeneous or periodic boundary conditions not always suitable.

 The  above proposition and its proof provide two important elements. The first message is that, even if one does not use well prepared initial data,
the discrete solution   will settle, after some short transient, on a discrete steady state that has  accuracy  $h^M$ everywhere, including on the boundaries.
The numerical results reported in section~\ref{sec:numerics} provide examples of this. Secondly, the proof is based on a constructive   
approach that explicitly introduces a projection of exact steady states onto the kernel of the scheme. This projection is based on 
a line-by-line/row-by-row  integration method. This method, which is recalled in more detail the next section (see also Appendix \ref{app:cons-proof1}),
provides at least one  way of constructing well prepared data that are exactly preserved by the scheme. 
This  is very useful when studying the evolution of very small perturbations.  The next section gives a more in depth discussion on other approaches to obtain well prepared data.

\subsubsection{Projections and well prepared initial data in practice}\label{sec:init_data}

We discuss ways of obtaining well-prepared data. In this case,  we are given a steady state of the PDE and aim at obtaining discrete data 
which are a steady state of the numerical method, and verify the accuracy estimate of proposition \ref{th:consistency1}.

A first approach to achieve this objective is  constructed in the proof of proposition \ref{th:consistency1} (see Appendix \ref{app:cons-proof1}). It consists
of  line-by-line/row-by-row iterations exploiting the integration tables $I_x$ and $I_y$.  Given boundary data $ \tilde u(x_0,y)$, $ \tilde v(x,y_0)$, and $\tilde p(x_0,y_0)$,
it boils down to evaluating nodal values of the unknowns from the  integrals  
\begin{equation}\label{eq:linebyline}
\begin{aligned}
u(x,y) = & \tilde u(x_0,y)  - \int_{x_0}^x ( \partial_xu_e  )_h(s,y) ds \\
v(x,y) =& \tilde v(x,y_0)  -  \int_{y_0}^y ( \partial_y v_e  )_h(x,s) ds  \\
p(x,y) =& \tilde p(x_0,y_0)  + (1-\lambda)\left( \int_{x_0}^x ( \partial_xp_e  )_h(s,y_0) ds  +  \int_{y_0}^y ( \partial_y p_e  )_h(x,s) ds \right)  \\&\qquad\;\;\;\quad+ \lambda\left(
 \int_{y_0}^y ( \partial_y p_e  )_h(x_0,s) ds  +  \int_{x_0}^x ( \partial_xp_e  )_h(s,y) ds \right)
\end{aligned}
\end{equation}
with $\lambda\in\mathbb{R}$ arbitrary, and having denoted by  $( \partial_xu_e  )_h$ the Lagrange interpolant on the Gauss-Lobatto points of the   partial derivative of the  exact velocity,
and similarly for the other quantities.   As shown in Appendix \ref{app:cons-proof1}
the discrete  data \eqref{eq:linebyline} above 
provide a projection of an exact  stationary solution of the PDE into the kernel of the scheme. They satisfy the consistency estimates of the proposition.
In~\cite{brt25} the above projection was called line-by-line/row-by-row (LobattoIIIA) projection (when the integrator tables are those obtained from the Gauss-Lobatto collocation points). 

%

The directional line-by-line integration can be easily set up and  provides quite satisfactory results. 
Unfortunately, it relies on the correct definition of the initial values on a specific boundary, and  a choice of   the integration direction. It thus  
introduces an asymmetry  in the   error of the initial data with respect to the exact one. 
This lack of symmetry is in contrast with the symmetry property of the  GF quadrature recalled in Remark~\ref{rem:GFsymm}.
To improve on this, one could imagine some iterative correction combining ODE integration in opposite directions. 

%

A more efficient approach, used here, is to   obtain well-prepared data  using a projection based on a global
optimization.  Compared to the homogeneous case \cite{brt25}, 
here we need to account for the complexity introduced by the forcing terms both in the PDE and in the boundary conditions.
%
%
Given an initial steady state $(\vec v_0 ,p_0)$,
we first compute the nodal components of the velocity  as
\begin{subequations}\label{eqs:initialization}
\begin{equation}\label{eq:optimization}
\mathop{\mathrm{arg\,min}}\limits_{\mathrm u,\mathrm v \in \mathcal{C}} \left( \lVert \mathrm u-\mathrm u_0\rVert^2+\lVert \mathrm v-\mathrm v_0\rVert^2 \right), \qquad \mathcal{C}:=\lbrace \mathrm u,\mathrm v:  {D}_x \otimes D_y I_y \mathrm u + D_xI_x \otimes {D}_y \mathrm  v=\mathrm S_p\rbrace,
\end{equation}
where $\mathrm u_0, \mathrm v_0$ are the interpolation of the analytical initial conditions. Once these are known,
we reconstruct $\mathrm p$ using the compatibility conditions \eqref{eq:steady-compatibility-p} with $\lambda=\frac12$, which reads
\begin{align}\label{eq:pressure_ic}
	\begin{split}
p_h(x_\alpha,y_\beta):=\lambda  \left(  \int_{x_0}^{x_\alpha} S_u(\xi,y_\beta) \dd \xi + p_e(x_0,y_\beta) \right) + (1-\lambda) \left( \int_{y_0}^{y_\beta} S_v(x_\alpha,\xi) \dd \xi +p_e(x_\alpha,y_0)\right),
\end{split}
\end{align}
\end{subequations}
where, as everywhere else in this paper, the  source  terms are expanded on the finite element basis, and their  integrals are simply evaluated by multiplication of the nodal  array of source terms with 
the operators $I_x$ and $I_y$.  To solve the   optimization problem \eqref{eq:optimization}, we use the trust-region constrained method which is  implemented in \texttt{scipy.optimize.minimize}. On the other hand, we recall that the so obtained $p$ is not in both kernels of the two velocity evolution operators, but it carries an $O(h^{M})$ error in the non-homogeneous case, because of the boundary conditions.

 Finally,
we will  also use 
the simple strategy of running a long-time simulation and using its result as the equilibrium value for a perturbation simulation, adding the perturbation on top of it. 

\section{Stabilization methods in multi-dimensional Global Flux form}\label{sec:GF_stab}

In this section, we    adapt  the SU and OSS stabilization techniques to the GF quadrature framework, including source terms as well.
For both approaches, we will study the kernels, showing that  Proposition~\ref{th:div_kernel-phi} still holds, and
verify the consistency estimates associated to Proposition \ref{th:consistency1}. 
 
\subsection{SU-GF stabilization} 
We start from the integral form associated to the streamline upwind stabilization, presented in  equation \eqref{eq:SUPG_weak} and use the GF form \eqref{eq:md-GF-continuous}. This leads to
\begin{equation}\label{eq:SUPG-GF_weak}
	\begin{cases}
		\int_{\Omega} \alpha h \,\partial_x \varphi \left( \partial_t p_h + \partial_{xy}(G_{p})_h  \right)\dd x\dd y,\\
		\int_{\Omega} \alpha h \,\partial_y \varphi \left( \partial_t p_h +  \partial_{xy}(G_{p})_h \right)\dd x\dd y,\\
		\int_{\Omega} \alpha h  \, \nabla  \varphi\cdot  \left[ \partial_t \vec v_h + \nabla p_h - \vecc{\del_x (\mathcal{K}_u)_h}{\del_y (\mathcal K_v)_h}   \right]\dd x\dd y.
	\end{cases}
\end{equation}
Using the same approximation ansatz used for the GF Galerkin approximation, we can 
  readily write the matrix form for the GF approximation: 
\begin{equation}\label{eq:SUPG_GF_matrix_collect}
	\begin{cases}
	&\alpha h \left( {D}^x \otimes {M}_y \partial_t \mathrm{p}+ {D}^x_x \otimes {D}_y\left( \mathrm{U}
		+ \mathrm{V} -	{\mathrm{K}}_p   \right) \right) ,\\
		&\alpha h \left( {M}_x \otimes {D}^y \partial_t \mathrm{p}+ {D}_x \otimes {D}^y_y \left( \mathrm{U}
		+ \mathrm{V} -	{\mathrm{K}}_p   \right) \right),\\
		&\alpha h \left( {D}^x \otimes {M}_y \partial_t \mathrm{u}+ {D}^x_x \otimes {M}_y \left( \mathrm{p}- \mathrm K_u \right)\right) +\\
		&\phantom{mmmmmm}\alpha h \left( {M}_x \otimes {D}^y \partial_t \mathrm{v} + {M}_x \otimes {D}^y_y \left(\mathrm{p} -\mathrm K_v \right)  \right).
%
	\end{cases}
\end{equation}
or, expanding the integrated velocities $\mathrm U$, and $\mathrm V$ and sources $\mathrm K_u$, $\mathrm K_v$, $\mathrm K_p$
\begin{equation}\label{eq:SUPG_GF_matrix_collected}
	\begin{cases}
		\textrm {ST}^u_\text{SU-GF}(\mathrm{u},\mathrm{v},\mathrm{p}) &= \alpha h \left( {D}^x \otimes {M}_y \partial_t \mathrm{p}+ {D}^x_x \otimes {D}_y \left( \id_x\otimes I_y \mathrm{u} +I_x \otimes \id_y  \mathrm{v} - I_x \otimes I_y \mathrm{S}_p   \right)\right),\\
		\textrm {ST}^v_\text{SU-GF}(\mathrm{u},\mathrm{v},\mathrm{p}) &= \alpha h \left( {M}_x \otimes {D}^y \partial_t \mathrm{p}+ {D}_x \otimes {D}^y_y \left( \id_x \otimes  I_y \mathrm{u} +I_x \otimes \id_y  \mathrm{v} -I_x \otimes  I_y \mathrm{S}_p   \right)\right),\\
		\textrm {ST}^p_\text{SU-GF}(\mathrm{u},\mathrm{v},\mathrm{p}) &= \alpha h \left( {D}^x \otimes {M}_y \partial_t \mathrm{u}+ {D}^x_x \otimes {M}_y \left( \mathrm{p}- I_x \otimes \id_y \mathrm{S}_u \right)\right) +\\
		&\phantom{mmmmmm}\alpha h \left( {M}_x \otimes {D}^y \partial_t \mathrm{v} + {M}_x \otimes {D}^y_y \left(\mathrm{p} -\id_x \otimes I_y \mathrm{S}_v\right)  \right).
	\end{cases}
\end{equation}
Comparing this expression to \eqref{eq:SUPG_matrix}, one observes
that  the second derivatives are applied to the integrated, global quantities, which contain the sources and all the space derivatives, i.e., all terms governing the (discrete) steady states
discussed in section~\ref{sec:GF_source}. In contrast, standard approaches apply stabilization differently to different terms and thus equilibria between these terms are easily destroyed.

The full multi-dimensional GF Streamline Upwind stabilized scheme  is defined by
\begin{equation}\label{eq:supg-gf}
\begin{aligned}
\mathrm R_c^u(\mathrm{u},\mathrm{v},\mathrm{p}) + 	\mathrm{ST}^u_\text{SU-GF}(\mathrm{u},\mathrm{v},\mathrm{p}) =0 ,\\
\mathrm R_c^v(\mathrm{u},\mathrm{v},\mathrm{p}) + 	\mathrm{ST}^v_\text{SU-GF}(\mathrm{u},\mathrm{v},\mathrm{p}) =0 ,\\
\mathrm R_c^p(\mathrm{u},\mathrm{v},\mathrm{p}) + 	\mathrm{ST}^p_\text{SU-GF}(\mathrm{u},\mathrm{v},\mathrm{p}) =0 ,
\end{aligned}
\end{equation}
with  GF Galerkin centered operators defined in \eqref{eq:centralGalerkinGF_matrix_collected}
  and stabilization given by \eqref{eq:SUPG_GF_matrix_collected}. 
The steady states of the stabilized scheme are characterized by the following statement. 
\begin{proposition}[Streamline Upwind stabilized GF scheme: discrete solutions and consistency]\label{th:stab-su}    The  discrete steady kernel of the
stabilized scheme \eqref{eq:supg-gf}    verifies both  propositions \ref{th:div_kernel} and \ref{th:div_kernel-phi}.
\begin{proof}
The proof of the first point is identical to those of propositions  \ref{th:div_kernel} and \ref{th:div_kernel-phi}, applied to the discrete stabilization operators \eqref{eq:SUPG_GF_matrix_collect}. 
%
\end{proof}
\end{proposition}
%
%
%
We recall  that  the consistency estimate is related to the fact that discrete solutions are obtained using a    line-by-line/row-by-row projection  \eqref{eq:linebyline}.
In particular, when this projection is  performed with the LobattoIIIA solver, we obtain a consistency of order $K+2$ for interpolation of degrees $K\ge 2$, see \cite{hairer}, which leads to the super-convergence result.
We conclude that the SU-GF scheme has the same stationarity preservation properties as the unstabilized method.

\subsection{OSS-GF stabilization}
Similarly to the previous section, we now apply
the variational form associated to the OSS stabilization method to
the GF formulation:
\begin{equation}\label{eq:OSS-GF_weak}
	\begin{cases}
		\int_{\Omega} \alpha h \, \partial_x \varphi  \left(  \partial_{xy}(G_{p})_h - w^{\partial_{xy}G_{p}}  \right)\dd x\dd y,\\
		\int_{\Omega} \alpha h \,\partial_y \varphi \left( \partial_{xy}(G_{p})_h - w^{\partial_{xy}G_{p}}  \right)\dd x\dd y,\\
		\int_{\Omega} \alpha h   \,\partial_x\varphi \left[ \partial_x ( p -   \mathcal{K}_{u} )_h  - w^{\partial_x p}  \right]\dd x\dd y
		+  \int_{\Omega} \alpha h   \partial_y\varphi \left[ \partial_y ( p -   \mathcal{K}_{v} )_h  - w^{\partial_y p }  \right]\dd x\dd y,
	\end{cases}
\end{equation}
complemented with the $L^2$ projections 
\begin{equation}\label{eq:OSS-GF_L2proj}
	\begin{cases}
		  \int_{\Omega}  \varphi  \left( w^{\partial_{xy}G_{p}} - \partial_{xy}(G_p)_h    \right)\dd x\dd y &= 0,\\
		 \int_{\Omega}  \varphi  \left( w^{\partial_x p} - \partial_x (p-\mathcal{K}_{u})_h    \right)\dd x\dd y &=0,\\
		 \int_{\Omega}  \varphi  \left( w^{\partial_y p} - \partial_y (p-\mathcal{K}_{v})_h    \right)\dd x\dd y &=0.
	\end{cases}
\end{equation}
Differently from the standard formulation, in this case we do not have cancellation of so many terms in  the assembly.   
However, all terms group together nicely in terms of the GF.
After some simplifications, we obtain the following matrix form of the OSS-GF stabilization (cf. \eqref{eq:GF-vars} and \eqref{eq:def_Zmatrix} for the notation):
\begin{equation}\label{eq:OSS_GF}
	\begin{cases}
		\mathrm{ST}^u_\text{OSS-GF}(\mathrm{u},\mathrm{v},\mathrm{p})=\alpha h \,Z_x \otimes  D_y ( U + V - K_p),\\
		\mathrm{ST}^v_\text{OSS-GF}(\mathrm{u},\mathrm{v},\mathrm{p})=\alpha h \,D_x \otimes  Z_y ( U + V - K_p),\\
		\mathrm{ST}^p_\text{OSS-GF}(\mathrm{u},\mathrm{v},\mathrm{p})=\alpha h \left[ Z_x \otimes  M_y (\mathrm{p} -K_u) + M_x\otimes Z_y  (\mathrm{p} -K_v)  \right].
	\end{cases}
\end{equation}
These expressions are somewhat similar to \eqref{eq:OSS_stab_mat}, but 
involve now the  multi-dimensional GFs, which is the key to the steady state preserving property of the method. 
The full expressions in terms of the arrays of pressure and velocities are easily obtained by combining the above formulas with \eqref{eq:GF-vars}. 
The full multi-dimensional GF OSS stabilized scheme  is defined by
\begin{equation}\label{eq:oss-gf}
\begin{aligned}
\mathrm R_c^u(\mathrm{u},\mathrm{v},\mathrm{p}) + 	\mathrm{ST}^u_\text{OSS-GF}(\mathrm{u},\mathrm{v},\mathrm{p}) =0 ,\\
\mathrm R_c^v(\mathrm{u},\mathrm{v},\mathrm{p}) + 	\mathrm{ST}^v_\text{OSS-GF}(\mathrm{u},\mathrm{v},\mathrm{p}) =0 ,\\
\mathrm R_c^p(\mathrm{u},\mathrm{v},\mathrm{p}) + 	\mathrm{ST}^p_\text{OSS-GF}(\mathrm{u},\mathrm{v},\mathrm{p}) =0 ,
\end{aligned}
\end{equation}
with  GF Galerkin centered operators defined in \eqref{eq:centralGalerkinGF_matrix_collected}, 
and stabilization given by \eqref{eq:OSS_GF} combined with  \eqref{eq:GF-vars}. 
The form of the OSS-GF stabilization allows to easily prove properties similar to those of the SU-GF scheme.  

%
\begin{proposition}[Orthogonal Subscale Stabilized GF scheme: discrete solutions and consistency]\label{th:stab-oss}  
The  discrete steady kernel of the
stabilized scheme \eqref{eq:oss-gf}    verifies both  propositions \ref{th:div_kernel} and \ref{th:div_kernel-phi}.
\begin{proof}
The proof of the first point is identical to those of propositions  \ref{th:div_kernel} and \ref{th:div_kernel-phi}, applied to the discrete stabilization operators \eqref{eq:OSS_GF}. 
\end{proof}
\end{proposition}

Once more we recall  that  the quality of the consistency estimate is related to the super-convergence of the LobattoIIIA
ODE integrator  \cite{hairer} involved in the
line-by-line/row-by-row projection  \eqref{eq:linebyline}.  
 The analysis shows that the OSS-GF scheme  has  the same steady state preservation properties as
 the unstabilized method.

\section{Explicit time integration using deferred correction} \label{sec:DeC}

In all the previous formulations, we have always presented the time-continuous form. In order to obtain an explicit and arbitrarily high-order time discretization, we will employ the Deferred Correction (DeC) method \cite{dutt2000spectral,minion2003semi,abgrall2017high,micalizzi2022new}. This technique consists of defining two discretizations: $\mathcal{L}^2$ a high order implicit method, e.g., Lobatto IIIA, and $\mathcal{L}^1$ an explicit low order version of $\mathcal{L}^2$.

%
%
%
%
%
%
%
%
%
To define the high order operator $\mathcal{L}^2$, we introduce a continuous FEM discretization of the time interval $[t_n,t_{n+1}]$ through some Lagrangian basis functions $\gamma_{r}(t)$ for $r=0,\dots,M$, defined on the Gauss-Lobatto temporal sub-levels $t_n=t_n^0<\dots < t^M_n=t_{n+1}$. For an ODE $q'+F(q)=0$, let us denote by $q^m\approx q(t_n^m)$ for $m=0,\dots,M$ the approximation of $q$ in the sub-time level. Let  $\underline{q}$ denote the array of all the $q^m$. For each sub-level $m=1,\dots,M$, we define the high-order operator
\begin{equation}
	\mathcal{L}^{2,m}(\underline{q}) := q^m-q^0 + \sum_{r=0}^M F(q^r)\int_{t_n^0}^{t^m_n} \gamma_r(t)\, \dd t \approx q^m-q^0 + \int_{t_n^0}^{t^m_n} F(q(t)) \,\dd t.
\end{equation}
We define $\theta^m_r:=\frac{1}{\dt}\int_{t_n^0}^{t^m_n} \gamma_r(t) \dd t$.
The equation $\mathcal{L}^2(\underline{q})=0$ 
is the Lobatto IIIA method with $M+1$ stages. Instead of solving it directly, we introduce a low-order operator that will be actually solved, while the high-order operator will be used just as a correction term:
\begin{equation}
	\mathcal{L}^{1,m}(\underline{q}) := q^m-q^0 + (t^m_n-t_n^0) F(q^0)\approx q^m-q^0 + \int_{t_n^0}^{t^m_n} F(q(t_n)) \dd t.
\end{equation}
We define $\beta^m := \frac{t^m_n-t^0_n}{\dt}$.
$\mathcal{L}^1(\underline{q})=0$ corresponds to a series of explicit Euler methods applied to different intervals.
Moreover, in the $\mathcal{L}^1$ we can introduce further simplifications such as mass lumping or removing the stabilization, to create a very simple scheme to solve.
The DeC iterative method is then defined by
\begin{equation}
	\label{eq:DeC}
	\begin{cases}
		\underline{q}^{(0)} :=(q(t_n),\dots, q(t_n)),\\
		\mathcal{L}^1(\underline{q}^{(k)})=\mathcal{L}^1(\underline{q}^{(k-1)})-\mathcal{L}^2(\underline{q}^{(k-1)}), \qquad k=1,\dots, \kappa,\\
		q_{n+1}:= q^{(\kappa)}(t^M_n).
	\end{cases}
\end{equation}
with $\kappa$ the number of iterations performed.
It can be shown that the order of accuracy of the method is the minimum between the order of accuracy of $\mathcal{L}^2$ and the number of iterations $\kappa$, under classical assumptions on the time-step size \cite{abgrall2017high,micalizzi2022new}. Hence, it will be natural to consider $\mathbb P^K$ spatial discretizations that lead to order $K+1$ in the unsteady regime. Knowing that $\mathcal L^2$ achieves order $2M$ with Gauss--Lobatto points, we will set it such that $2M\geq K+1$, so that at least order $K+1$ is achieved in the $\mathcal L^2$ operator. Then, the number of iterations are naturally set to $\kappa:=K+1$ to have, for the whole DeC, a space-time order $K+1$ at least.

The DeC allows to combine the spatial and time discretization into unique operators $\mathcal L^1$ and $\mathcal L^2$. Let us 
therefore endow also the PDE variables with a superscript denoting the sub-time step stages of the DeC.
As an example, the SU $\mathcal L^2$ operator for the $u$ equation reads
\begin{equation}
	\begin{split}
		\mathcal{L}^{2,m}_u (\underline{q}):=& (M_x\otimes M_y) (\mathrm u^m-\mathrm u^0) +\sum_{r=0}^M \theta^m_r \Delta t \left( (D_x\otimes M_y) \mathrm p^r -(M_x\otimes M_y)\mathrm S_u^r \right) \\
		&+ \alpha h (D^x\otimes M_y) (\mathrm p^m-\mathrm p^0) + \alpha h \sum_{r=0}^M \theta^m_r \Delta t \left( (D_x^x\otimes M_y) \mathrm u^r + (D^x\otimes D_y) \mathrm v^r-(D^x\otimes M_y) \mathrm  S_p^r  \right),
	\end{split}
\end{equation}
which, combined with the $v$ and $p$ equations create a coupled linear system of $3N_xN_yM$ equations.
Since in our case $M_x$ is diagonal (the quadrature rule used to compute the all matrices is given by the Gauss--Lobatto nodes that defining the Lagrangian basis function), this issue can be overcome by defining the $\mathcal L^1$ operator as
\begin{equation}
	\begin{split}
		\mathcal{L}^{1,m}_u (\underline{q}):=& (M_x\otimes M_y) (\mathrm u^m-\mathrm u^0) + \beta^m_r \Delta t \left( (D_x\otimes M_y) \mathrm p^0 -(M_x\otimes M_y)\mathrm S_u^0 \right),
	\end{split}
\end{equation}
which leads to an explicit scheme in the DeC formulation. This means that the complex SUPG structure of time derivative in the stabilization term can be simply solved within an explicit scheme. 
A detailed Fourier stability  analysis of DeC time stepping with    stabilized  finite element discretizations, including both the SU and OSS stabilization methods, 
is provided in \cite{michel2021spectral} for the  one-dimensional advection equation.  A two-dimensional extension on structured triangulations is presented in   \cite{michel2022spectral}.  
In this work, the time step is set to
$$
\Delta t = \text{CFL} \dfrac{h}{c}.
$$
Specific values of the CFL, and of other discretization parameters, are discussed in the next section.
The GF version follows a similar approach, where we change the spatial discretization operators, i.e.,
\begin{equation}
	\begin{split}
		\mathcal{L}^{2,m}_u& (\underline{q}):= (M_x\otimes M_y) (\mathrm u^m-\mathrm u^0) +\sum_{r=0}^M \theta^m_r \Delta t \left( (D_x\otimes M_y) \mathrm p^r -(D_xI_x\otimes M_y)\mathrm S_u^r \right) \\
		+ &\alpha h (D^x\otimes M_y) (\mathrm p^m-\mathrm p^0) + \alpha h \sum_{r=0}^M \theta^m_r \Delta t ( D_x^x\otimes D_y) \left( (\id_x\otimes I_y) \mathrm u^r + (I_y\otimes \id_y) \mathrm v^r-(I_x\otimes I_y) \mathrm  S_p^r  \right),
	\end{split}
\end{equation}
to be consistent with the GF ones, as well as in the $\mathcal L^1$ operators.

For the sake of brevity, we do not expand in detail all the equations for the SUPG method, nor for the OSS methods, but similar arguments can be applied: stabilization terms are present only in the $\mathcal L^2$ operators where time derivative are replaced by $q^m-q^0$ for the $m$-th equation; all the GF operators are inserted both in $\mathcal{L}^2$ and $\mathcal{L}^1$.

In the steady state regime,  the time discretization   plays no  role. In particular, all terms vanish when
the discrete data is in the kernel of the spatial discretization, so all the     super-convergence results apply independently of the time scheme.
More details can be found in \cite{brt25}.

%
%

\section{Implementation and numerical validation}\label{sec:numerics}
In this section, we validate the numerical method presented above with a series of tests with various source terms. We will start with Coriolis effects in Section~\ref{sec:vortex_coriolis}, then we will add a mass source in Section~\ref{sec:mass_source}, and in Section~\ref{sec:stommel_gyre_num} 
we will end with the Stommel Gyre test that includes  several sources at once. For each of them, we will perform different studies: steady state equilibrium study with convergence, characterization of the equilibrium and perturbation analysis.

Concerning the implementation, for all involved polynomials and quadrature rules, in space and time, we have used Gauss--Lobatto points that guarantee super-convergence in the steady state case. 
This choice, corresponding to the classical Spectral Element approach \cite{PATERA1984468,Kopriva2010}, also makes   the variational statement in strong form  \eqref{eq:centralGalerkin_bilinear}  equivalent to the  weak form, thanks to a discrete summation by parts property \cite{Kopriva2010}.  This also makes all the mass matrices diagonal.

In all the  tests, we  use $\text{CFL}=0.1$ for all polynomial degrees smaller or equal to 5 and $\text{CFL}=\frac{1}{2(2p+1)}$ for higher degrees. For the SU stabilization, we will use $\alpha = 0.05$ for polynomial degrees less or equal to 5, and $\alpha=0.02$ for higher degrees. For OSS, we will set $\alpha = 0.01$ for $\mathbb Q^1$ and $\mathbb Q^2$, while $\alpha = 0.04$ for higher polynomials in all tests.
These parameters turned out to be stable in all the performed simulations.



\subsection{Steady equilibria with  Coriolis force}\label{sec:vortex_coriolis}

We first study    the approximation of steady equilibria  for the system \eqref{eq:coriolis_steady} with the Coriolis term only.
In particular, let us  consider the vortex solution  \eqref{eq:coriolis_equilibrium}--\eqref{eq:coriolis_equilibrium_pressure}   on the square $\Omega = [0,1]^2$,
with a
Coriolis coefficient $c=0.2$ and $P_0 = 1$.
We choose $h(\rho):= 20 e^{-100 \rho^2}$ and thus $g(\rho):= \frac{1}{10} e^{-100 \rho^2}$. 
For these choices the vortex perturbation $h(\rho)$ reaches values below machine precision at the boundaries. 
Hence, one can work with  Neumann homogeneous boundary conditions on all variables.\\

\begin{table}
	\centering
	\caption{Vortex with Coriolis force:  $p$ and $u$ errors and convergence rates  with SU and OSS stabilization}\label{tab:coriolis_vortex_conv_all}
	\vspace{0.2cm}
	\foreach \n in {2,...,6}{
		\pgfmathtruncatemacro\result{\n-1}
		\centering 
			\begin{adjustbox}{max width=\textwidth}
			\begin{tabular}{|c||cc|cc||cc|cc|||cc|cc||cc|cc|}
				\hline
				&\multicolumn{4}{|c||}{SU $\mathbb Q^{\result}$}
				&\multicolumn{4}{|c|||}{SU-GF $\mathbb Q^{\result}$}
				&\multicolumn{4}{|c||}{OSS $\mathbb Q^{\result}$}
				&\multicolumn{4}{|c|}{OSS-GF $\mathbb Q^{\result}$}
				\\ \hline
				{\color{white}\tiny$N\!\!$}$N${\color{white}\tiny$\!\!N$} & err $u$ & err $p$ & O-$u$ & O-$p$& err $u$ & err $p$ & O-$u$ & O-$p$& err $u$ & err $p$ & O-$u$ & O-$p$& err $u$ & err $p$ & O-$u$ & O-$p$ \\
				\hline
				\input{figures/LinAc2D_coriolis_vortex/big_table_ord\n.tex}\\ \hline 
			\end{tabular}
		\end{adjustbox}\\[2pt]
	}
\end{table}

We start by investigating the grid   convergence of the schemes. To this end, we  initialize the discrete solution with 
the analytical steady state interpolated at grid points, and we run until
 time $T=1$.  We then compute the errors with respect to the exact nodal solution.  
The computations have been run with both the SU and  OSS stabilized schemes,  in standard and GF formulation. 
Polynomial  degrees $K\in\{1,\dots,5\}$ are considered, with meshes of decreasing size as the degree is increased
to avoid reaching machine accuracy already on the first meshes and to have comparable number of total degrees of freedom between different orders. 
We report in Table~\ref{tab:coriolis_vortex_conv_all}
the resulting errors, and convergence rates. By symmetry, and to save space, only the $u$ velocity component is reported.

The errors obtained with the GF formulations are systematically  lower than those obtained
with the standard scheme, even on the coarsest meshes. This initial error offset is variable and can go from a factor 2 to a full order of magnitude. 
As the meshes are refined, the results of both the SU stabilized method and the OSS confirm the theoretical $K+2$ super-convergence estimates for $K\ge 2$.
The error reductions observed when using the GF schemes are then between one and almost three orders of magnitude.
The standard schemes require one or two additional mesh refinements to match the  accuracy levels of the GF formulation. \\


%
%
\begin{figure}
	\centering
	\includegraphics[width=0.31\textwidth]{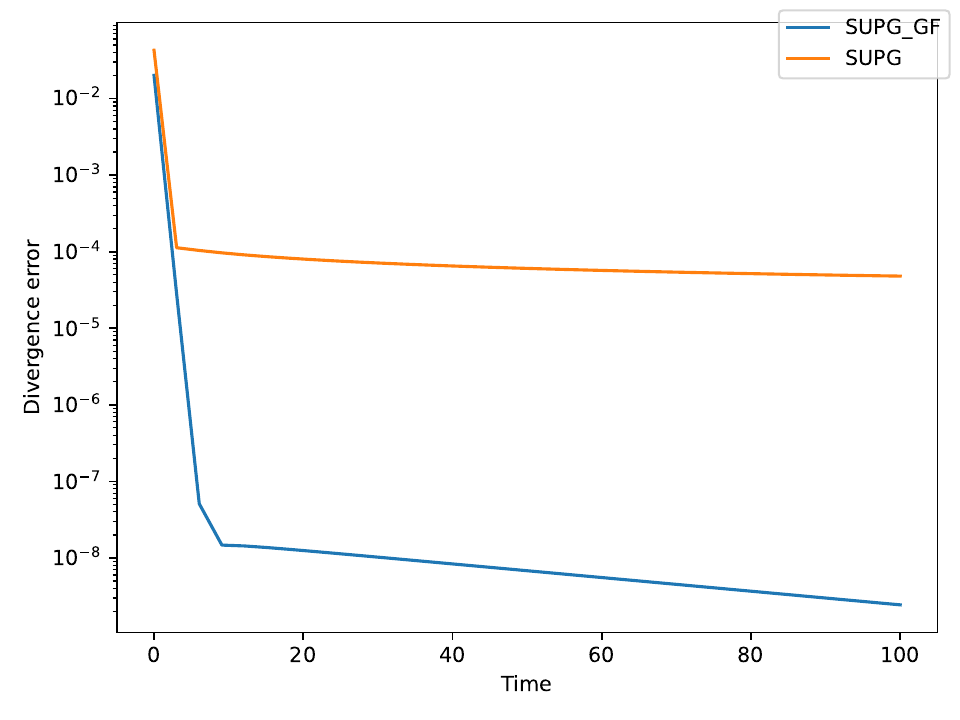}
	\includegraphics[width=0.31\textwidth]{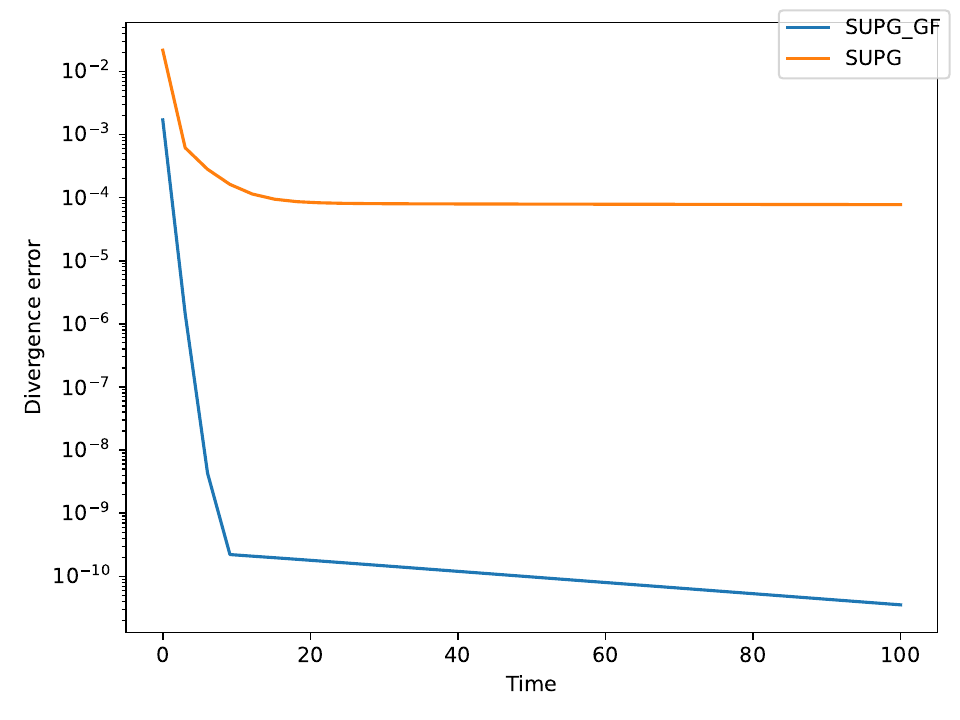}
	\includegraphics[width=0.31\textwidth]{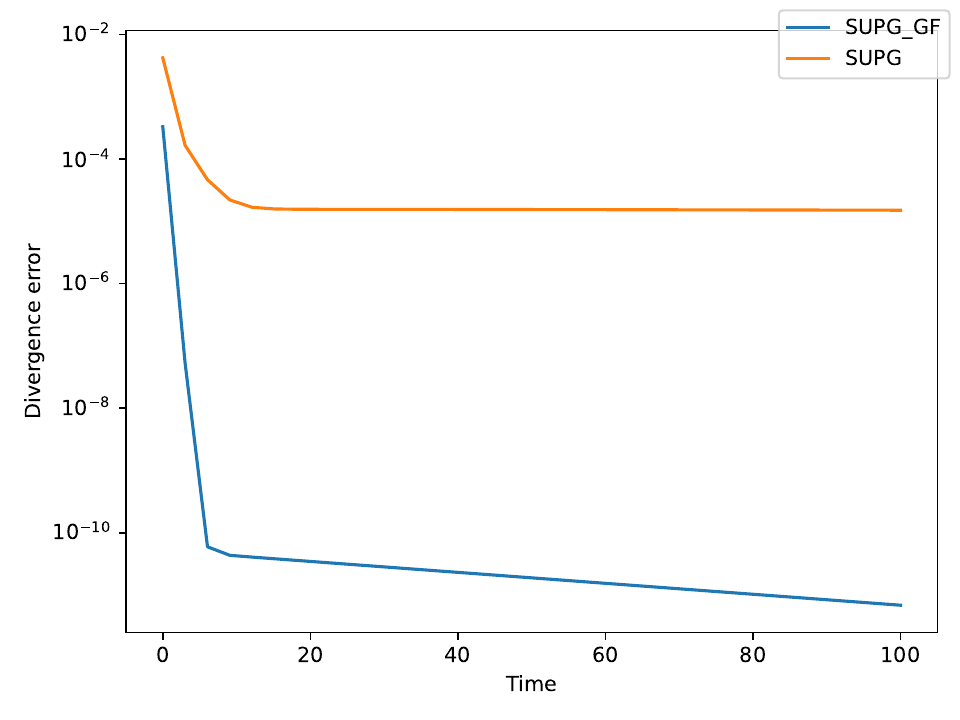}\\
	\includegraphics[width=0.31\textwidth]{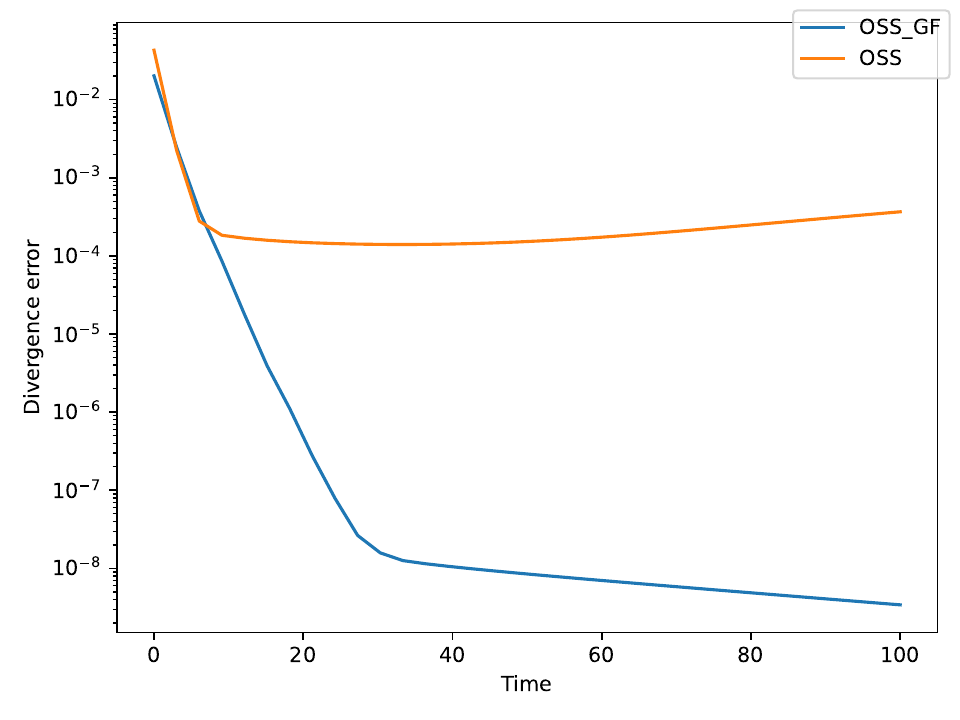}
	\includegraphics[width=0.31\textwidth]{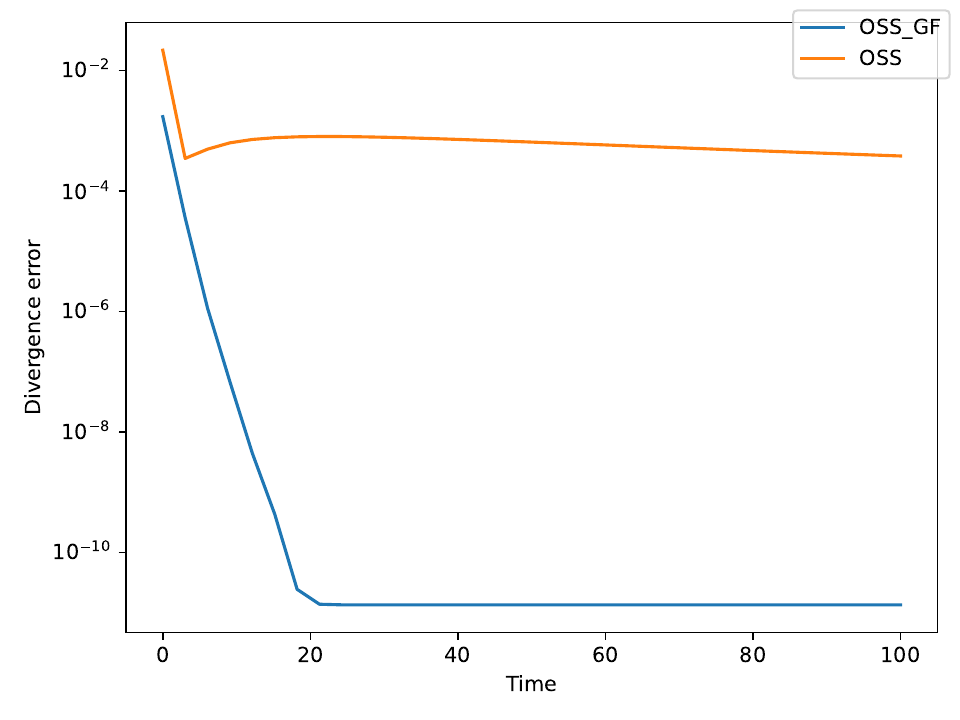}
	\includegraphics[width=0.31\textwidth]{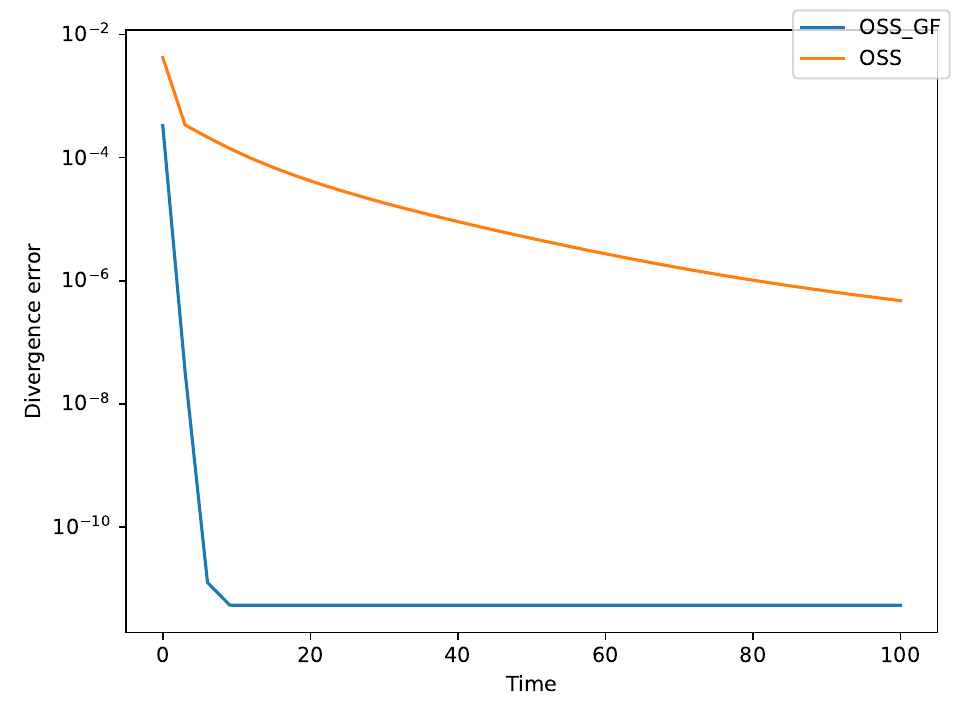}
	\caption{Vortex with Coriolis force: time evolution of the norm of discrete divergence (${D}_x \otimes  M_y \mathrm u + M_x \otimes {D}_y \mathrm v$ for classical methods and ${D}_x \otimes  D_yI_y \mathrm u + D_x I_x \otimes {D}_y \mathrm v$ for GF simulations). Left column: $\mathbb Q^1$ on   $40 \times 40$ mesh. Middle column:  ,  $\mathbb Q^2$ on $20\times 20$ mesh. Right column:   $\mathbb Q^3$ with $13\times 13$ mesh. Top: SU stabilization. Bottom: OSS stablization.  }\label{fig:disc_div_long_vortex_coriolis}
\end{figure}

We then consider   long-time simulations to study the convergence to the discrete steady state.
We measure the evolution of the discrete divergence operators relevant for each method: the usual Galerkin weak form
 ${D}_x \otimes M_y \mathrm u + M_x \otimes {D}_y \mathrm  v$ for the standard schemes, and  
 ${D}_x \otimes D_y I_y \mathrm u + D_xI_x \otimes {D}_y \mathrm  v$ for the GF methods, 
 as given in Propositions \ref{th:div_kernel} and \ref{th:div_kernel-phi}. 
The results are reported in Figure \ref{fig:disc_div_long_vortex_coriolis}. The SU-GF converges  extremely quickly  (roughly 10 time steps), 
leading to values of the divergence well below the errors reported in Table \ref{tab:coriolis_vortex_conv_all} on the corresponding meshes.
The standard SU scheme also seems to 
reach some stationary state, but the state is
not in the kernel of the Galerkin
weak form. We are unable to provide
any characterization of this state, but it displays visible artefacts and a departure from circular symmetry, see Figure~\ref{fig:simulation_long_coriolis}. 

For the OSS stabilization the situation is  very similar. Very low values of the discrete divergence are reached within few steps with the OSS-GF formulation,
while the standard scheme does not seem to reach anything close to the kernel of the Galerkin divergence.  We again are unable to 
characterize the long term solutions of the standard scheme in any way.\\

\begin{figure}
	\centering $\mathbb Q^1, N_x=N_y=40$\\[2mm]
	\begin{minipage}{0.29\textwidth}
	\centering SU standard $\lVert \mathbf v \rVert$\\
	\includegraphics[width=\textwidth,trim={0 0 70 0},clip]{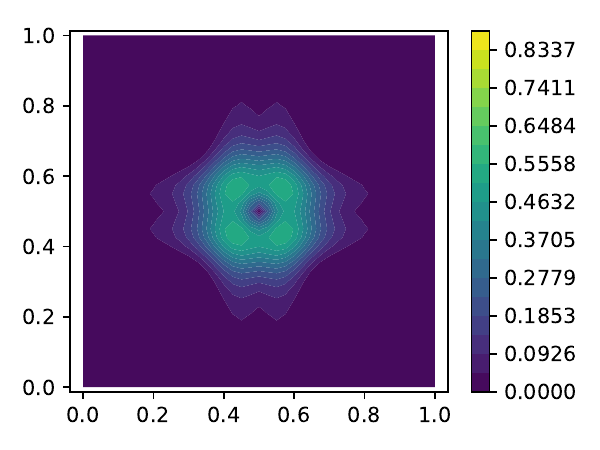}
\end{minipage}
	\begin{minipage}{0.29\textwidth}
	\centering SU--GF $\lVert \mathbf v \rVert$\\
	\includegraphics[width=\textwidth,trim={0 0 70 0},clip]{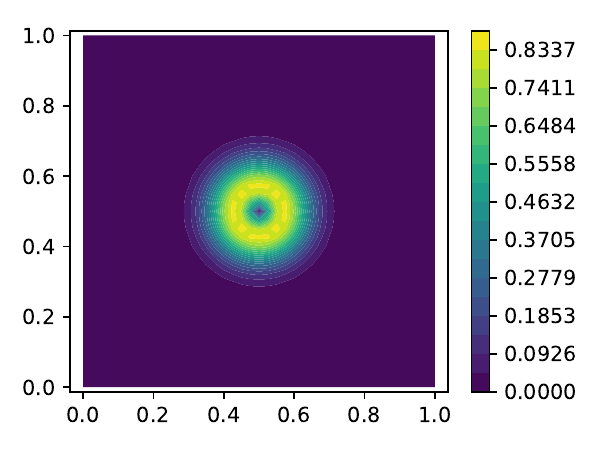}
\end{minipage}
	\begin{minipage}{0.29\textwidth}
	\centering exact $\lVert \mathbf v \rVert$\\
	\includegraphics[width=\textwidth,trim={0 0 70 0},clip]{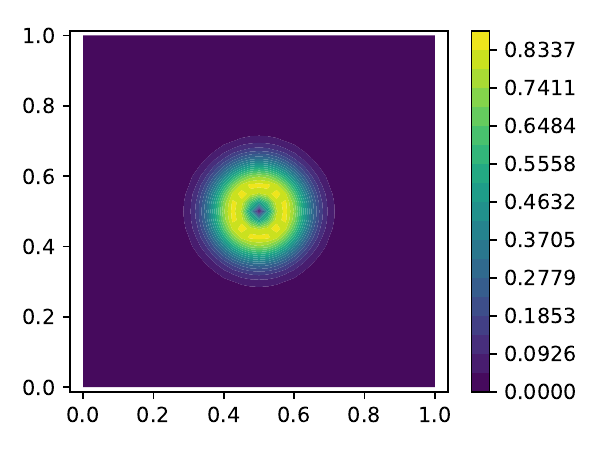}
\end{minipage}
\begin{minipage}{0.09\textwidth}
	\includegraphics[width=\textwidth,trim={220 20 0 0},clip]{figures/LinAc2D_coriolis_vortex_long/final_sol_unorm_exact_ord_2_N_0040.pdf}
\end{minipage}
\\	\begin{minipage}{0.29\textwidth}
	\centering OSS standard $\lVert \mathbf v \rVert$\\
	\includegraphics[width=\textwidth,trim={0 0 70 0},clip]{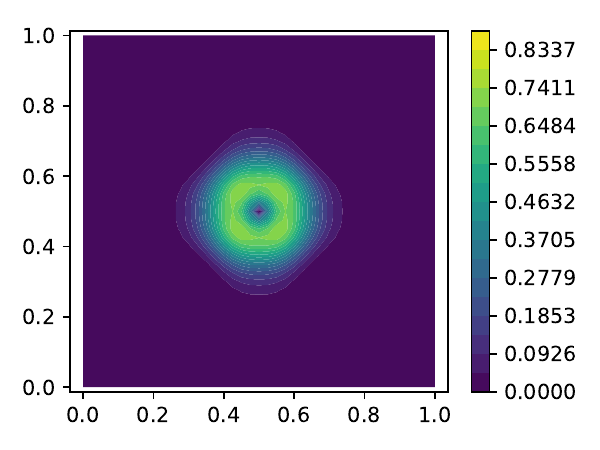}
\end{minipage}
	\begin{minipage}{0.29\textwidth}
	\centering OSS--GF $\lVert \mathbf v \rVert$\\
	\includegraphics[width=\textwidth,trim={0 0 70 0},clip]{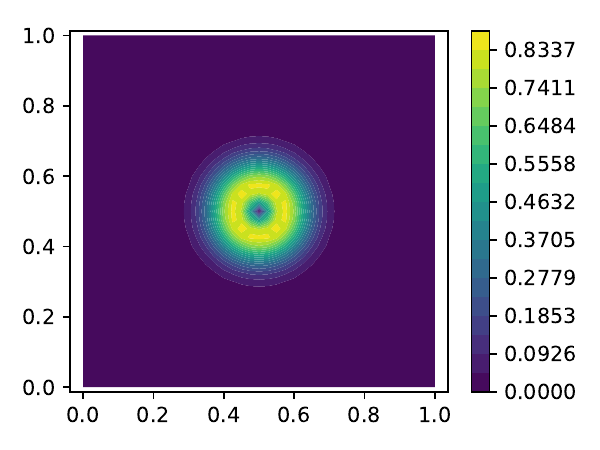}
\end{minipage}
	\begin{minipage}{0.29\textwidth}
	\centering exact $\lVert \mathbf v \rVert$\\
	\includegraphics[width=\textwidth,trim={0 0 70 0},clip]{figures/LinAc2D_coriolis_vortex_long/final_sol_unorm_exact_ord_2_N_0040.pdf}
\end{minipage}
\begin{minipage}{0.09\textwidth}
	\includegraphics[width=\textwidth,trim={220 20 0 0},clip]{figures/LinAc2D_coriolis_vortex_long/final_sol_unorm_exact_ord_2_N_0040.pdf}
\end{minipage}
	\centering $\mathbb Q^3, N_x=N_y=6$\\[2mm]
\begin{minipage}{0.29\textwidth}
	\centering SU standard $\lVert \mathbf v \rVert$\\
	\includegraphics[width=\textwidth,trim={0 0 70 0},clip]{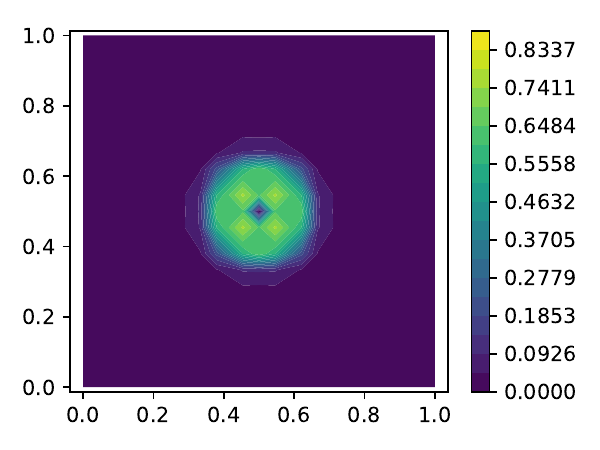}
\end{minipage}
\begin{minipage}{0.29\textwidth}
	\centering SU--GF $\lVert \mathbf v \rVert$\\
	\includegraphics[width=\textwidth,trim={0 0 70 0},clip]{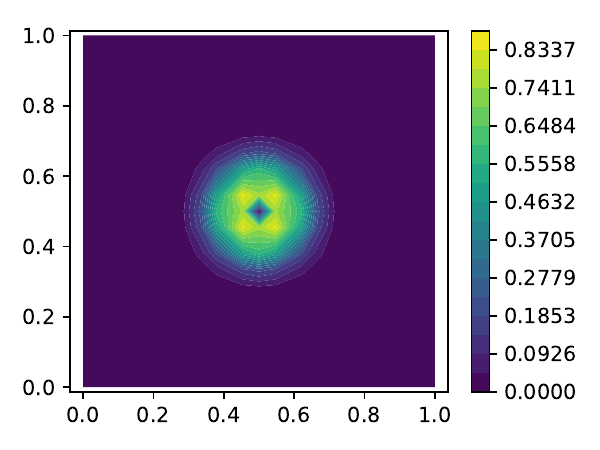}
\end{minipage}
\begin{minipage}{0.29\textwidth}
	\centering exact $\lVert \mathbf v \rVert$\\
	\includegraphics[width=\textwidth,trim={0 0 70 0},clip]{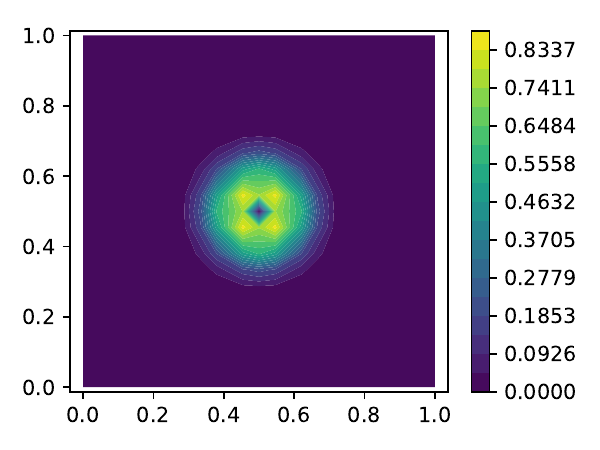}
\end{minipage}
\begin{minipage}{0.09\textwidth}
	\includegraphics[width=\textwidth,trim={220 20 0 0},clip]{figures/LinAc2D_coriolis_vortex_long/final_sol_unorm_exact_ord_4_N_0006.pdf}
\end{minipage}
\\	 \begin{minipage}{0.29\textwidth}
	\centering OSS standard $\lVert \mathbf v \rVert$\\
	\includegraphics[width=\textwidth,trim={0 0 70 0},clip]{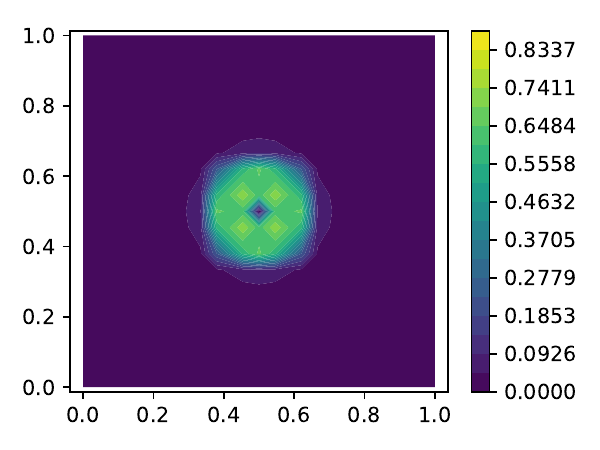}
\end{minipage}
\begin{minipage}{0.29\textwidth}
	\centering OSS--GF $\lVert \mathbf v \rVert$\\
	\includegraphics[width=\textwidth,trim={0 0 70 0},clip]{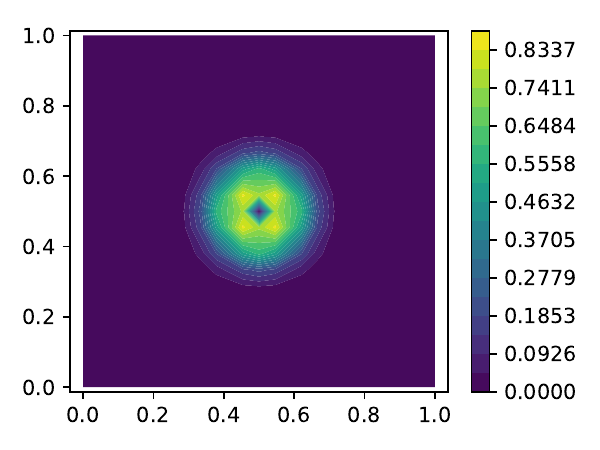}
\end{minipage}
\begin{minipage}{0.29\textwidth}
	\centering exact $\lVert \mathbf v \rVert$\\
	\includegraphics[width=\textwidth,trim={0 0 70 0},clip]{figures/LinAc2D_coriolis_vortex_long/final_sol_unorm_exact_ord_4_N_0006.pdf}
\end{minipage}
\begin{minipage}{0.09\textwidth}
	\includegraphics[width=\textwidth,trim={220 20 0 0},clip]{figures/LinAc2D_coriolis_vortex_long/final_sol_unorm_exact_ord_4_N_0006.pdf}
\end{minipage}
	\caption{Vortex with Coriolis force: simulations  at time $T=100$ with $\mathbb Q^1$ elements and $40\times 40$ cells (first and second row), and with $\mathbb Q^3$ with $6 \times 6$ cells (third and last row).}\label{fig:simulation_long_coriolis}
\end{figure}

To give a qualitative feeling of the difference in performance between the GF formulation and the standard one,  
in Figure~\ref{fig:simulation_long_coriolis}, we plot the velocity norm contours  at time $T=100$ for the standard scheme with SU stabilization and for the  multi-dimensional GF scheme, as well as for the OSS stabilization.
 We compare the   $\mathbb Q^1$  results on the  $N_x=N_y=40$  mesh (the second mesh in the grid convergence), and
the   $\mathbb Q^3$ results on the coarsest mesh with $N_x=N_y=6$ cells.  In the rightmost figure, we report for comparison  
the exact solution.
Clearly, even on these coarse meshes, the  results of the GF methods  are almost indistinguishable from the exact ones.
The standard schemes  break  the structure of the solution and dissipate it.
The longer the final time,
the stronger the effect of this dissipation.\\

\begin{figure}
	\centering
	\includegraphics[width=0.7\textwidth]{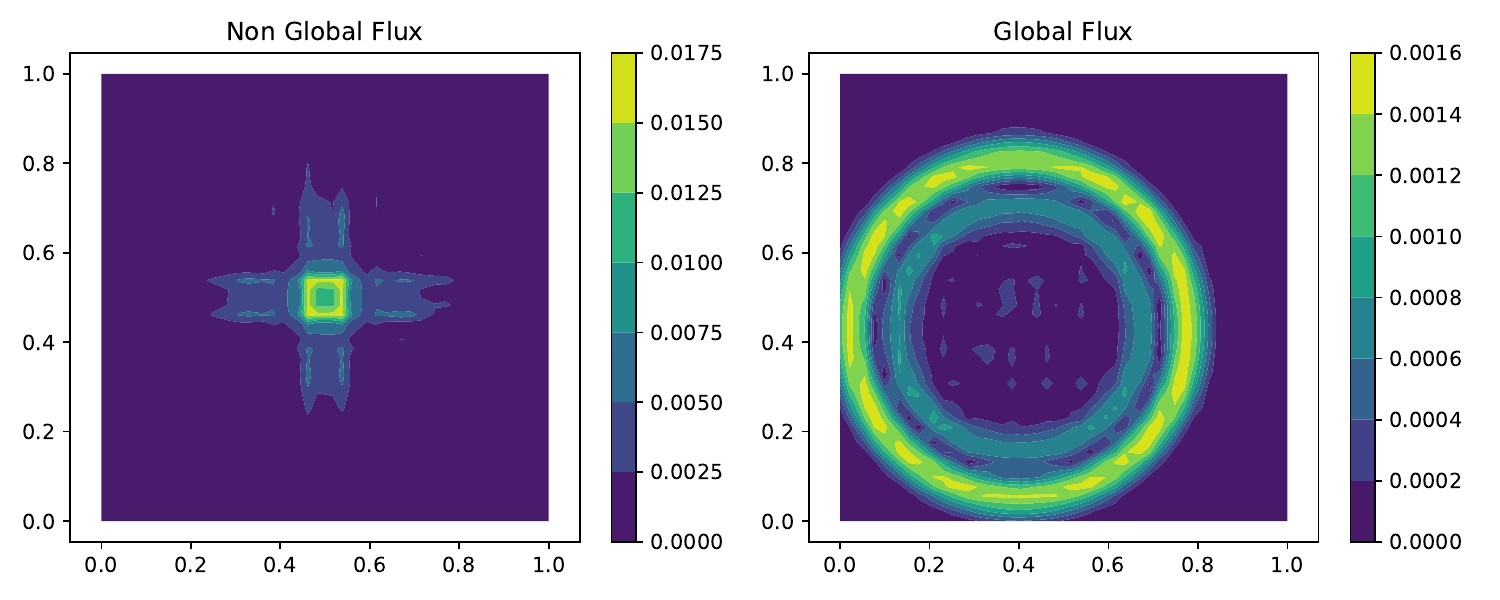}
	\includegraphics[width=0.7\textwidth]{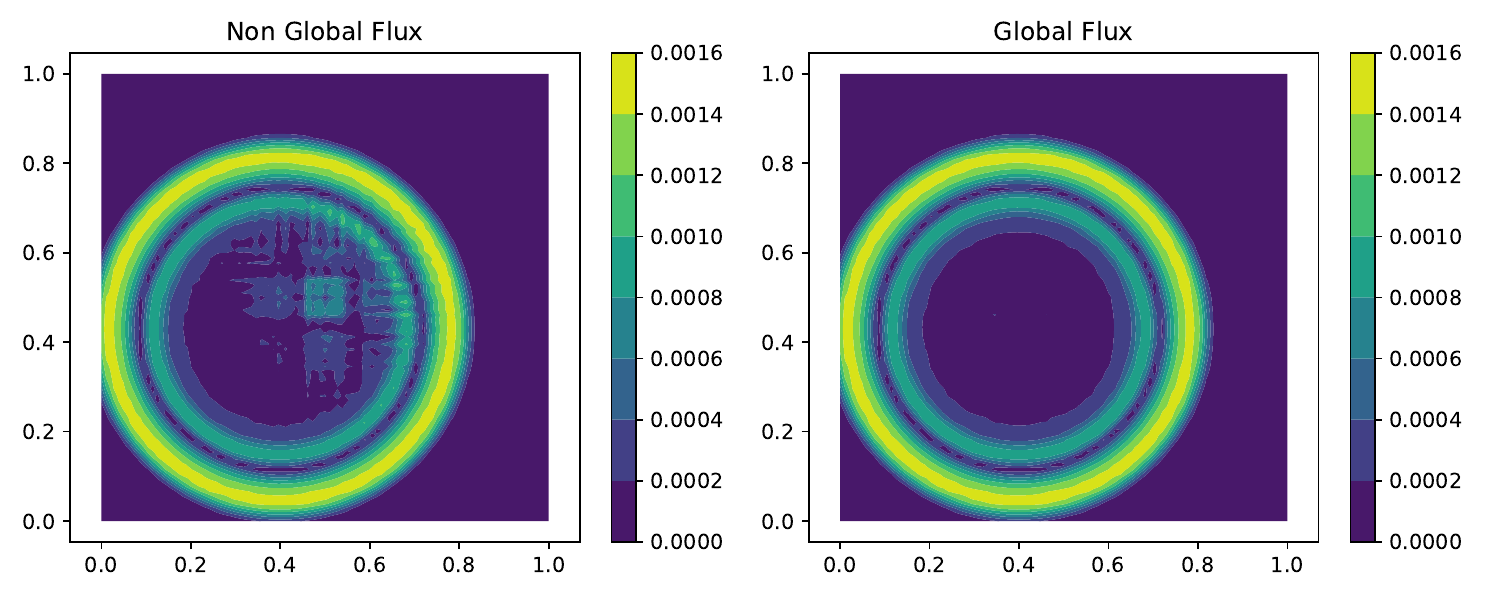}
	\caption{Vortex with Coriolis force: $\varepsilon=10^{-2}$ perturbation of the optimal equilibrium solution. Plot of $\lVert\vec{u}_{eq}-\vec{u}_p^{\text{SU}}\rVert$, with $\vec{u}_{eq}$ the optimized equilibrium \eqref{eq:coriolis_equilibrium}. 
	Top $\mathbb Q^3$ with 13 cells, bottom $\mathbb Q^3$ with 26 cells. All numerical results obtained using the SU method.}\label{fig:perturbation_coriolis_optimization}
\end{figure}
\begin{figure}
	\centering
	\includegraphics[width=0.7\textwidth]{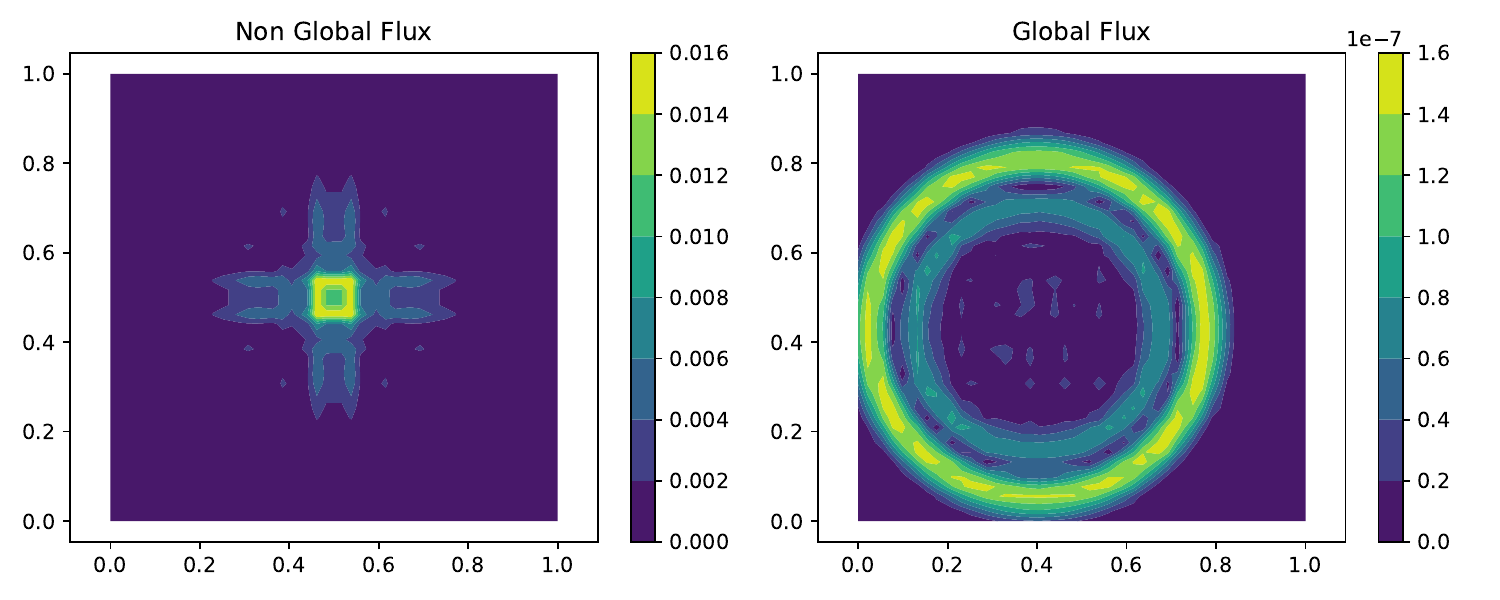}
	\includegraphics[width=0.7\textwidth]{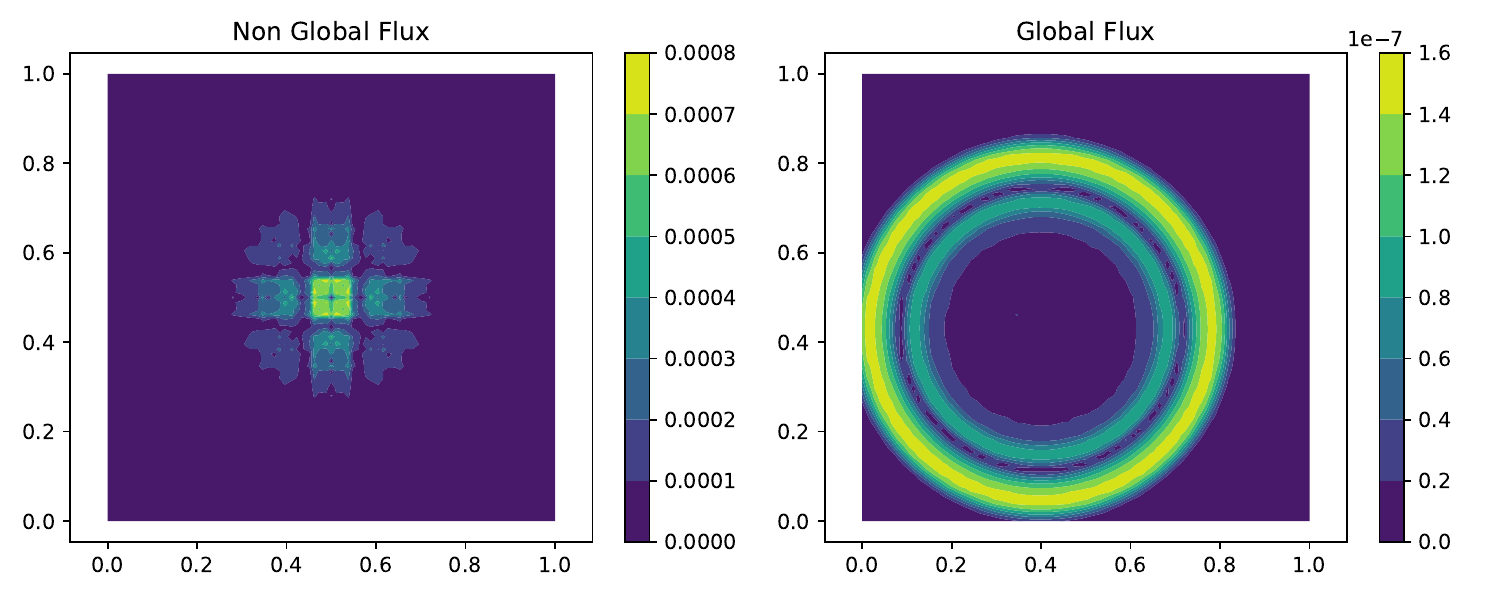}
	\caption{Vortex with Coriolis force: $\varepsilon=10^{-6}$ perturbation of the optimal equilibrium solution. Plot of $\lVert\vec{u}_{eq}-\vec{u}_p^{\text{SU}}\rVert$, with $\vec{u}_{eq}$ the optimized equilibrium \eqref{eq:coriolis_equilibrium}. 
	Top $\mathbb Q^3$ with 13 cells, bottom $\mathbb Q^3$ with 26 cells. All numerical results obtained using the SU method.}\label{fig:perturbation_coriolis_optimization_tiny}
\end{figure}

Finally, we perform a   perturbation analysis, which is classical in the study of stationarity preserving discretizations: The main reason for developing a stationarity preserving method is not to preserve just the steady state, but to ensure that physical perturbations of a steady state are not disturbed by spurious numerical perturbations originating from the steady state itself.
We perform this test on the vortex solution. We use the optimization process  described in section~\ref{sec:init_data} to get discrete well-prepared initial conditions,
and then add a local perturbation to the pressure of the form 
\begin{equation}\label{eq:pressure_perturbation}
	\delta_p({x}) = \varepsilon  e^{-\frac{1}{2(1-\rho({x})/r_0)^2} + \frac12},
\end{equation} 
where  $\vec{x}_p = (0.4,0.43)$, and $\rho({x}) = \|\vec{x} - \vec{x}_p\|$, with  $r_0=0.1$. 


Starting from this perturbed state, we evolve the solution until time   $T=0.35$.  We only discuss the results obtained with the SU methods, since those with OSS are  very similar.
In Figure~\ref{fig:perturbation_coriolis_optimization}, we test the perturbation with $\varepsilon=10^{-2}$, applied on top of the optimization--based well-prepared equilibrium solution for various schemes. We observe that SU-GF is able to accurately describe the motion of the perturbation with no spurious effects even on a very coarse $13\times 13$ mesh. The standard  SU scheme   on this mesh does not allow even to see the perturbation. Of course, mesh refinement reduces spurious artefacts and we can see a better result
on a $26\times 26$ discretization of the domain.

We then repeat the test with $\varepsilon=10^{-6}$. The results are reported on   Figure~\ref{fig:perturbation_coriolis_optimization_tiny}.
The  SU-GF  method allows again to obtain an excellent resolution already on the coarse  $13\times 13$ mesh.
For the standard method, in this case refining the mesh  to  $26\times 26$ is not sufficient to capture such a small feature.
Note that  these results are obtained with quadratic and cubic polynomial degrees.
This shows the advantage of combining high-order  schemes and stationarity preservation.
As one may expect, for very small values of the perturbation, the quality of the results still shows some dependence on the 
mesh size, and on the initialization strategy strategy. However, the
%
GF methods  show  systematically superior results.

\subsection{Mass source cases}\label{sec:mass_source}

We now study the case involving   a mass source as introduced in Section \ref{sec:mass_source_theory}. This is a very interesting case   where the truly multi-dimensional nature of the GF formulation is necessary
to balance the divergence with the source. 
We consider both the steady solution \eqref{eq:mass_steady} and the time dependent solution \eqref{eq:mass_moving}
as this allows to show the performance outside   equilibrium.

\subsubsection{Translating solution}

\begin{table}
	\centering
	\caption{Translating solution with mass source: $p$ and $u$ errors and convergence rates  with SU and OSS stabilization}\label{tab:translating_conv_all}
	\vspace{0.2cm}
	\foreach \n in {2,...,5}{
		\pgfmathtruncatemacro\result{\n-1}
		\centering 
		\begin{adjustbox}{max width=\textwidth}
			\begin{tabular}{|c||cc|cc||cc|cc|||cc|cc||cc|cc|}
				\hline
				&\multicolumn{4}{|c||}{SU $\mathbb Q^{\result}$}
				&\multicolumn{4}{|c|||}{SU-GF $\mathbb Q^{\result}$}
				&\multicolumn{4}{|c||}{OSS $\mathbb Q^{\result}$}
				&\multicolumn{4}{|c|}{OSS-GF $\mathbb Q^{\result}$}
				\\ \hline
				{\color{white}\tiny$N\!\!$}$N${\color{white}\tiny$\!\!N$} & err $u$ & err $p$ & O-$u$ & O-$p$& err $u$ & err $p$ & O-$u$ & O-$p$& err $u$ & err $p$ & O-$u$ & O-$p$& err $u$ & err $p$ & O-$u$ & O-$p$ \\
				\hline
				\input{figures/LinAc2D_moving_source/big_table_ord\n.tex}\\ \hline 
			\end{tabular}
		\end{adjustbox}\\[2pt]
	}
\end{table}

We start from  computing the convergence to the translating solution  \eqref{eq:mass_moving} to check that 
the GF methods maintain the classical order of accuracy in the non-stationary regime. 
We consider the particular case of \eqref{eq:mass_moving}  with   $h(x,y,t)\equiv 0$, 
with $p_0=1, \, b=0.001,\, \mathbf a =(-0.1,0.1)$, $g(x,y)=e^{-100((x-x_0)^2+(y-y_0)^2)}$, $x_0=0.65$, $y_0=0.39$. 
The final time of the simulation is  set to $T=0.1$. We compute  the solutions with the SU  and OSS stabilization  methods,
using both the GF and standard formulations.
%

The results for   polynomial degrees $K$ from 1 to 4 are reported in Table~\ref{tab:translating_conv_all}. To save space, only the $u$ velocity component is reported.
The convergence orders are  in between $K+\frac12$ and $K+1$  for the GF schemes, and in between  $K$ and $K+\frac12$  for the standard one.  
These confirm reasonably the expected $K+1$ accuracy of the methods.
Even if it is not a steady case, GF schemes  seem to perform  a little better, with   errors systematically below those of the standard method.

\subsubsection{Steady solution}
\begin{figure}
	\centering $\mathbb Q^1, N_x=N_y=20$\\[2mm]
	\begin{minipage}{0.29\textwidth}
		\centering SU $\lVert \mathbf v \rVert$\\
		\includegraphics[width=\textwidth,trim={0 0 70 0},clip]{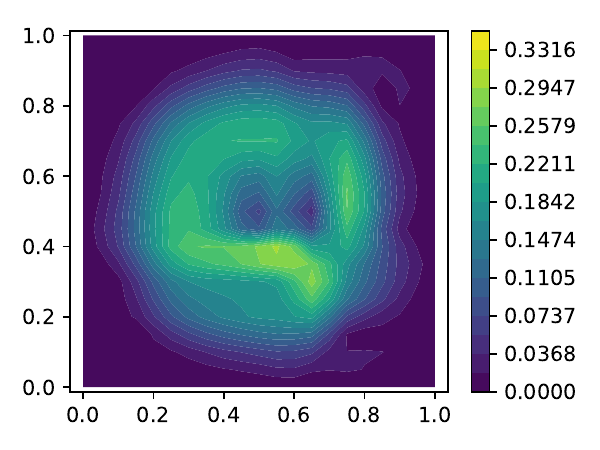}
	\end{minipage}
	\begin{minipage}{0.29\textwidth}
		\centering SU--GF $\lVert \mathbf v \rVert$\\
		\includegraphics[width=\textwidth,trim={0 0 70 0},clip]{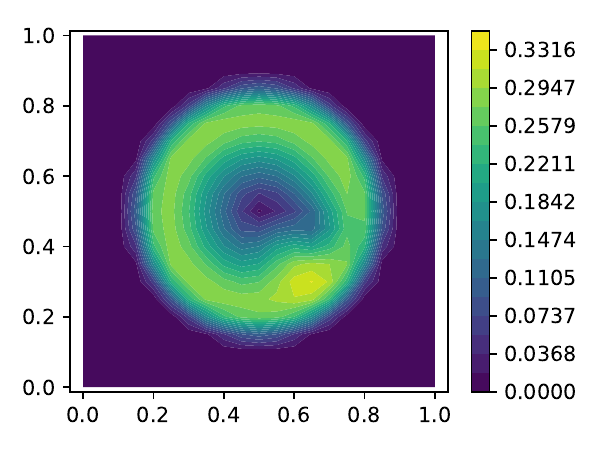}
	\end{minipage}
	\begin{minipage}{0.29\textwidth}
		\centering exact $\lVert \mathbf v \rVert$\\
		\includegraphics[width=\textwidth,trim={0 0 70 0},clip]{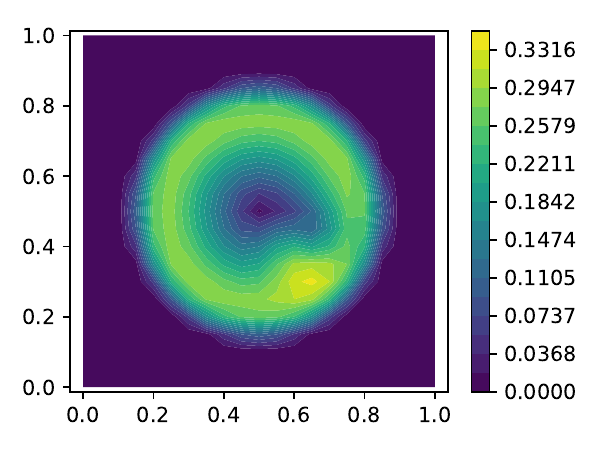}
	\end{minipage}
	\begin{minipage}{0.09\textwidth}
		\includegraphics[width=\textwidth,trim={220 20 0 0},clip]{figures/LinAc2D_source_vortex_long_dirichlet/final_sol_unorm_exact_ord_2_N_0020.pdf}
	\end{minipage}\\
	\begin{minipage}{0.29\textwidth}
		\centering OSS $\lVert \mathbf v \rVert$\\
		\includegraphics[width=\textwidth,trim={0 0 70 0},clip]{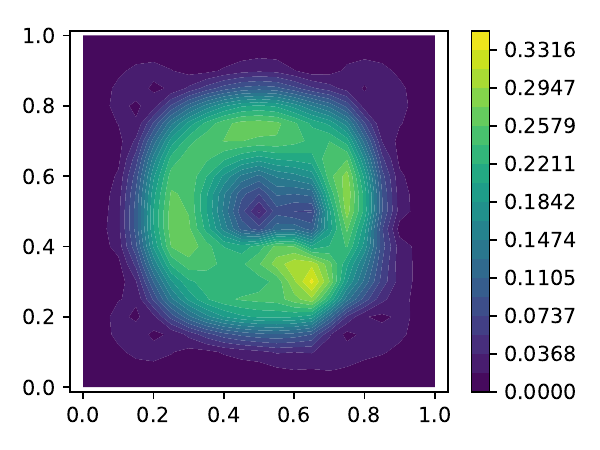}
	\end{minipage}
	\begin{minipage}{0.29\textwidth}
		\centering OSS--GF $\lVert \mathbf v \rVert$\\
		\includegraphics[width=\textwidth,trim={0 0 70 0},clip]{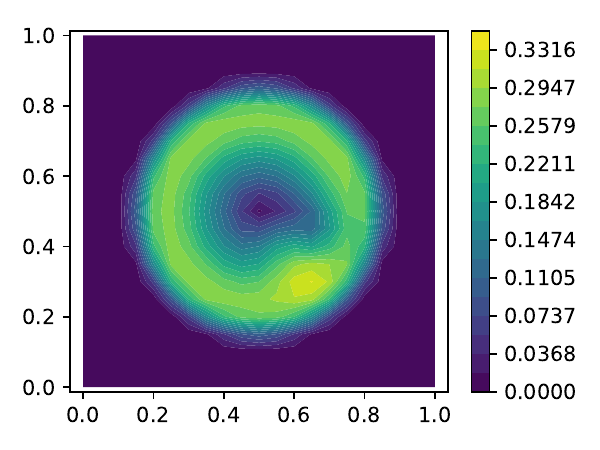}
	\end{minipage}
	\begin{minipage}{0.29\textwidth}
		\centering exact $\lVert \mathbf v \rVert$\\
		\includegraphics[width=\textwidth,trim={0 0 70 0},clip]{figures/LinAc2D_source_vortex_long_dirichlet/final_sol_unorm_exact_ord_2_N_0020.pdf}
	\end{minipage}
	\begin{minipage}{0.09\textwidth}
		\includegraphics[width=\textwidth,trim={220 20 0 0},clip]{figures/LinAc2D_source_vortex_long_dirichlet/final_sol_unorm_exact_ord_2_N_0020.pdf}
	\end{minipage}\\
	\centering $\mathbb Q^3, N_x=N_y=6$\\[2mm]
	\begin{minipage}{0.29\textwidth}
		\centering SU $\lVert \mathbf v \rVert$\\
		\includegraphics[width=\textwidth,trim={0 0 70 0},clip]{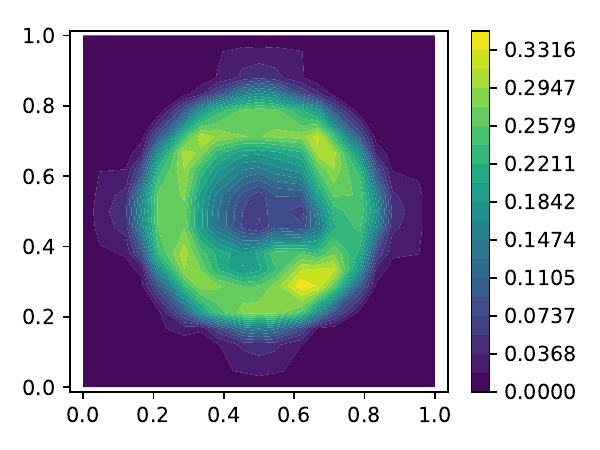}
	\end{minipage}
	\begin{minipage}{0.29\textwidth}
		\centering SU--GF $\lVert \mathbf v \rVert$\\
		\includegraphics[width=\textwidth,trim={0 0 70 0},clip]{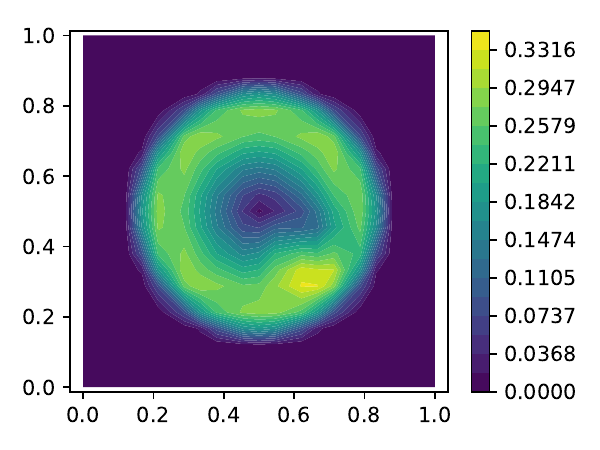}
	\end{minipage}
	\begin{minipage}{0.29\textwidth}
		\centering exact $\lVert \mathbf v \rVert$\\
		\includegraphics[width=\textwidth,trim={0 0 70 0},clip]{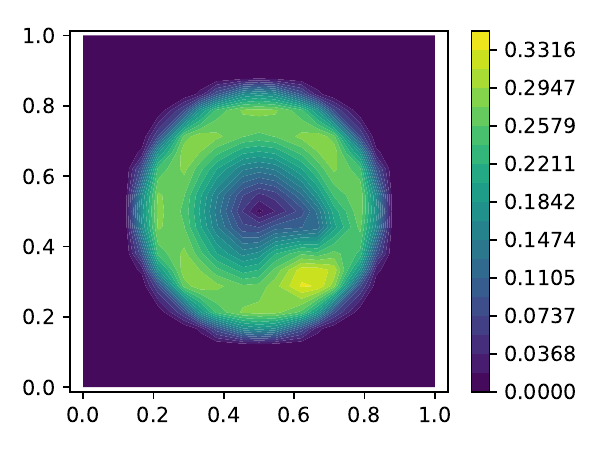}
	\end{minipage}
	\begin{minipage}{0.09\textwidth}
		\includegraphics[width=\textwidth,trim={220 20 0 0},clip]{figures/LinAc2D_source_vortex_long_dirichlet/final_sol_unorm_exact_ord_4_N_0006.pdf}
	\end{minipage}\\
	\begin{minipage}{0.29\textwidth}
		\centering OSS $\lVert \mathbf v \rVert$\\
		\includegraphics[width=\textwidth,trim={0 0 70 0},clip]{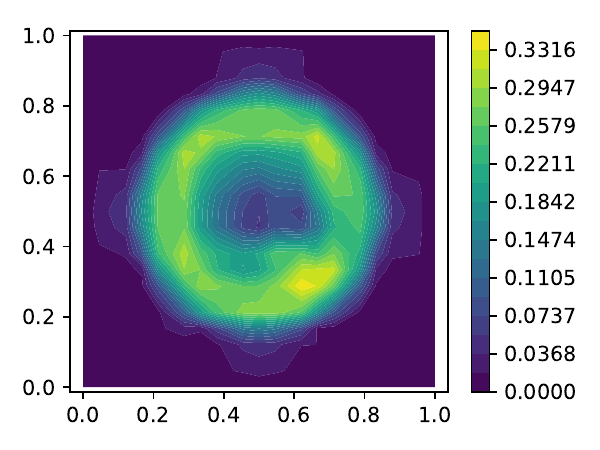}
	\end{minipage}
	\begin{minipage}{0.29\textwidth}
		\centering OSS--GF $\lVert \mathbf v \rVert$\\
		\includegraphics[width=\textwidth,trim={0 0 70 0},clip]{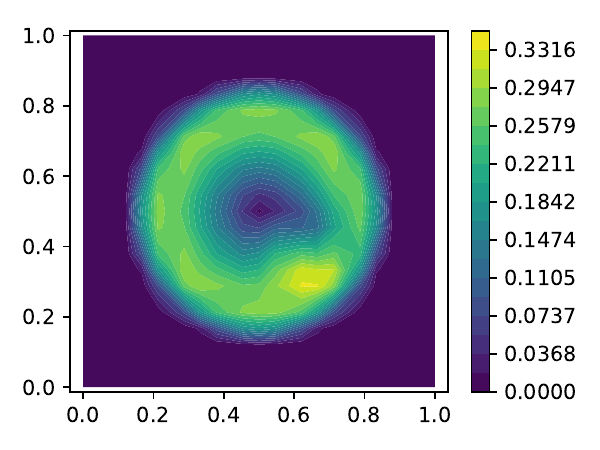}
	\end{minipage}
	\begin{minipage}{0.29\textwidth}
		\centering exact $\lVert \mathbf v \rVert$\\
		\includegraphics[width=\textwidth,trim={0 0 70 0},clip]{figures/LinAc2D_source_vortex_long_dirichlet/final_sol_unorm_exact_ord_4_N_0006.pdf}
	\end{minipage}
	\begin{minipage}{0.09\textwidth}
		\includegraphics[width=\textwidth,trim={220 20 0 0},clip]{figures/LinAc2D_source_vortex_long_dirichlet/final_sol_unorm_exact_ord_4_N_0006.pdf}
	\end{minipage}
	\caption{Vortex with mass source term: simulations at time $T=100$ with $\mathbb Q^1$ elements and $40\times 40$ cells (top) and with $\mathbb Q^3$ with $6 \times 6$ cells}\label{fig:simulation_long_source}
\end{figure}
We now consider \eqref{eq:mass_steady} with $h,g $ defined by 
$h(\rho)= 20 e^{-100 \rho^2}$, with $\rho$ the distance from the center of the domain,
and $g(x,y) := \frac{1}{100} e^{-100 \rho_1^2}$,  where  $\rho_1:=\|\vec{x} - \vec{x}_1\|$, with $\vec{x}_1 = (0.65, 0.39)^T$.
This solution defines a steady vortex with a scalar (mass/pressure) source added in $\vec{x}_1 $. 
We start by considering the solutions at time $T=100$ to see the qualitative improvement brought by the GF approach. 
As before, we plot the norm of the velocity, for both the SU and OSS stabilized methods in both standard and GF formulation.
Figure~\ref{fig:simulation_long_source} reports the contour plots for the $\mathbb Q^1$ case on the coarse $N_x=N_y=20$ mesh, and for the $\mathbb Q^3$ case on the $N_x=N_y=6$ one.
As before, already on these coarse meshes, the GF results are indistinguishable from the exact solution.
The standard schemes break the circular symmetric structure of the solution,  due to mismatch between the Galerkin
term and the numerical dissipation, which precludes the stationarity preservation property.

\begin{table}
	\centering
	\caption{Vortex with mass source: $p$ and $u$  errors and convergence rates   with SU and OSS stabilization}\label{tab:source_vortex_conv_all}
	\vspace{0.2cm}
	
	\foreach \n in {2,...,6}{
		\pgfmathtruncatemacro\result{\n-1}
		\centering 
		\begin{adjustbox}{max width=\textwidth}
			\begin{tabular}{|c||cc|cc||cc|cc|||cc|cc||cc|cc|}
				\hline
				&\multicolumn{4}{|c||}{SU $\mathbb Q^{\result}$}
				&\multicolumn{4}{|c|||}{SU-GF $\mathbb Q^{\result}$}
				&\multicolumn{4}{|c||}{OSS $\mathbb Q^{\result}$}
				&\multicolumn{4}{|c|}{OSS-GF $\mathbb Q^{\result}$}
				\\ \hline
				{\color{white}\tiny$N\!\!$}$N${\color{white}\tiny$\!\!N$} & err $u$ & err $p$ & O-$u$ & O-$p$& err $u$ & err $p$ & O-$u$ & O-$p$& err $u$ & err $p$ & O-$u$ & O-$p$& err $u$ & err $p$ & O-$u$ & O-$p$ \\
				\hline
				\input{figures/LinAc2D_source_vortex_dirichlet/big_table_ord\n.tex}\\ \hline 
			\end{tabular}
		\end{adjustbox}\\[2pt]
	}
\end{table}

Next, we consider a study of grid convergence. As before, we initialize the solution by  sampling the exact one at nodes, and then run
the schemes   until a final time  $T=1$.  Tables \ref{tab:source_vortex_conv_all} 
reports the errors and the convergence rates  for the SU and OSS stabilization methods, for polynomial degrees from 1 to 5. The error is measured in this case by comparing, at final time, the value
of the nodal approximation of the steady state residual. 
To save space, only the $u$ velocity component is reported.
We can see that the theoretical  $K+2$ super-convergence for the GF method is confirmed also for this case. As for the other cases,
the GF formulation gives systematically lower errors, starting with the coarsest meshes where the errors are about half. As the mesh
is refined, this error gain goes up to a factor 10 or 50 due to the super-convergence of the GF schemes. 

  For the GF case, we report  on figure \ref{fig:convergence_div_source} 
  the norm of $D_x\otimes D_yI_y \mathrm u + D_xI_x \otimes D_y \mathrm v -D_xI_x \otimes D_y I_y \mathrm S_p $,
   while for the standard case we plot the norm of $D_x\otimes M_y \mathrm u + M_x \otimes D_y \mathrm v -M_x \otimes M_y \mathrm S_p $.
   Boundaries are excluded in both computations. 
   We can observe that the standard residual converges with order $K$ at most, while the GF converges with order $K+1$ for all $K\geq 2$, 
   as a consequence of its super-convergence properties.  \\

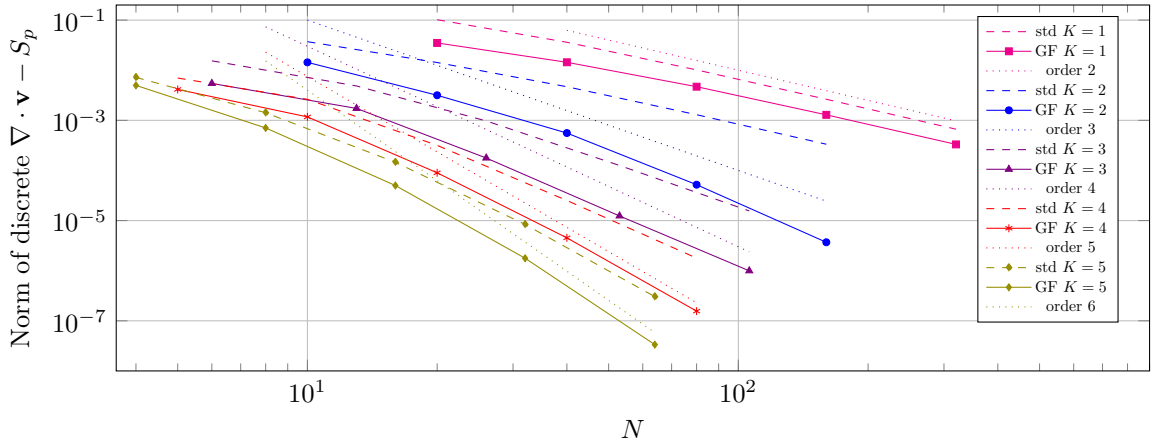
\begin{figure} 
	\centering
		\begin{tikzpicture}
			\begin{axis}[
				xmode=log, ymode=log,
				xmin=3.6,xmax=900,
				ymin=1.e-8,ymax=0.2,
				grid=major,
				xlabel={$N$},
				ylabel={Norm of discrete $\nabla \cdot \mathbf v -S_p$},
				xlabel shift = 1 pt,
				ylabel shift = 1 pt,
				legend pos= north east,
				legend style={nodes={scale=0.6, transform shape}},
				width=.95\textwidth,
				height=.4\textwidth
				]
				
				\addplot[dashed, magenta]             table [y=err_nbr, x=N, col sep=comma]{figures/LinAc2D_source_vortex_dirichlet/divergenceSimple_discretization_SUPG_error_ord2.csv};
				\addlegendentry{std $K=1$}
				\addplot[mark=square*,mark size=1.3pt,mark options={solid},magenta]  table [y=err_nbr, x=N, col sep=comma]{figures/LinAc2D_source_vortex_dirichlet/divergenceGF_discretization_SUPG_error_ord2.csv};
				\addlegendentry{GF $K=1$}
				\addplot[magenta,dotted,domain=40:320]{100./x/x};
				\addlegendentry{order 2}		
				
				\addplot[dashed, blue]             table [y=err_nbr, x=N, col sep=comma]{figures/LinAc2D_source_vortex_dirichlet/divergenceSimple_discretization_SUPG_error_ord3.csv};
				\addlegendentry{std $K=2$}
				\addplot[mark=otimes*,mark size=1.3pt,mark options={solid},blue]  table [y=err_nbr, x=N, col sep=comma]{figures/LinAc2D_source_vortex_dirichlet/divergenceGF_discretization_SUPG_error_ord3.csv};
				\addlegendentry{GF $K=2$}
				\addplot[blue,dotted,domain=10:160]{100./x/x/x};
				\addlegendentry{order 3}		
				
				\addplot[dashed, violet]             table [y=err_nbr, x=N, col sep=comma]{figures/LinAc2D_source_vortex_dirichlet/divergenceSimple_discretization_SUPG_error_ord4.csv};
				\addlegendentry{std $K=3$}
				\addplot[mark=triangle*,mark size=1.5pt,mark options={solid},violet]  table [y=err_nbr, x=N, col sep=comma]{figures/LinAc2D_source_vortex_dirichlet/divergenceGF_discretization_SUPG_error_ord4.csv};
				\addlegendentry{GF $K=3$}
				\addplot[violet,dotted,domain=8:106]{300./x/x/x/x};
				\addlegendentry{order 4}
				
				\addplot[dashed,red]          table [y=err_nbr, x=N, col sep=comma]{figures/LinAc2D_source_vortex_dirichlet/divergenceSimple_discretization_SUPG_error_ord5.csv};
				\addlegendentry{std $K=4$}
				\addplot[mark=asterisk,mark size=1.5pt,mark options={solid},red]  table [y=err_nbr, x=N, col sep=comma]{figures/LinAc2D_source_vortex_dirichlet/divergenceGF_discretization_SUPG_error_ord5.csv};
				\addlegendentry{GF $K=4$}
				\addplot[red,dotted,domain=8:80]{750./x/x/x/x/x};
				\addlegendentry{order 5}	
				
				\addplot[mark=diamond*,mark size=1.5pt,dashed,olive]             table [y=err_nbr, x=N, col sep=comma]{figures/LinAc2D_source_vortex_dirichlet/divergenceSimple_discretization_SUPG_error_ord6.csv};
				\addlegendentry{std $K=5$}
				\addplot[mark=diamond*,mark size=1.3pt,olive]  table [y=err_nbr, x=N, col sep=comma]{figures/LinAc2D_source_vortex_dirichlet/divergenceGF_discretization_SUPG_error_ord6.csv};
				\addlegendentry{GF $K=5$}
				\addplot[olive,dotted,domain=8:64]{4000./x/x/x/x/x/x};
				\addlegendentry{order 6}		
			\end{axis}
	\end{tikzpicture}
	\caption{Steady vortex with mass source: convergence with respect to the number of elements in $x$  of $L^2$ norm of $\nabla\cdot\vec v-S_p$ of the analytical solution discretized in standard (std) way or with the GF technique.}\label{fig:convergence_div_source}
\end{figure}

\begin{figure}
	\centering
	\includegraphics[width=0.325\textwidth,trim={0 0 367 0}, clip]{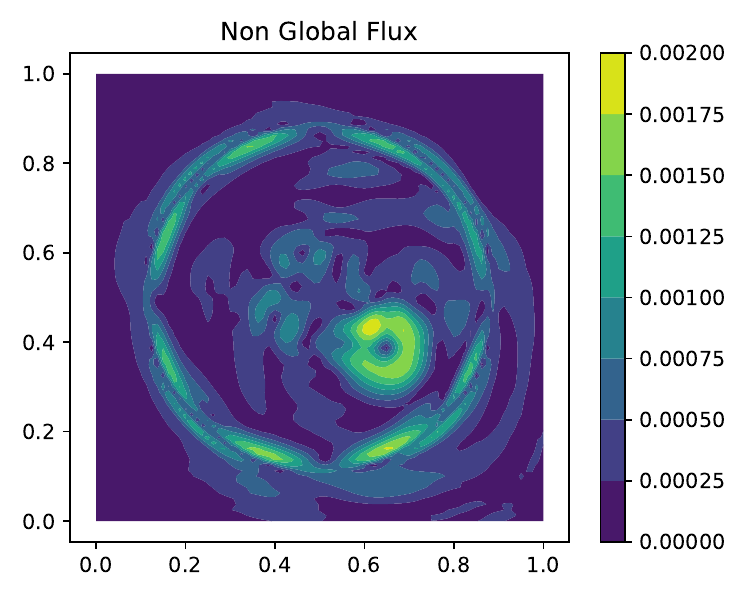}
	\includegraphics[width=0.325\textwidth,trim={0 0 367 0}, clip]{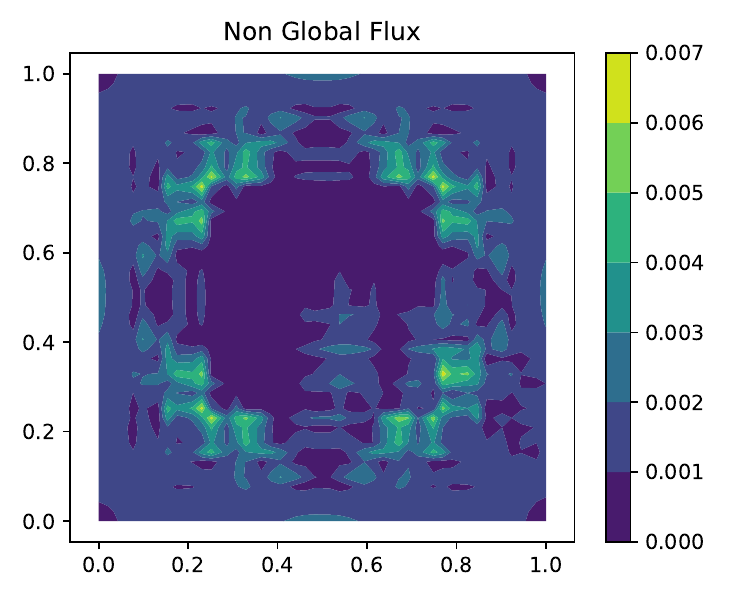}
	\includegraphics[width=0.325\textwidth,trim={0 0 367 0}, clip]{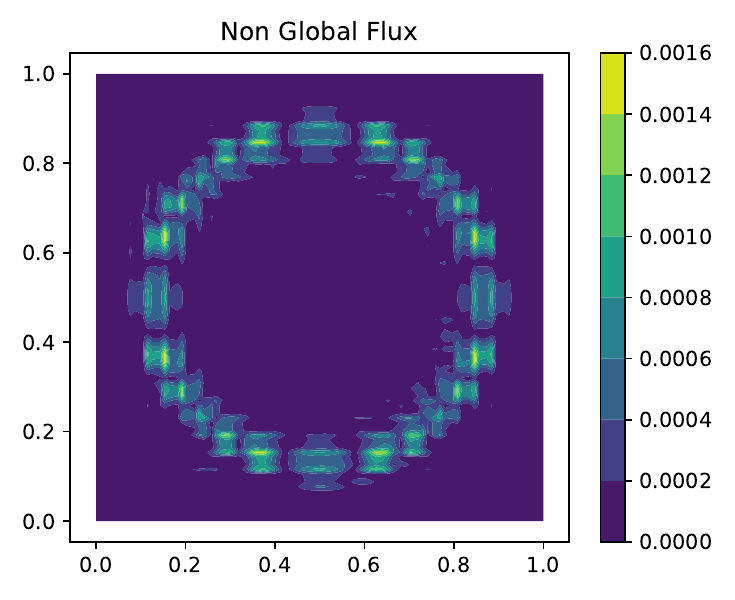}\\
	\includegraphics[width=0.325\textwidth,trim={352 0 15 0}, clip]{figures/LinAc2D_source_vortex_opt_perturbation/diff_vel_norm_ord2_N0080_pert_1.000e-03.pdf}
	\includegraphics[width=0.325\textwidth,trim={352 0 15 0}, clip]{figures/LinAc2D_source_vortex_opt_perturbation/diff_vel_norm_ord4_N0013_pert_1.000e-03.pdf}
	\includegraphics[width=0.325\textwidth,trim={352 0 15 0}, clip]{figures/LinAc2D_source_vortex_opt_perturbation/diff_vel_norm_ord4_N0026_pert_1.000e-03.pdf}
	\caption{Vortex with mass source term:  $\epsilon=10^{-3}$ pressure perturbation of the  optimal equilibrium solution, see Section~\ref{sec:init_data}. 
	  Plot of the $\|    \vec v_{eq} - \vec v_p^{\text{OSS}}   \|$, with $\vec{v}_{eq}$  the   equilibrium velocity. Left $\mathbb Q^1$ with $80 \times 80$ cells, center $\mathbb Q^3$ with 13 cells, right $\mathbb Q^3$ with 26 cells. Numerical solutions obtained with the OSS (top) and OSS-GF (bottom) methods.}\label{fig:perturbation_optimal_source}
\end{figure}

Finally, we consider the evolution of a small perturbation of the steady state. 
Equilibrium initial data is prepared using   the optimization process of Section~\ref{sec:init_data}.
A pressure perturbation of the form  \eqref{eq:pressure_perturbation} with $\varepsilon=10^{-3}$  is then added.
In this case, we report in  Figure~\ref{fig:perturbation_optimal_source} the results of the OSS method; the ones obtained with the SU schemes are very similar.
The plots show the OSS solutions  
 for $K=1$  on a relatively fine $80 \times 80$ mesh,
and   for $K=3$  on coarse  $13 \times 13$ and  $26 \times 26$ meshes.
The OSS-GF method shows huge superiority with respect to the standard formulation both in the high and in the low order case.
In the standard case, in particular, we can clearly see that  the source term is perturbing the solution in the right bottom quadrant of the domain, while other 
errors spread   around the exact solution profile. 
In this case, even refining the mesh to $N_x=N_y=26$ does not allow to capture the perturbation with the standard scheme.

\subsection{Stommel Gyre test case}\label{sec:stommel_gyre_num}

\begin{figure}
	\begin{minipage}{.5\textwidth}
		\centering
		\includegraphics[width=0.85\textwidth]{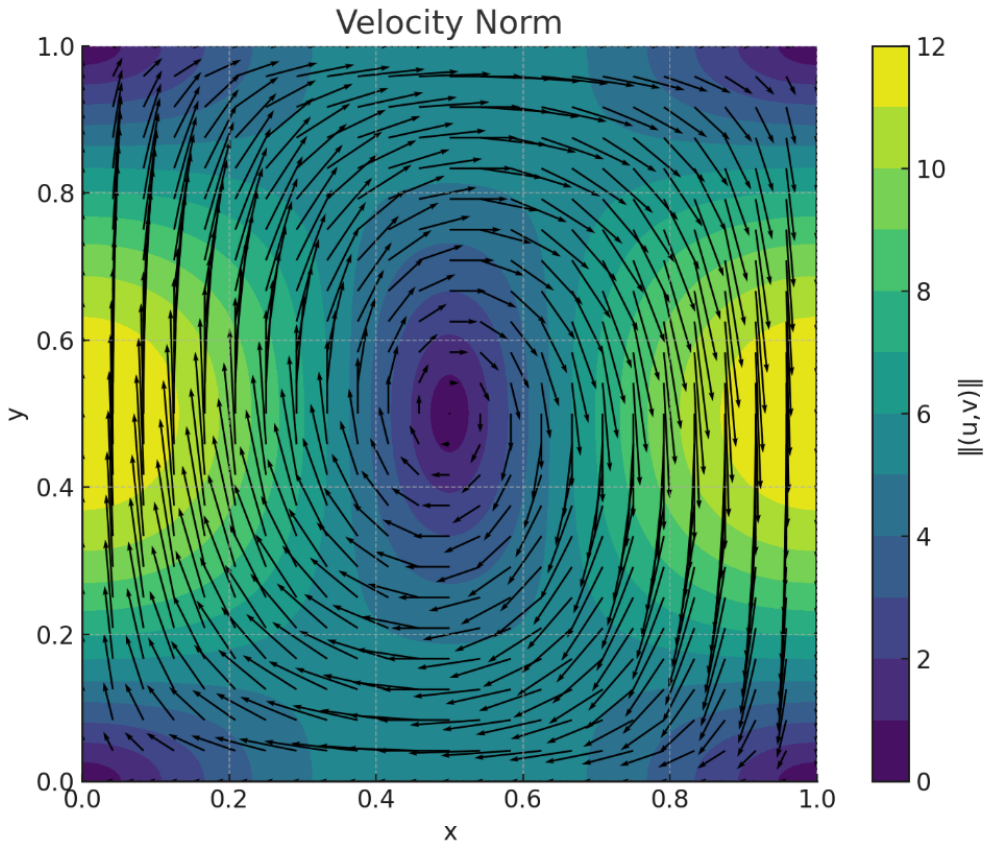}
	\end{minipage}\hfill
	\begin{minipage}{.5\textwidth}
		\centering
		\includegraphics[trim={720 0 0 20},clip,width=0.85\textwidth]{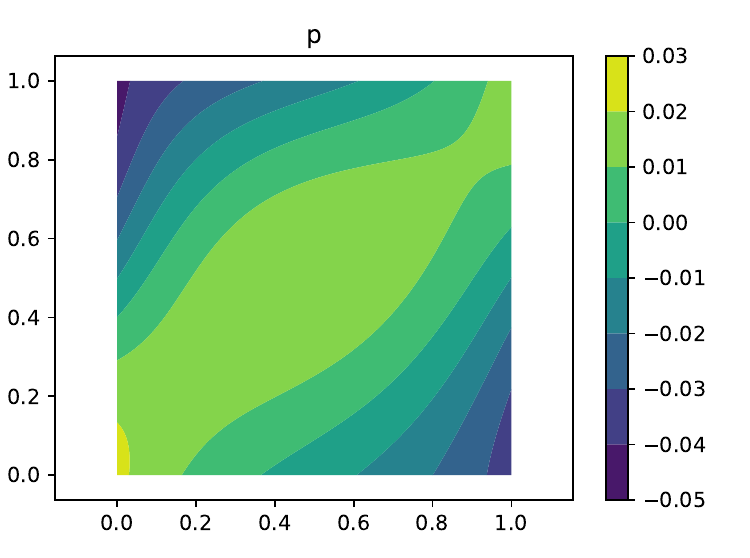}
	\end{minipage}
	\vspace{-2.5cm} 
	\caption{Stommel Gyre solution for $\lambda = b=1$, $f_0=\phi_0=R=0.01$, $F=0.1$, left velocity norm contour and velocity vector field, right pressure contour levels}\label{fig:simulation_SG_P6}
\end{figure}
We now consider the Stommel Gyre  test case described in Section \ref{sec:stommel_gyre}.
This  problem involves several  sources  of different forms.
The contours of   the  exact velocity  and pressure are reported for completeness  in Figure~\ref{fig:simulation_SG_P6}.
The precise test description is provided in Section~\ref{sec:stommel_gyre}.
In the simulations, we have set the solution parameters to  $\lambda = b=1$ (domain), $c_1=c_0=f=0.01$ (friction), $F=0.1$ (wind force).
Note that this is a non-compact case,  for which the Dirichlet  conditions allow a reasonable prediction already with the non-stationary preserving methods.

%

\begin{table}[htbp]
	\centering
	\caption{Stommel Gyre test convergence results with SU stabilization for $u$, $v$ and $p$}\label{tab:SG_conv}
	\vspace{0.1cm}
	
	\foreach \n in {2,...,6}{
		\pgfmathtruncatemacro\result{\n-1}
  		 \centering 
		\begin{adjustbox}{max width=0.45\textwidth}
			\begin{tabular}{|c|ccc|ccc|}
				\hline
				\multicolumn{7}{|c|}{SU $\mathbb Q^{\result} $}\\
				\hline
				{\tiny\color{white}N}$N${\tiny\color{white}N} & err $u$ & err $v$ & err $p$ & ord $u$ & ord $v$ & ord $p$ \\
				\hline
				\input{figures/LinAc2D_SG/errors_SUPG_ord\n.tex}\\ \hline
			\end{tabular}
		\end{adjustbox}
		\begin{adjustbox}{max width=0.45\textwidth}
			\begin{tabular}{|c|ccc|ccc|}
				\hline
				\multicolumn{7}{|c|}{SU-GF $\mathbb Q^{\result} $}\\
				\hline
				{\tiny\color{white}N}$N${\tiny\color{white}N} & err $u$ & err $v$ & err $p$ & ord $u$ & ord $v$ & ord $p$ \\
				\hline
				\input{figures/LinAc2D_SG/errors_SUPG_GF_ord\n.tex}\\ \hline
			\end{tabular}
		\end{adjustbox}
									~\\[2pt]
	}
\end{table}

We start by  studying the error with respect to exact solution for various mesh sizes. As done in the previous cases,
we start  from the   exact nodal values, and  run simulations until the final time   $T=1$. Table~\ref{tab:SG_conv} shows the 
errors and  convergence orders 
obtained with the SU stabilized schemes with polynomial degrees $K$ from 1 to 5.
Once again, we observe the super-convergence predicted by the analysis of Section~\ref{sec:consistency_analysis}, with dramatic error reductions
for $K \ge 3$.  The super-convergence is also visualized on Figure~\ref{fig:convergence_div_SG} in terms of  convergence of the divergence operator. As before,
we compute the norm of   $D_x\otimes D_yI_y \mathrm u + D_xI_x \otimes D_y \mathrm v$ for the GF method,
and of $D_x\otimes M_y \mathrm u + M_x \otimes D_y \mathrm v$ for 
the standard case. Domain boundaries are excluded in both computations. 
We find
$K+1$ accurate approximations of the divergence for polynomials of degree $K$, as well as reductions in the error
on this operator up to several orders of magnitude for $K\ge 2$.\\

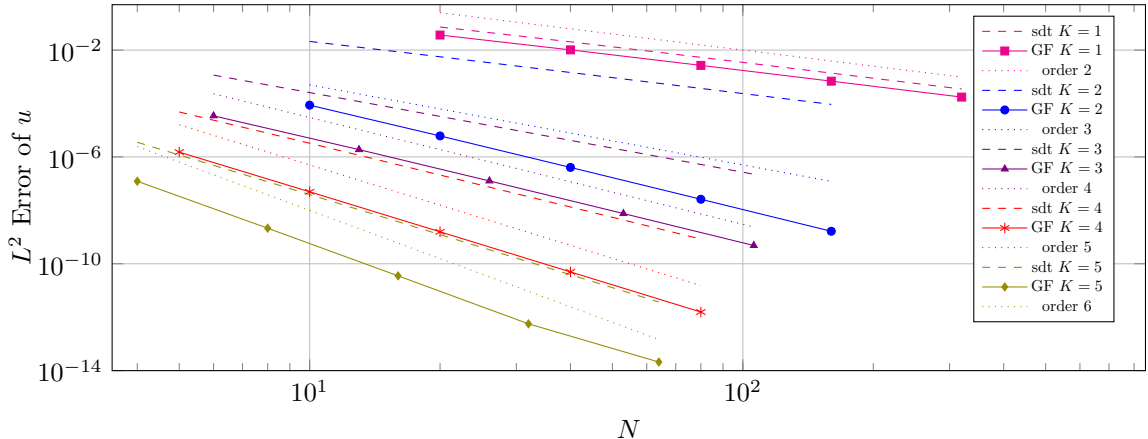
\begin{figure} 
	\centering
		\begin{tikzpicture}
			\begin{axis}[
				xmode=log, ymode=log,
				xmin=3.5,xmax=850,
				ymin=1.e-14,ymax=0.5,
				grid=major,
				xlabel={$N$},
				ylabel={$L^2$ Error of $u$},
				xlabel shift = 1 pt,
				ylabel shift = 1 pt,
				legend pos= north east,
				legend style={nodes={scale=0.6, transform shape}},
				width=.95\textwidth,
				height=.4\textwidth
				]
				
				\addplot[dashed, magenta]             table [y=err_nbr, x=N, col sep=comma]{figures/LinAc2D_SG/divergenceSimple_discretization_SUPG_error_ord2.csv};
				\addlegendentry{sdt $K=1$}
				\addplot[mark=square*,mark size=1.5pt,mark options={solid},magenta]  table [y=err_nbr, x=N, col sep=comma]{figures/LinAc2D_SG/divergenceGF_discretization_SUPG_error_ord2.csv};
				\addlegendentry{GF $K=1$}
				\addplot[magenta,dotted,domain=20:320]{100./x/x};
				\addlegendentry{order 2}		
				
				\addplot[dashed, blue]             table [y=err_nbr, x=N, col sep=comma]{figures/LinAc2D_SG/divergenceSimple_discretization_SUPG_error_ord3.csv};
				\addlegendentry{sdt $K=2$}
				\addplot[mark=otimes*,mark size=1.5pt,mark options={solid},blue]  table [y=err_nbr, x=N, col sep=comma]{figures/LinAc2D_SG/divergenceGF_discretization_SUPG_error_ord3.csv};
				\addlegendentry{GF $K=2$}
				\addplot[blue,dotted,domain=10:160]{0.5/x/x/x};
				\addlegendentry{order 3}		
				
				\addplot[dashed, violet]             table [y=err_nbr, x=N, col sep=comma]{figures/LinAc2D_SG/divergenceSimple_discretization_SUPG_error_ord4.csv};
				\addlegendentry{sdt $K=3$}
				\addplot[mark=triangle*,mark size=1.5pt,mark options={solid},violet]  table [y=err_nbr, x=N, col sep=comma]{figures/LinAc2D_SG/divergenceGF_discretization_SUPG_error_ord4.csv};
				\addlegendentry{GF $K=3$}
				\addplot[violet,dotted,domain=6:106]{0.3/x/x/x/x};
				\addlegendentry{order 4}			
				
				\addplot[dashed, red]             table [y=err_nbr, x=N, col sep=comma]{figures/LinAc2D_SG/divergenceSimple_discretization_SUPG_error_ord5.csv};
				\addlegendentry{sdt $K=4$}
				\addplot[mark=asterisk,mark size=2pt,mark options={solid},red]  table [y=err_nbr, x=N, col sep=comma]{figures/LinAc2D_SG/divergenceGF_discretization_SUPG_error_ord5.csv};
				\addlegendentry{GF $K=4$}
				\addplot[red,dotted,domain=5:80]{0.05/x/x/x/x/x};
				\addlegendentry{order 5}	
				
				\addplot[dashed,olive]             table [y=err_nbr, x=N, col sep=comma]{figures/LinAc2D_SG/divergenceSimple_discretization_SUPG_error_ord6.csv};
				\addlegendentry{sdt $K=5$}
				\addplot[mark=diamond*,mark size=1.5pt,mark options={solid},olive]  table [y=err_nbr, x=N, col sep=comma]{figures/LinAc2D_SG/divergenceGF_discretization_SUPG_error_ord6.csv};
				\addlegendentry{GF $K=5$}
				\addplot[olive,dotted,domain=4:64]{0.01/x/x/x/x/x/x};
				\addlegendentry{order 6}				
%
			\end{axis}
	\end{tikzpicture}
	\caption{Stommel-Gyre: convergence of $L^2$ error in $\nabla\cdot\vec u$ with respect to the number of elements in $x$ .}\label{fig:convergence_div_SG}
\end{figure}

\begin{figure}
	\centering
	
	\includegraphics[width=0.325\textwidth,trim={0 0 367 0}, clip]{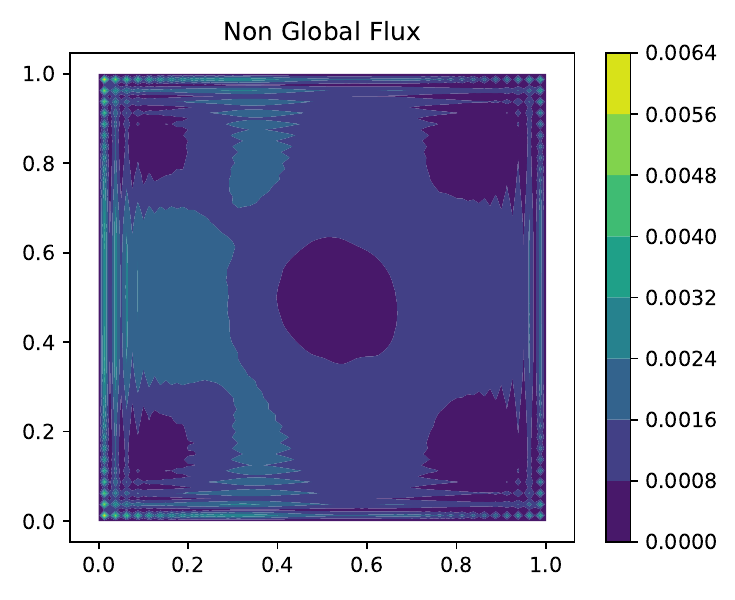}
	\includegraphics[width=0.325\textwidth,trim={0 0 367 0}, clip]{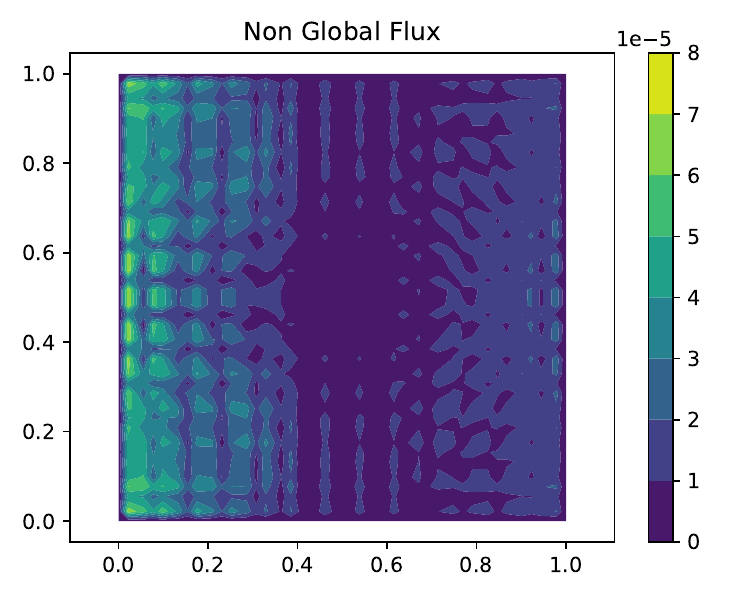}
	\includegraphics[width=0.325\textwidth,trim={0 0 367 0}, clip]{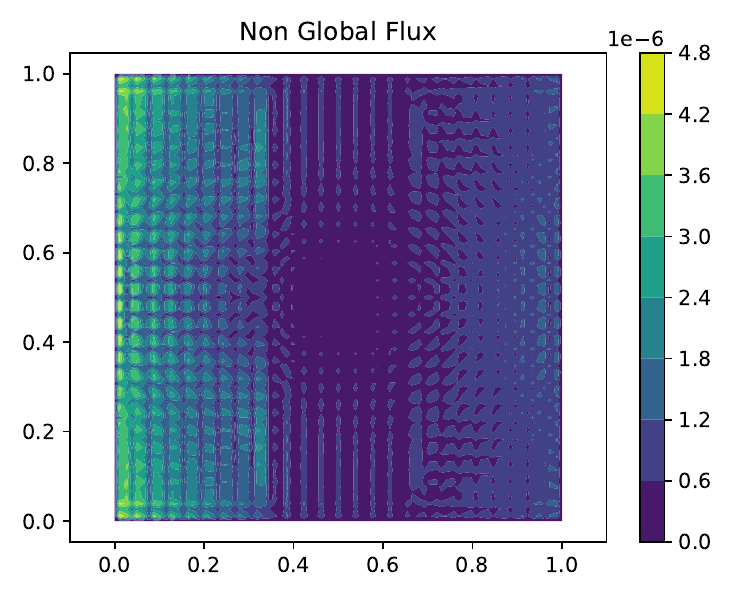}\\
	\includegraphics[width=0.325\textwidth,trim={352 0 15 0}, clip]{figures/LinAc2D_SG_num_perturbation/diff_vel_norm_ord2_N0080_pert_1.000e-03}
	\includegraphics[width=0.325\textwidth,trim={352 0 15 0}, clip]{figures/LinAc2D_SG_num_perturbation/diff_vel_norm_ord4_N0013_pert_1.000e-05}
	\includegraphics[width=0.325\textwidth,trim={352 0 15 0}, clip]{figures/LinAc2D_SG_num_perturbation/diff_vel_norm_ord4_N0026_pert_1.000e-08}
	\caption{Perturbation test for SG: long-time equilibrium solution. Plot of $\lVert\vec{u}_{eq}-\vec{u}_p\rVert$, with $\vec{u}_{eq}$ the analytical equilibrium. Left $\mathbb Q^1$ with $80 \times 80$ cells $\varepsilon=10^{-3}$, center $\mathbb Q^3$ with 13 cells $\varepsilon=10^{-5}$, right $\mathbb Q^3$ with 26 cells $\varepsilon=10^{-8}$}\label{fig:perturbation_numerical_SG}
\end{figure}

Finally, we run a perturbation test.  To this end, for each method we run a long time simulation up to $T=100$. We consider the state obtained to be 
 a reasonable discrete equilibrium.  
Then, we add to each of these discrete states the perturbation \eqref{eq:pressure_perturbation}, for      values of $\varepsilon$ depending on the 
underlying accuracy: $\varepsilon=10^{-3}$ for   $\mathbb Q^1$; $\varepsilon=10^{-5}$ for   $\mathbb Q^2$;  $\varepsilon=10^{-8}$ for   $\mathbb Q^3$.  
We evolve these perturbations  up to $T=0.35$.
In Figure~\ref{fig:perturbation_numerical_SG}, we see again that the perturbations, even very small ones, are well represented within few cells for the SU-GF scheme, 
while the standard one produces spurious waves resulting from  the interaction with the source terms. 
As in all other cases, also for this more physical and complex test the GF schemes outperform the standard ones.

\section{Conclusions and outlook}\label{sec:conclusions}

In this work, we have generalized the steady state preserving Finite Element approach based on   multi-dimensional  GF quadrature  of \cite{brt25} by 
including 
source terms of quite general forms.
The relevant equilibria here are not necessarily described by  a purely solenoidal constraint and a constant pressure.
More general non-trivial states are determined by the balance between the derivatives of the unknowns in different directions and the source term. 
With   GF quadrature   the equations are rewritten in such a way that all these terms are treated simultaneously,
and   stationarity preserving stabilized Finite Element methods can be derived easily.
We have characterized the discrete steady states of these methods,
with a particular focus on the role of boundary conditions at steady state. We have shown applications with schemes up to order 6 to genuinely multi-dimensional 
benchmarks including not only  Coriolis terms, as done in other works, but more general  mass and momentum sources, 
including a physical test case
well known in the meteorology community. All the results confirm the theoretical predictions:
the new Finite Element schemes obtained via the multi-dimensional GF formulation 
are  nodally super-convergent  at steady state if Gauss-Lobatto points are used. Because of stationarity preservation and super-convergence, the new methods significantly outperform standard ones
when considering unsteady setups close to steady states.  

Future work will be dedicated to a generalization of  the present continuous stabilized Finite Element approach to non-linear systems of balance laws, as well as to formulations based on discontinuous approximations. A finite-volume version of the same idea is discussed in   \cite{bcrt25} showing similar  enhanced features.

\section*{Acknowledgements}

MR is a member of the team CARDAMOM, Inria research center at the University of Bordeaux.\\

\section*{Data availability}
The data that support the findings of this study are available from the corresponding author,   upon reasonable request.\\

\section*{Credits}

All authors contributed to the study conception and design.  Code development  
performed by Davide Torlo.  
The first draft of the manuscript was written by Mario Ricchiuto and all authors commented on previous versions of the manuscript. 
All authors read and approved the final manuscript.

\appendix

\section{Energy stability}
%
%
%
%

We study here the energy stability of the variational forms associated to the SUPG and  OSS method, including Coriolis, friction, and space dependent sources, including the SU and OSS stabilizations.
The results are a generalization of those already presented e.g. in \cite{CODINA1997373,CODINA20001579,michel2021spectral,michel2022spectral,brt25}.

\subsection{Energy stability: SU stabilization}\label{sec:app-stabSU} 

Consider a linear operator $\mathcal L \colon (C^1)^m\to (C^1)^m$, $$\mathcal L = \sum_{i=1}^d \mathcal A_i \frac{\del}{\del x_i} + B, \qquad \qquad A_0, \ldots, A_d, B \in \mathscr M^{m \times m}(\mathbb R) \text{ constant matrices with real entries}$$ and the PDE $\del_t \mathsf{Q} + \mathcal L \mathsf{Q} = 0$, $\mathsf{Q} \colon \mathbb R^+_0 \to (C^1)^m$. Consider a symmetric matrix $M$ such that
\begin{align}
 E = \frac12 \int\limits_{\Omega} \mathsf{Q}^t M \mathsf{Q} \, \dd \vec x =:\frac12 (\mathsf{Q}, \mathsf{Q})_M =: \frac12 \| \mathsf{Q} \|^2_M
\end{align}
is the relevant energy. Define the adjoint $\mathcal L^\dagger$ of $\mathcal L$ as $(Q_1, \mathcal L Q_2)_M = (\mathcal L^\dagger Q_1, Q_2)_M$ $\forall Q_1,Q_2$, i.e.
\begin{align}
 \mathcal L^\dagger= -\sum_{i=1}^d \mathcal A_i^t \frac{\del}{\del x_i}.
\end{align}
Energy preservation follows if $\mathcal L$ is skew-adjoint, i.e. if $\mathcal L = - \mathcal L^\dagger$. For linear acoustics, this is the case since
\begin{align}
 \mathcal L = \left( \begin{array}{cc} 0_{d \times d} & \mathrm{grad} \\ \mathrm{div} & 0 \end{array} \right)
\end{align}
and $M=\id_{(d+1) \times (d+1)}$.

Consider now the PDE $\partial_t\mathsf{Q} + \mathcal L \mathsf{Q} = \mathsf{S}$ with sources $\mathsf{S} \colon (C^1)^m \to (C^1)^m$.
Assume homogeneous or periodic boundary conditions.
The SUPG method, obtained as the variational form combining  \eqref{eq:centralGalerkin_bilinear} and \eqref{eq:SUPG_stab}, with $\mathsf w$ a test function and $\tau \in \mathbb R$ reads
\begin{align}
 \int\limits_{\Omega} (\mathsf{w} + \tau \mathcal L \mathsf{w})^t M( \partial_t \mathsf{Q} + \mathcal L \mathsf{Q} - \mathsf{S}) \, \dd \vec x &= 0 .
\end{align}
Using as test function $\mathsf{w}:=\mathsf{Q} + \tau \del_t \mathsf{Q} - \tau \mathsf{S}$ leads to
\begin{align}
 0 &= \Big(\mathsf{Q} + \tau (\del_t \mathsf{Q} + \mathcal L \mathsf{Q} - \mathsf{S}) + \tau^2 \mathcal L\del_t \mathsf{Q}-  \tau^2 \mathcal L \mathsf{S} ,  \partial_t \mathsf{Q} + \mathcal L \mathsf{Q} - \mathsf{S}\Big)_M\\
\begin{split}
&=\big( \mathsf{Q}, \partial_t \mathsf{Q} \big)_M + \cancel{\big( \mathsf{Q}, \mathcal L \mathsf{Q} \big)_M}-\big( \mathsf{Q}, \mathsf{S}\big)_M
 +\tau  \Big( \del_t \mathsf{Q} + \mathcal L \mathsf{Q} - \mathsf{S}, \partial_t \mathsf{Q} + \mathcal L \mathsf{Q}- \mathsf{S}\Big)_M
 \\
 &\qquad +\tau^2 \Big( \mathcal L \del_t \mathsf{Q},  \partial_t \mathsf{Q} +  \mathcal L \mathsf{Q} - \mathsf{S} \Big)_M - \tau^2 \Big( \mathsf{S}, \mathcal L^\dagger  ( \partial_t \mathsf{Q} + \mathcal L \mathsf{Q} - \mathsf{S}) \Big)_M 
\end{split} \\ 
\begin{split} &= \frac{\dd}{\dd t} \left(E + \frac12 \tau^2 (\mathcal L \mathsf{Q}, \mathcal L \mathsf{Q})_M \right) -\big( \mathsf{Q}, \mathsf{S}\big)_M  + \tau \| \del_t \mathsf{Q} + \mathcal L \mathsf{Q} - \mathsf{S}\|^2_M+ \tau^2 \cancel{\big( \mathcal L\del_t \mathsf{Q},  \del_t \mathsf{Q}\big )_M }
 \\
 &\qquad \cancel{-\tau^2 \big( \mathcal L \del_t \mathsf{Q},  \mathsf{S} \big)_M 
 \quad - \tau^2 \big( \mathsf{S}, \mathcal L^\dagger   \partial_t \mathsf{Q}  \big)_M }
 - \tau^2 \big( \mathcal L\mathsf{S},   \mathcal L \mathsf{Q}  \big)_M 
 + \tau^2 \cancel{\big( \mathsf{S}, \mathcal L^\dagger \mathsf{S}) \big)_M}, 
\end{split} \\ 
\end{align}
where all the terms  leading to integrals of the type $ \big( (\cdot), \mathcal L(\cdot) \big)_M$ lead to integrals of divergences and cancel out when assuming homogeneous or periodic
boundary conditions or, equivalently, because they are skew-adjoint. We are thus left with
%
%
$$
\frac{\dd}{\dd t} \left(E + \frac12 \tau^2 (\mathcal L \mathsf{Q}, \mathcal L \mathsf{Q})_M \right) = 
\left(\mathsf{Q} ,\mathsf{S}\right)_M +\tau^2 (\mathcal L\mathsf{S},\mathcal L\mathsf{Q})_M  -  \tau \| \del_t \mathsf{Q} + \mathcal L \mathsf{Q} - \mathsf{S}\|^2_M
 \leq 
\left(\mathsf{Q} ,\mathsf{S}\right)_M +\tau^2 (\mathcal L\mathsf{S},\mathcal L\mathsf{Q})_M.
$$

It is interesting to note that for the system under consideration, 
the SUPG energy is a triple norm associated to the following scalar product associated to the mixed space $H(\text{div})-H^1$:
\begin{align}
(\mathsf{Q}_1,\mathsf{Q}_2)_{\text{SUPG}} &= \mathsf{Q}_1^t\mathsf{Q}_2  +  \tau^2( \nabla\cdot(u_1,v_1)\nabla\cdot(u_2,v_2)  +  \nabla p_1 \cdot   \nabla p_2)\\
E_{\text{SUPG}} &= 
\dfrac{1}{2}(\mathsf{Q},\mathsf{Q})_{\text{SUPG}}.
\end{align}
The stability estimate   can be written in terms of this scalar product as 
$$
\frac{\dd}{\dd t}E_{\text{SUPG}} 
\le  (\mathsf{Q},\mathsf{S})_{\text{SUPG}} .
$$

\subsection{Energy stability: OSS  stabilization}\label{sec:app-stabOSS} 

For the OSS method we can proceed in a similar abstract manner. We write the stabilized variational form as:
find $\mathsf{Q}\in (V_h)^{d+1}$ such that for every $\mathsf{w} \in (V_h)^{d+1}$
\begin{equation}
	\begin{cases}
	\left(\mathsf{w},\partial_t \mathsf{Q} + \mathcal{L} \mathsf{Q}-\mathsf{S}\right) _M + \tau \left(\mathcal L \mathsf{w}, \mathcal{L}\mathsf{Q}-\mathsf{S}  - \Pi \right)_M=0,\\
	\left(\mathsf{w},\Pi - (\mathcal{L}\mathsf{Q}-\mathsf{S}) \right)_M=0.
	\end{cases}	
\end{equation}
We substitute $\mathsf{w}=\mathsf{Q}$ in the first equation, and $\mathsf{w} = \Pi$ in the second and summing the two, we obtain
$$
\begin{aligned}
	0=&\left(\mathsf{Q},\partial_t \mathsf{Q} \right)_M + \cancel{\left(\mathsf{Q},\mathcal{L} \mathsf{Q} \right)_M}- \left(\mathsf{Q},\mathsf{S} \right)_M + \tau \left(\mathcal L \mathsf{Q}, \mathcal{L}\mathsf{Q}-\mathsf{S}  - \Pi \right)_M+\left(\Pi,\Pi - (\mathcal{L}\mathsf{Q}-\mathsf{S}) \right)_M\\
	=&\left(\mathsf{Q},\partial_t \mathsf{Q} \right)_M- \left(\mathsf{Q},\mathsf{S} \right)_M + \tau \left(\mathcal{L}\mathsf{Q}-\mathsf{S}  - \Pi,\mathcal{L}\mathsf{Q}-\mathsf{S}  - \Pi\right)_M +\tau\cancel{\left(\mathsf{S},\mathcal{L}\mathsf{Q}-\mathsf{S}  - \Pi\right)_M}.
	\end{aligned}
	$$
	This leads to the final estimates
	$$
 \frac{\dd}{\dd t} E = \left(\mathsf{Q},\mathsf{S} \right)_M -   \tau \|\mathcal{L}\mathsf{Q}-\mathsf{S}  - \Pi,\mathcal{L}\mathsf{Q}-\mathsf{S}  - \Pi\|^2_M \leq \left(\mathsf{Q},\mathsf{S} \right)_M.
$$


 \section{Proof of the consistency estimates for $S=S(x,y)$}\label{app:cons-proof1}

We prove the result of proposition \ref{th:consistency1} considering the case $S=S(x,y)$ for simplicity.  Let us first consider the  pressure equation.
By hypothesis the identity  $\partial_xu_e +\partial_y v_e=S_p(x,y)$ 
is true in every point. It is thus true at every collocation point, and thus   true for the interpolated  steady operator, namely  
   $$
   \partial_xu_e(x_\alpha,y_\beta)+\partial_yv_e(x_\alpha,y_\beta) =S_p(x_\alpha,y_\beta)\;\forall  (x_\alpha,y_\beta)  \Longrightarrow (\partial_xu_e)_h + (\partial_yv_e)_h =(S_p)(x,y)_h .
   $$
   Crucial here is the fact that all the functions ($\partial_xu_e,\,\partial_yv_e$ and $S_{p}$) are interpolated in the same functional space, so the interpolation of zero
   is consistent.
Consider now the array composed by the sampled values of each term of the point-wise exact  operators:
$$
[\mathrm{D}u_x^{e} ]_{\alpha,\beta} :=    \partial_xu_e(x_\alpha,y_\beta)\;,\quad
[\mathrm{D}v_y^{e} ]_{\alpha,\beta} :=    \partial_yv_e(x_\alpha,y_\beta)\;,\quad
[\mathrm{S}_p ]_{\alpha,\beta} :=   S_p(x_\alpha,y_\beta)  .
$$
Trivially we have $\mathrm{D}u_x^{e}+\mathrm{D}v_y^{e}-\mathrm{S}_p^{e}=0$.
We thus  start from the non-stabilized GF  weak  form  and add and remove appropriate terms to  errors in the projected fluxes.  In particular, we remove from the discrete 
right hand side  the array 
of  the sampled values of the operator  (which is still denoted by  $(\partial_xu_e)_h + (\partial_yv_e)_h -(S_p)_h $  with some abuse of notation to allow
matching terms in the proof): 
 \begin{equation*}
 	\begin{split}
 &(D_x \otimes D_y I_y  ) \mathrm{u}+  (D_x I_x \otimes D_y) \mathrm{v}  - (D_x I_x \otimes D_y I_y) \mathrm{S}_p  \\&\phantom{mmm}- (D_x\otimes D_y)(I_x \otimes I_y) (\underbrace{\mathrm{D}u_x^{e}+\mathrm{D}v_y^{e}-\mathrm{S}_p^{e}}_{\equiv 0})\\ 
&= (D_x  \otimes D_y I_y )( \mathrm{u}  -  (I_x\otimes \iden_y)\mathrm{D}u_x^{e} )   + (D_x I_x \otimes D_y)  ( \mathrm{v} -   (\iden_x\otimes I_y)( \partial_yv_e)_h )  \\   &\phantom{mmm} - \cancel{  (D_x I_x \otimes D_y I_y) \mathrm{S}_p +  (D_x I_x \otimes D_y I_y) \mathrm{S}_p}.
 \end{split}
\end{equation*}

For a fixed $y=y_{\beta}$, the term  $ \mathrm{u}  -  (I_x\otimes \iden_y)( \partial_xu_e)_h$ corresponds, up to a function of $y$ only, to the integration of 
$$
u'(x,y_{\beta})=\partial_xu_e(x, y_{\beta}).
$$
using the ODE solver associated to $I_x$. The same can be said for the term $ \mathrm{v} -   (\iden_x\otimes I_y)( \partial_yv_e)_h$. We can thus
construct discrete states  by marching along
grid-lines using the ODE solvers defined by $I_x$ and $I_y$ by 
\begin{equation}\label{eq:proj1}
\begin{aligned}
u'(x,y_{\beta})= &\partial_xu_e(x, y_{\beta})= -(  \partial_yv_e(x,y_\beta) -S_p(x,y_\beta)), \\
v'(x_\alpha,y)= &\partial_yv_e(x_\alpha, y)=- ( \partial_xu_e(x_\alpha,y)-S_p(x_\alpha,y)).
\end{aligned}
\end{equation}
By construction the resulting nodal values  $u_h(x_\alpha,y_\beta)$ and $v_h(x_\alpha,y_\beta)$ verify
\begin{equation}\label{eq:uproj1}
\mathrm{u}_h+   (I_x\otimes \iden_y)       ( (\partial_y v_e)_h  -(S_p)_h)= \textrm{c}_x(y)\;,\quad
v_h+ (\iden_x \otimes I_y)((\partial_xu_e)_h -(S_p)_h)= \textrm{c}_y(x) 
\end{equation}
and are in the kernel of the above discrete operator. Moreover, they satisfy by construction
the boundary estimate $(\tilde{ \vec{v}}^{\partial})_{\partial  \Omega}=\tilde{ \vec{v}}^{\partial}_e$ with perturbations
associated to the order $\mathcal{O}(h^M)$ of the ODE integrators, while on the  boundaries  $(x_0,y_\beta)$ and $(x_\alpha,y_0)$
we have exactly    $(\tilde{ \vec{v}}^{\partial})_{\partial  \Omega}=\vec{v}^{\partial}_e$  
  as initial conditions   in \eqref{eq:proj1}.
This can be shown by considering that for an arbitrary point $(x_\alpha,y_\beta)$ in the domain we have
\begin{equation*}
\begin{split}
u_h(x_\alpha,y_\beta) =  u_e(x_0,y_\beta) - \int_{x_0}^{x_\alpha}(\partial_y v_e(x,y_\beta) -(S_p)_h(x,y_\beta))\dd x + \mathcal{O}(h^{M})=u_e(x_\alpha, y_\beta)+ \mathcal{O}(h^{M}), \\
v_h(x_\alpha,y_\beta) =  v_e(x_\alpha,y_0) - \int_{y_0}^{y_\beta}(\partial_xu_e(x_\alpha, y)-(S_p)_h(x_\alpha,y))\dd y + \mathcal{O}(h^{M})=v_e(x_\alpha, y_\beta)+ \mathcal{O}(h^{M}).
\end{split}
\end{equation*}

For the pressure we can immediately remark that the discrete variational form  \eqref{eq:centralGalerkinGF_matrix_collected} allows to define
two separate discrete pressures that verify
\begin{equation}\label{eq:pproj1}\begin{aligned}
\hat p_h(x_\alpha,y_\beta) =  p_e(x_0,y_\beta) + \int_{x_0}^{x_\alpha} (S_u)_h(x,y_\beta)  \dd x  =p_e(x_\alpha, y_\beta)+ \mathcal{O}(h^{M}), \\
\bar p_h(x_\alpha,y_\beta) =  p_e(x_\alpha,y_0) + \int_{y_0}^{y_\beta} (S_v)_h(x_\alpha,y)\dd y  =p_e(x_\alpha, y_\beta)+ \mathcal{O}(h^{M}).
\end{aligned}
\end{equation}
By constructions these pressures   satisfy the global error estimate of order $\mathcal{O}(h^M)$, and $\hat p_h$ is in the kernel of the first equation of  \eqref{eq:centralGalerkinGF_matrix_collected} and $\bar p_h$ is in the kernel of the second equation of \eqref{eq:centralGalerkinGF_matrix_collected}, since they verify
$$
\begin{aligned}
\hat p_h - (I_x\otimes \iden_y) (S_u)_h= \textrm{c}_x(y)\;,\quad
\bar p_h - (\iden_x \otimes I_y)(  S_v)_h= \textrm{c}_y(x)\,.
\end{aligned}
$$
A unique  solution of the discrete system is identified up to a  compatible  $\mathcal{O}(h^M)$ perturbations of the  boundary data. 
A possibility is to set 
$$
 p(x_\alpha,y_\beta) = \bar p(x_0,y_\beta)  + \int_{x_0}^{x_\alpha} (S_u)_h(x,y_\beta)  \dd x =  p_e(x_0,y_0) + \int_{y_0}^{y_\beta} (S_v)_h(x_0,y)\dd y  +\int_{x_0}^{x_\alpha} (S_u)_h(x,y_\beta)  \dd x 
$$
or equivalently 
$$
 p(x_\alpha,y_\beta) = \hat p(x_\alpha,y_0) + \int_{y_0}^{y_\beta} (S_v)_h(x_\alpha,y)\dd y =p_e(x_0,y_0) + \int_{x_0}^{x_{\alpha}} (S_u)_h(x,y_0)  \dd x  + \int_{y_0}^{y_\beta} (S_v)_h(x_\alpha,y)\dd y 
$$
or any linear combination $\forall \lambda \in \mathbb{R}$:
\begin{equation}\label{eq:pressureBC}
\begin{aligned}
 p(x_\alpha,y_\beta)  =     p_e(x_0,y_0) +   \lambda & \left[  \int_{y_0}^{y_\beta} (S_v)_h(x_0,y)\dd y  + \int_{x_0}^{x_\alpha} (S_u)_h(x,y_\beta)  \dd x \right]\\
 +  (1-\lambda) &\left[  \int_{x_0}^{x_{\alpha}} (S_u)_h(x,y_0)  \dd x  + \int_{y_0}^{y_\beta} (S_v)_h(x_\alpha,y)\dd y  \right].
\end{aligned}
\end{equation}
To conclude, we note that  the  collocation method associated to the integration tables obtained from the Lagrange bases on Gauss-Lobatto points
is the well known LobattoIIIA method, which verifies the hypotheses of the proposition with $M=K+2$ when using  $K+1$ points (interpolation polynomial degree  $K$) and $K\ge 2$ \cite{hairer}.
For $K=1$ the classical $h^2$ consistency applies. This completes the proof.

\bibliographystyle{abbrv}
\bibliography{biblio}

\end{document}